\documentclass[a4paper,11pt]{article}
\usepackage{amsmath, amsthm, amsfonts, amssymb, bbm}
\usepackage{graphicx, psfrag}
\usepackage[usenames,dvipsnames]{xcolor}
\usepackage{float}

\usepackage{cite}
\usepackage{hyperref}
\usepackage{stackrel}
\usepackage{diagbox}
\usepackage{ulem}
\usepackage[utopia]{mathdesign}
\usepackage{mathtools}
\usepackage{mhchem}
\usepackage{enumitem}

\usepackage[color=purple]{todonotes}

\usepackage[margin=0.9in]{geometry}

\usepackage{dsfont}
\usepackage{soul}

\usepackage{tikz}
\usetikzlibrary{arrows}

\DeclareFontFamily{U}{BOONDOX-calo}{\skewchar\font=45 }
\DeclareFontShape{U}{BOONDOX-calo}{m}{n}{
  <-> s*[1.05] BOONDOX-r-calo}{}
\DeclareFontShape{U}{BOONDOX-calo}{b}{n}{
  <-> s*[1.05] BOONDOX-b-calo}{}
\DeclareMathAlphabet{\mathcalboondox}{U}{BOONDOX-calo}{m}{n}
\SetMathAlphabet{\mathcalboondox}{bold}{U}{BOONDOX-calo}{b}{n}
\DeclareMathAlphabet{\mathbcalboondox}{U}{BOONDOX-calo}{b}{n}
\newcommand{\mcb}[1]{{\mathcalboondox #1}}

\usepackage{xcolor}
\definecolor{brightcerulean}{rgb}{0.11, 0.67, 0.84}
\definecolor{cerulean}{rgb}{0.0, 0.48, 0.65}
\definecolor{Gray}{rgb}{0.5, 0.5, 0.5}
\definecolor{columbiablue}{rgb}{0.61, 0.87, 1.0}
\definecolor{OliveGreen}{rgb}{0, 0.26, 0.15}
\definecolor{dark red}{rgb}{0.7, 0., 0.2}
\definecolor{dark red 2}{rgb}{0.81, 0.09, 0.13}
\definecolor{dark blue}{rgb}{0, 0.18, 0.39}
\definecolor{dark blue 2}{rgb}{0.03, 0.27, 0.49}
\definecolor{dark green2}{rgb}{0.07, 0.53, 0.03}
\definecolor{dark green}{rgb}{0, 0.44, 0}
\definecolor{lightcarminepink}{rgb}{0.9, 0.4, 0.38}
\definecolor{lightcoral}{rgb}{0.94, 0.5, 0.5}
\definecolor{lightcornflowerblue}{rgb}{0.6, 0.81, 0.93}
\definecolor{lightcyan}{rgb}{0.88, 1.0, 1.0}
\definecolor{lavenderblush}{rgb}{1.0, 0.94, 0.96}
\definecolor{deeppeach}{rgb}{1.0, 0.8, 0.64}
\definecolor{darkchampagne}{rgb}{0.76, 0.7, 0.5}
\definecolor{desertsand}{rgb}{0.93, 0.79, 0.69}
\definecolor{classicrose}{rgb}{0.98, 0.8, 0.91}
\definecolor{myyeallow}{rgb}{0.98, 0.91, 0.71}
\definecolor{carolinablue}{rgb}{.1, .7, 1}
\definecolor{antiquewhite}{rgb}{0.98, 0.92, 0.84}
\definecolor{mycolor}{rgb}{0.98, 0.91, 0.71}
\definecolor{darkerblue}{rgb}{0.45, 0.7, 1.0}
\definecolor{limes}{rgb}{0.2, .9, 0.2}
\definecolor{softgreen}{rgb}{0.8, 1, 0.8}
\definecolor{softblue}{rgb}{.9, .9, 1}
\definecolor{camel}{rgb}{0.76, 0.6, 0.42}
\definecolor{brightcerulean}{rgb}{0.11, 0.67, 0.84}
\definecolor{cerulean}{rgb}{0.0, 0.48, 0.65}
\definecolor{Gray}{rgb}{0.5, 0.5, 0.5}
\definecolor{blizzardblue}{rgb}{0.67, 0.9, 0.93}
\definecolor{cyan}{rgb}{0.0, 1.0, 1.0}
\definecolor{mauve}{rgb}{0.86, 0.82, 1.0}
\definecolor{paleblue}{rgb}{0.2, 0.85, 0.97}
\definecolor{bronze}{rgb}{0.8, 0.5, 0.2}
\definecolor{nogreen}{rgb}{.9, 0.1, 0.3}
\definecolor{newblue}{rgb}{.1,.4,1}

\newcommand{\E}{\mathbbm{E}}
\newcommand{\Id}{\mathbbm{1}}

\newcommand{\R}{\mathbb{R}}
\newcommand{\N}{\mathbb{N}}
\newcommand{\Z}{\mathbb{Z}}
\newcommand{\T}{\mathbb{T}}

\newcommand{\dd}{\mathrm{d}}
\newcommand{\eps}{\varepsilon}

\newcommand{\cF}{\mcb F}

\def\one{\mathrm{(i)}}
\def\two{\mathrm{(ii)}}
\def\three{\mathrm{(iii)}}

\newcommand{\ola}[1]{\overleftarrow{#1}}
\newcommand{\ora}[1]{\overrightarrow{#1}}

\DeclareMathOperator{\sgn}{sgn}

\DeclareMathAlphabet{\mathpzc}{OT1}{pzc}{m}{it}

\newtheorem{prop}{Proposition}[section]
\newtheorem{thm}[prop]{Theorem}
\newtheorem{lem}[prop]{Lemma}
\newtheorem{defin}[prop]{Definition}

\newtheorem{cla}[prop]{Claim}

\newtheorem{remark}[prop]{Remark}

\newenvironment{rem}{\begin{remark}\normalfont}{\end{remark}}

\numberwithin{equation}{section}
\definecolor{verdescuro}{RGB}{0,120,0}

\allowdisplaybreaks

\title{From ABC to KPZ}
\author{G. Cannizzaro\thanks{University of Warwick, Department of Statistics, Zeeman Building, CV4 7AL, Coventry. E-mail: {\tt giuseppe.cannizzaro@warwick.ac.uk}} \and P. Gon\c calves\thanks{Instituto Superior T\'ecnico, Department of Mathematics, Av. Rovisco Pais 1, 1049-001, Lisbon. E-mail: {\tt pgoncalves@tecnico.ulisboa.pt}} \and R. Misturini\thanks{Universidade Federal do Rio Grande do Sul, Instituto de Matem\'atica e Estat\'istica. Av. Bento Gon\c calves, 9500. CEP 91509-900, Porto Alegre, Brazil. E-mail: {\tt ricardo.misturini@ufrgs.br}}
\and A. Occelli\thanks{LAREMA, Universit\'e d'Angers, 2 Bd de Lavoisier, 49045 Angers, France. E-mail: {\tt alessandra.occelli@univ-angers.fr}}}

\begin{document}

\maketitle

\begin{abstract}
We study the equilibrium fluctuations of an interacting particle system evolving on the discrete ring 
{with $N\in\mathbb N$ points, denoted by $\mathbb T_N$}, 
and with three species of particles that we name $A,B$ and $C$, but such that at each site there is only one particle. 
We prove that proper choices of  density fluctuation fields (that match those from nonlinear fluctuating 
hydrodynamics theory) associated to the {(two)} conserved quantities converge, 
in the limit $N\to\infty$, to a system of  stochastic partial differential equations, that can either be 
the Ornstein-Uhlenbeck equation or the {Stochastic Burgers equation.} 
{To understand the cross interaction between the two conserved quantities, we derive a general 
version of the Riemann-Lebesgue lemma which is of independent interest. }
\end{abstract}


\section{Introduction}

One of the major open problems in statistical mechanics is the characterisation of the macroscopic evolution equations from the large-scale description of the conserved quantities in Newtonian particle systems. By replacing a deterministic dynamics with a stochastic one, many mathematical techniques have been developed in the last forty years and very interesting results have been obtained. The first class of results is related to the well-known hydrodynamic limit, which consists of deriving the space-time evolution equations for the conserved quantities of a system from the underlying random evolution of its microscopic {counterpart}. The second  is related to the description of the fluctuations of the random microscopic system around its typical profile.  In the former case, the limit is deterministic and given in terms of a solution to a PDE, the hydrodynamic equation; while  in the  latter, the limit is random and given in terms of a solution to a stochastic PDE.
The focus of this article {is on the second problem} and more specifically our goal is to derive the fluctuations 
for a model that has {\it more than one} conservation law. To illustrate our results we first give a brief overview on  
what happens for a simplified version of our model which has only one conservation law. 

\subsection{The exclusion process}
{One of the most studied interacting particle systems (and perhaps the most classical one)}  is the exclusion process. Let us consider its evolution on the discrete one dimensional torus with $N$ sites, that we denote by $\mathbb T_N$. At each site of $\mathbb T_N$ there can be at most one particle and, after an exponential clock of rate $1$, particles at the bond $\{x,y\}$ exchange their positions {independently} at a rate given by a transition probability  $p(\cdot)$. We restrict ourselves to the case in which particles only move to {nearest neighbour sites} so that if $z$ is the size of the jump, then  $p(z)=0$ if $|z|>1.$ 
When $p(\cdot)$ is symmetric i.e. $p(1)=p(-1)=1/2$, the system is the well known symmetric simple exclusion process (SSEP); when $p(1)=1-p(-1)=p\neq 1/2 $ the system is the  {so-called} asymmetric simple  exclusion process (ASEP) and when $p(1)-p(-1)=E/N^\gamma$, $\gamma\in[0,\infty]$, the system is the weakly asymmetric simple exclusion process (WASEP), which interpolates {between} the SSEP (for $\gamma=\infty$) and the ASEP (for $\gamma=0$ and $E=2p-1$). Since particles only swap positions on $\T_N$, the conserved quantity is the number of particles, i.e. $\sum_{x\in\mathbb T_N}\eta_t(x)$ where the variable $\eta_t(x)$ denotes the number of particles at site $x$ and time $t$. From the exclusion  rule $\eta_t(x)\in\{0,1\}$ and if $\eta_t(x)=1$, we say that the site $x$ is occupied, and empty otherwise.  
The invariant measures are the Bernoulli product measures of parameter $\varrho\in(0,1)$, that indicates the density of particles.

We consider the system speeded up in the time scale $tN^a$, where $a>0$ and we define the empirical measure associated to the unique conserved quantity, i.e. the number of particles, as       
                                    $$\pi_t^{N}(du): =\frac 1N \sum _{x\in \mathbb T _N} \eta_{tN^a}(x)\, \delta_{\frac xN}\left( du\right),$$
                                    {where for $\tilde u\in\mathbb T$ the notation $\delta_{\tilde u}(du)$ is for the Dirac measure at $\tilde u$.}                                     
The hydrodynamic equation for $\pi^{N}$ depends on the asymmetry, which in turn determines the relevant time-scale: 
for SSEP the time scale is diffusive, $tN^2$, and {the hydrodynamic equation}  is the heat equation given by $\partial_t\varrho(t,u)=\tfrac12\Delta \varrho(t,u)$; for ASEP the time scale is hyperbolic, $tN$, and the equation is the inviscid Burgers equation $\partial_t\varrho(t,u)= E\nabla F(\varrho(t,u))$, for $F(\varrho)=\varrho(1-\varrho)$; for WASEP with $\gamma=1$, 
the time scale is again diffusive, $tN^2$, and the equation is the viscous Burgers' equation $$\partial_t \varrho(t,u)=1/2\Delta \varrho(t,u)+E\nabla F(\varrho(t,u))\,,$$ where $F$ is the same as before, see \cite{DPS} for a proof on the weakly asymmetric case and references therein for the other cases.

Now we {turn to the description of} the fluctuations around the hydrodynamic limit. The study of non-equilibrium fluctuations is usually very intricate as it requires a good knowledge of the decay of correlations of the system, which is generally difficult to have. From now on, we assume that the system starts from the Bernoulli product measure of parameter $\varrho$. 
We define the fluctuation field associated to the density 
as the linear functional defined on a smooth test function $f$ as
 $${\mathcal{Y}^{N}_t({f})=\frac{1}{\sqrt N} \sum_{x\in\mathbb{T}_N}f(\tfrac xN)(\eta_{tN^{a}}(x)-\varrho) }.$$
Once again, the strength of the asymmetry plays a crucial role in the type of limit $\mathcal Y$ 
we can get for $\mathcal Y^N$: for SSEP in the diffusive time scaling $tN^2$, 
it is given by the solution to a Ornstein--Uhlenbeck equation; 
for WASEP with weak asymmetry, i.e. $\gamma>1/2 $, and under diffusive scaling, 
by the solution to a Ornstein--Uhlenbeck equation, 
while for $\gamma=1/2$ by the energy solution to the stochastic Burgers {(SB)} equation 
(see Definition \ref{def:energy_solution}); and finally for the ASEP in hyperbolic scaling $tN$, 
the fluctuations are linearly transported in time with a velocity given by $(1-2\varrho)(1-2p)$. 
Note that if $\varrho=1/2$, then the evolution for the limit field in ASEP in the hyperbolic time scale is trivial 
and the same holds even for $\varrho\neq 1/2$, as can be {seen} by redefining the field in a time moving frame with the velocity $(1-2\varrho)(p-q)N^{a-1}t$. 
To get non-trivial fluctuations we have to speed up time and choose $a=3/2$. Upon doing so 
the limiting field is the so-called KPZ fixed point (KPZ-fp)~\cite{QS}, which has been constructed in \cite{MQR}. In \cite{G} it was proved that up to the time scale $tN^\delta$ with $\delta<1/3$  there is no evolution of the field.  In fact, the results of \cite{G} applied to WASEP show that below the line $a=\tfrac43(\gamma+1)$ the {evolution is trivial}, but in fact, {this behaviour} goes up to the line $a=\gamma +3/2$ and in all this line the limit is the KPZ-fp, see \cite{QS}.

Summarising, for simple exclusion processes by  changing the strength of the asymmetry, the system has either diffusive behaviour (Gaussian fluctuations, for $\gamma>1/2$), or KPZ behaviour (when $\gamma\leq 1/2$) and in between, for $\gamma=1/2$, there is the {SB equation}. 
In fact, these large-scale statistics can be obtained from a variety of different {microscopic} models and, {therefore, they are universal, in the sense that the limit does not depend on the details of the underlying microscopic model}. 

\subsection{Universality for one component systems}

 To characterise the universality classes for one-component systems, 
 let $h=h(t,u)$ be the stochastic process encoding the quantity of interest for our model (usually called height function), 
 and $r, z$ be two constants. We consider the space-time renormalization group operator with exponents $1:1/r:z/r$ given by
 $$\mathfrak R_\lambda h(t,u)=\lambda^{-1}h(t\lambda^{z/r},u\lambda^{-1/r}).$$ 
 A universality class consists of the basin of attraction of the limit  $H:=\lim_{\lambda\to0}\mathfrak R_\lambda h$ under the operation of rescaling $\mathfrak R_\lambda$ defined above. 
For one component systems, as the exclusion processes described above, several universality classes arise: the Edwards-Wilkinson (EW) class whose exponents are 1:2:4 and the super-diffusive KPZ class, whose scaling exponents are 1:2:3.  The connection between these two universality classes, namely the EW and the KPZ, is the KPZ equation or the SBE. 
In {the} case of the simple exclusion process for an asymmetry of order $O(N^{-\gamma})$, the crossover goes from diffusive behaviour, i.e. the EW class (corresponding to the phase where symmetry dominates, $\gamma>1/2$), to KPZ for 
{$\gamma<1/2$} and time scale $a=\frac32+\gamma$, and the transition goes through the  SB (or the KPZ) equation, corresponding to the phase where both symmetry and asymmetry have the same impact $\gamma=1/2$, see \cite{QS,Vi}.

We highlight that the list of universality classes is not exhausted by those described above. In \cite{CH}, it was introduced 
a temperature-dependent model which reduces to the classical ballistic deposition model at zero temperature (and thus conjectured to display KPZ-type fluctuations), 
but whose {infinite}-temperature version is a random interface whose large-scale statistics are neither EW nor KPZ. 
Its scaling limit is the Brownian Castle (BC), a renormalization fixed point, whose scaling exponents are 1:1:2. 
The {BC} is itself conjectured to be universal, in that any interface model which displays 
both horizontal and vertical fluctuations but no smoothing should belong to its basin of attraction. 
Moreover, while for SEP it is the asymmetry whose tuning determines the crossover from EW to KPZ 
and under a suitable scaling leads to the SB equation, for the connection between BC and EW a similar role is played by the smoothing. 
As shown in~\cite{CHS}, there is an uncountable family of (different) processes, the so-called $\nu$-Brownian Castles, 
for $\nu$ a probability measure on $[0,1]$, which interpolate BC and EW and therefore represent the analogue
of the SB equation in this context.

\subsection{Nonlinear Fluctuation Hydrodynamic Theory }

Universality classes are identified by exponents and scaling functions that characterise the macroscopic behaviour of the fluctuations of the thermodynamical quantities of interest in a microscopic system.  To see what universality classes might pop up and how the aforementioned exponents and scaling functions arise, we outline the approach of  Nonlinear Fluctuation Hydrodynamics Theory (NLFH), which precisely describes the fluctuations of the conserved quantities of multi-component systems in terms of stochastic PDEs. The starting point of NLFH is the hydrodynamic scenario in a strong asymmetric regime (to make an analogy to WASEP, this corresponds to $\gamma=0$ or simply bear in mind ASEP). 

Suppose that our microscopic system has a family of invariant measures parametrised by a quantity $\varrho$ and 
let us denote by {$\mu_\varrho$} the corresponding measure, 
with $\langle \cdot \rangle_{\mu_\varrho}$ the average with respect to it. 
Assume that  the hydrodynamic equation for the thermodynamical quantity of interest is (in the hyperbolic scaling $tN$) given by $$\partial_t \varrho(t,u)+\partial_x j(t,u)=0,$$ where  
$j(t,u)=\langle j \rangle_{\mu_{\varrho(t,u)}}$ and $j$ is the instantaneous current of the system, {and from now on we will always denote the derivative with respect to the first variable (i.e. time) by $\partial_t$ and $\partial_x$ that with respect to 
the second (i.e. space)}. Writing the current in terms of the density $\varrho$, then we have $\partial_x j(\varrho(t,u))= j'(\varrho(t,u))\partial_x \varrho(t,u)$ so that the previous display reads
$$\partial_t \varrho(t,u)+j'(\varrho(t,u))\partial_x\varrho(t,u)=0.$$
Now we add a diffusion term $D\partial_x^2 \varrho(t,u)$ and a conservative noise $B\partial_x \xi $, where $\xi$ is a space-time white-noise.  
The idea now consists of expanding $\varrho(t,u)$ around its stationary value $\varrho$ as $\varrho(t,u)=\varrho+Y(t,u)$. Before proceeding, let us see what happens in ASEP. In this case the averaged current is given by $j(\varrho)=(p-q)\varrho(1-\varrho)$. For simplicity let us take $p=1$, but the same argument works for any $p\in(1/2,1]$. Then we obtain 
$$\partial_t Y(t,u)=D \partial_x^2 Y(t,u)-\partial_x (1-2\varrho) Y(t,u)+\partial_x (Y(t,u))^2-B \partial_x \xi_u(t),$$
which is nothing but a stochastic Burgers equation.  
Note that in the equation above, the first {term on the right-hand} side is the diffusive term, the second is a drift that can be removed via a Galileian transformation, i.e. by looking at the system in a time dependent moving frame, 
and the third, the quadratic term, comes from the strong asymmetric regime. 
Such a term cannot be removed and is the reason why different limit behaviours are observed 
in the symmetric, in which this term is absent, and asymmetric {cases}. 

For a multi-component system, we follow the same procedure and, 
since by the NLFH theory the universal behaviour is dictated by 
the quadratic term, we neglect the higher order contributions in the expansion of $\varrho(t,u)$. 
In other words, we expand the current-density relation up to second order and we get
$$\partial_t Y(t,u)=D\Delta Y(t,u)-j'{(\varrho)}\nabla  Y(t,u)+\frac{j''(\varrho)}{2}\nabla (Y(t,u))^2-B \nabla \xi_u(t). $$ 
As above, by changing variables $u\to u-j'(\varrho)$ (Galileian transformation) the drift term disappears 
and therefore we derive once more a stochastic Burgers equation.

In order to obtain a much richer diagram (than the one for the simple exclusion explained above) 
in which different universality classes arise and interplay non-trivially with each {other}, 
we now turn our attention to multi-component systems, which are the main focus of the present paper. 
More precisely, consider a system with $M$ conserved quantities and let $\vec{\varrho}$ be the vector whose $\alpha$-th entry denotes the $\alpha$-th quantity, {$\alpha\in\{1,2,\cdots, M\}$}. 
The hydrodynamic equation is then given by a system of conservation laws: $\partial_t \varrho_\alpha(t,u)+\nabla j_{\alpha}(t,u)=0$, where $\alpha\in\{1,2,\cdots, M\}$. 
Once again, we expand the density as $\varrho_\alpha(t,u)=\varrho_\alpha+Y_\alpha(t,u)$ 
and express the associated current $j_\alpha(\vec{\varrho})$ as a function of $\vec{\varrho}$. 
Adding a diffusion and a noise term, in the form of an $M$-dimensional space-time white noise $\vec{\xi}$, and neglecting higher order contributions, we reach
$$\partial_t \vec {Y}=-\nabla\Big\{J\vec{Y}+\frac 12\sum_{\alpha=1}^M\vec{Y}^T H^\alpha\vec{Y}+D\nabla\vec{Y}+B\vec{\xi}  \Big\},$$
where $J$ is the jacobian of the current matrix whose entries are given by $J^\alpha_\beta=\frac{\partial j_\alpha}{\partial \varrho_\beta}$ and  $H^\alpha_{\beta,\delta}$ are the Hessians of the current matrix whose entries are given by {$H^\alpha_{\beta,\delta}=\frac{\partial^2 j_\alpha}{\partial \varrho_\beta \partial \varrho_\delta}$}. 
At this point we transform the fluctuation fields $\vec{Y}$ into normal fields $\vec{\phi}$ via $\vec{\phi}=R\vec{Y}$. 
The matrix $R$ is chosen to diagonalise the Jacobian $J$, i.e. $RJR^{-1}={\rm diag}(v_\alpha)$, 
where $\{v_\alpha, {\alpha\in\{1,2,\cdots, M\}}\}$ 
are the eigenvalues of $J$. The evolution of $\vec{\phi}$ is then given by
\begin{equation}\label{eq:spde_normal_field}
\partial_t \phi_\alpha(t,u)=-\nabla\Big\{v_\alpha\phi_\alpha+\vec{\phi}^TG^\alpha \vec{\phi} +(\tilde D\nabla\vec{\phi})_\alpha+(\tilde B\vec{\xi})_\alpha  \Big\}
\end{equation}
where $\tilde D= RDR^{-1}$, $\tilde B=RB$ and $G^\alpha$ are the coupling matrices given by
$$G^\alpha=\frac{1}{2}\sum_{\beta=1}^M R_{\alpha,\beta}(R^{-1})^T H^\beta R^{-1}.
$$
Note that the SPDE in~\eqref{eq:spde_normal_field} has a drift term given by $v_\alpha$, which suggests that the normal field $\phi_\alpha$ should be taken in a moving frame with velocity $v_\alpha$. 
\medskip

In order to derive the universal large-scale behaviour, the NLFH looks at the structure function of $\vec{\phi}$ 
which is defined as 
\begin{equation}\label{def:Structure}
{S_{\alpha,\beta}(t,u)=\langle \phi_\alpha(t,u)\phi_\beta(0,0)\rangle_{\mu_\varrho}\,.}
\end{equation}
Now, under the strict hyperbolicity condition which requires all velocities $v_\alpha$ to be different, 
the off-diagonal components ($\alpha\neq\beta$) of the structure function are expected to decay very fast 
and thus should not contribute to the asymptotic limit.
The behaviour of the diagonal elements, i.e. $S_\alpha(t,u)=S_{\alpha,\alpha}(t,u)$, should instead be given 
by
\[
S_\alpha(t,u)\sim (C_\alpha t)^{-1/{z_\alpha}}f_\alpha \Big((C_\alpha t)^{-1/{z_\alpha}}(u-v_\alpha t)\Big)\,,
\]
where $z_\alpha$ is a dynamical exponent and $f_\alpha$ is a scaling function. 
To deduce the value of $z_\alpha$ and 
the form of $f_\alpha$, the effect of the nonlinearity and the noise must be suitably balanced and this can be done 
via a memory kernel, that we now discuss. Since the normal 
fields solve~\eqref{eq:spde_normal_field}, we see that $S_\alpha:=S_{\alpha,\alpha}$ in~\eqref{def:Structure} satisfies
\begin{equation*}
\partial_t S_\alpha(t,u)=-v_\alpha \partial_x S_\alpha(t,u)+\tilde D_{\alpha,\alpha}\partial_x^2 S_\alpha(t,u)+\int_0^t\int_\mathbb R
dv S_{\alpha}(t-s,u-v) \partial _x^2M_{\alpha,\alpha}(s,v)
\end{equation*}
 where $M_{\alpha,\alpha}$ is the memory kernel and is given by 
\begin{equation}\label{e:memoryK}
M_{\alpha,\alpha}(s,v)=2\sum_{\beta,\delta}(G_{\beta,\gamma}^\alpha)^2 S_\beta(s,v)S_{\delta}(s,v)\,.
\end{equation}
Once again, under the strict hyperbolicity condition, the {off-diagonal} terms should not contribute 
so that the dynamical exponent $z_\alpha$ and the scaling function $f_\alpha$, and consequently the 
corresponding universality class,  
should be determined by the diagonal ones. Neglecting the {off-diagonals}, the 
memory kernel reduces to 
$M_{\alpha,\alpha}(s,v)=2\sum_{\beta}(G_{\beta,\gamma}^\alpha)^2 (S_\beta({s,v}))^2$.

Let us now describe the possible limits. To this end, we introduce the set $\mathbb I_{\alpha}$ of those indices such that
$G^\alpha_{\beta,\beta}$ in~\eqref{e:memoryK} is not zero, i.e. 
$\mathbb I_{\alpha}:=\{\beta : G^\alpha_{\beta,\beta}\neq 0\}$. 
For two component systems (which is the case for the model studied in this article) one can get 
\begin{itemize}
\item diffusive behaviour if $\mathbb I_{\alpha}=\emptyset$, i.e. the EW universality class for the field $\alpha$ corresponding to $z_\alpha=2$ and $f_\alpha$ the usual heat kernel,
\item KPZ behaviour if $\alpha\in \mathbb I_{\alpha}$, corresponding to $z_\alpha=3/2${,}
\end{itemize}
which are the same obtained for systems with only one conservation law (see the discussion above for SEP). 
But now if for the normal field $\alpha$ the self-coupling term vanishes i.e. $G_{\alpha,\alpha}^\alpha=0$, but for $\beta$ we have $\beta\in I_\alpha$, i.e. $G^\alpha_{\beta,\beta}\neq 0$, then the dynamical exponent is equal to $z_\alpha=\min_{\beta\in  I_\alpha}\{1+\frac {1}{z_\beta}\}$ and the scaling function is a $z_\alpha$-stable distribution given in Fourier space by 
\[
\hat{S}(t,k)=\frac {1}{\sqrt {2\pi}}\exp\Big\{-iv_\alpha tk-C_\alpha t|k|^{z_\alpha}\times (1-i A_\alpha \sgn(k)\tan (\tfrac{\pi z_\alpha}{2}))\Big\},
\] 
where $\hat{S}$ denotes the Fourier transform given by $\hat{S}(t,k)=\frac {1}{\sqrt {2\pi}}\int_\mathbb R S(t,u)e^{-iku}du$, $C_\alpha$ is a constant and $A_\alpha\in[-1,1]$ an asymmetry. 
Now, since $z_\alpha$ satisfies $z_\alpha=\min_{\beta\in  I_\alpha}\{1+\frac {1}{z_\beta}\}$, 
we necessarily have that for all $n$, $z_\alpha=F_{n+3}/F_{n+2}$
where $F_n$ is the Fibonacci number defined by the recursion $F_{n+2}=F_{n+1}+F_n$ and $F_1=F_2=1$. 
As a consequence, {if} in the system there is {not} a normal field with dynamical exponent $z_\alpha=2$ or $3/2$, 
then the only possibility is to have $z_\alpha=\frac{1+\sqrt 5}{2}$, the Golden number, for all the fields. The reason why only this type of $z_\alpha$-L\'evy stable distributions appear with $z_\alpha$ given as above, is still  mysterious. 
The predictions for systems with two conservation laws was first derived in \cite{SS}, 
while for $n$ conservation laws in \cite{PSSS}.

Let us stress that the assumption that the current-density relation is strictly-hyperbolic (i.e. the velocities of the normal fields are all different) is crucial because it {\it formally} allows to neglect the crossed terms 
$S_\alpha(s,v)S_\beta(s,v)$, $\beta\neq\alpha$, 
and focus only on the diagonals. 
One of the goals of this article is to give for the first time a rigorous mathematical proof of 
the fact that non-diagonal terms of coupling matrices are indeed negligible. 
To do so, we consider a multi-component model with two conservation laws that we now describe.

\subsection{The particle exchange model}\label{sec:pem}
We study a generalization of the simple exclusion process by allowing three types of particles, that we name $A,B$ and $C$. In the $ABC$ model each site of the one-dimensional torus $\mathbb{T}_N=\Z/N\Z$ is occupied by one and only one particle, which can be of type $A$, $B$ or $C$. In the \textit{classical} $ABC$ model, introduced by Evans et. al in \cite{ekkm1-1998,ekkm2-1998}, particles exchange their positions with nearest neighbours particles with the asymmetric rates: $AB\to BA$, $BC\to CB$, $CA\to AC$ with rate $q=q(N)<1$ and $BA\to AB$, $CB\to BC$, $AC\to CA$ with rate $1$. For this dynamics on the torus the invariant measure is explicitly known only in the case that the number of particles of each species are equal, in which case it is given by a Gibbs measure of a certain Hamiltonian having long range pair interactions. In this context, in the weakly asymmetric regime {$q=1-O\left(\frac{\beta}{N}\right)$} introduced by  Clincy et al. in \cite{cde2003}, we recently obtained in~\cite{GMO2023} the system of hydrodynamic equations (with boundary conditions) that describes the evolution of the density field of each species for the dynamics in a open interval connected with reservoirs (in which case the invariant measure is not explicitly know).

In the present work we consider an $ABC$ model with different rates. The $ABC$ model is a particular case of the general $n$-component particle exchange model that is presented in \cite{Schuetz17KPZ}  and, as argued therein, can also be seen as a fluctuating directed polymer in $d\geq 2$. 
We now assume that the interaction rates depend on three constants $E_A, E_B$ and $E_C$: for $(\alpha, \beta)\in\{A,B,C\}$ the transposition $(\alpha,\beta)\to(\beta,\alpha)$ occurs at rate $1+\frac{E_\alpha-E_\beta}{2N^{\gamma}}$, as it is summarised in Figure~\ref{fig:model} below.

\begin{figure}[h]
	\centering
	\begin{tikzpicture}
		\draw[line width=.35mm] (0,0) -- (11,0);
		\draw[line width=.35mm] (0,-.1) -- (0,0.1);
		\draw[line width=.35mm] (1,-.1) -- (1,0.1);
		\draw[line width=.35mm] (2,-.1) -- (2,0.1);
		\draw[line width=.35mm] (3,-.1) -- (3,0.1);
		\draw[line width=.35mm] (4,-.1) -- (4,0.1);
		\draw[line width=.35mm] (5,-.1) -- (5,0.1);
		\draw[line width=.35mm] (6,-.1) -- (6,0.1);
		\draw[line width=.35mm] (7,-.1) -- (7,0.1);
		\draw[line width=.35mm] (8,-.1) -- (8,0.1);
		\draw[line width=.35mm] (9,-.1) -- (9,0.1);
		\draw[line width=.35mm] (10,-.1) -- (10,0.1);
		\draw[line width=.35mm] (11,-.1) -- (11,0.1);
		\shade [ball color=lightcyan] (0,0.3) circle [radius=.25cm];
		\node[font=\large] at (0,0.3){ C};
		\shade [ball color=magenta] (11,0.3) circle [radius=.25cm];
		\node[font=\large] at (11,0.3){ A};
		\draw[<->,line width=.35mm] (10,0.6) to [out=60,in=110] (11,0.6);
		\node[font=\large] at (10.5,1.2){ $1+\tfrac{E_B-E_A}{2N^\gamma}$};
		\draw[<->,line width=.35mm] (0,0.6) to [out=60,in=110] (1,0.6);
		\node[font=\large] at (0.5,1.2){ $1+\tfrac{E_C-E_A}{2N^\gamma}$};
		\shade [ball color=magenta] (1,0.3) circle [radius=.25cm];
		\node[font=\large] at (1,0.3){ A};
		\shade [ball color=magenta] (2,0.3) circle [radius=.25cm];
		\shade [ball color=carolinablue] (3,0.3) circle [radius=.25cm];
		\draw[<->,line width=.35mm] (4,0.6) to [out=60,in=110] (5,0.6);
		\node[font=\large] at (4.5,1.2){ $1+\tfrac{E_B-E_C}{2N^\gamma}$};
		\shade [ball color=carolinablue] (4,0.3) circle [radius=.25cm];
		\shade [ball color=lightcyan] (5,0.3) circle [radius=.25cm];
		\shade [ball color=lightcyan] (6,0.3) circle [radius=.25cm];
		\shade [ball color=carolinablue] (7,0.3) circle [radius=.25cm];
		\shade [ball color=magenta] (8,0.3) circle [radius=.25cm];
		\shade [ball color=lightcyan] (9,0.3) circle [radius=.25cm]; 
		\shade [ball color=carolinablue] (10,0.3) circle [radius=.25cm];
		\node[font=\large] at (2,0.3){ A};
		\node[font=\large] at (3,0.3){ B};
		\node[font=\large] at (4,0.3){ B};
		\node[font=\large] at (5,0.3){ C};
		\node[font=\large] at (6,0.3){ C};
		\node[font=\large] at (7,0.3){ B};
		\node[font=\large] at (8,0.3){ A};
		\node[font=\large] at (9,0.3){ C};
		\node[font=\large] at (10,0.3){ B};
		\draw[<->,line width=.35mm] (2,.6) to [out=60,in=110] (3,.6);
		\node[font=\large] at (2.5,1.2){ $1+\tfrac{E_A-E_B}{2N^\gamma}$};
		\draw[<->,line width=.35mm] (6,.6) to [out=60,in=110] (7,.6);
		\node[font=\large] at (6.5,1.2){ $1+\tfrac{E_C-E_B}{2N^\gamma}$};
		\draw[<->,line width=.35mm] (8,.6) to [out=60,in=110] (9,.6);
		\node[font=\large] at (8.5,1.2){ $1+\tfrac{E_A-E_C}{2N^\gamma}$};

	\end{tikzpicture}
	\caption{Dynamics of the model. }\label{fig:model}
\end{figure}

It turns out, as discussed in \cite{SRB96}, that for this model, in each irreducible class, the invariant measure is uniform over all possible configurations (see Lemma \ref{lemma_invariant}). As discussed in \cite{bertini} the hydrodynamic limit for the density of particles $A$ and $B$, in the diffusive time scaling $a=2$ and for $\gamma=1$, is given by the  system of equations~\eqref{eq:hydro}. In the present work we are interested in the fluctuations around the hydrodynamic limit.

{There are three important cases to distinguish depending on  the choice of  rates, 
apart from the case $E_A=E_B=E_C$, which is trivial as the  dynamics does not distinguish particles of type $A$,  $B$ or $C$ for which the limit of the fluctuations is Ornstein--Uhlenbeck. 
The case \textbf{(I)} is when $E_A=E_B$ so that $A$ and $B$ are exchanged at rate $1$. 
The case \textbf{(II)} is when $E_B=E_C$ 
and particles of type $B$ and $C$ are exchanged at rate $1$. 
Let us already mention here that the fact that two types of particles exchange at rate $1$ suggests that one normal field 
should have diffusive behaviour. And, in case \textbf{(II)} for example, if we look at the particles of type $A$, since the dynamics does not distinguish
between $B$ and $C$, we can conclude that the field of particles of type $A$ behaves as in WASEP. The case \textbf{(III)} is the most general and all exchange rates are weakly asymmetric, 
with an intensity regulated by the parameter $\gamma$.}

We note that the dynamics above conserves  two quantities: the total number of particles of type $A$ and the total number of particles of type $B$. Nevertheless any linear combination of these two quantities is again conserved. The invariant measures of this model are explicit: for any constant densities $\rho_A$, $\rho_B$ and $\rho_C=1-\rho_A-\rho_B$, the product measure $\nu_\rho$ given on $x\in\mathbb T_N$ and $\alpha\in\{A,B,C\}$ by
$\nu_\rho(\eta: \eta(x)=\alpha)=\rho_\alpha$,
is an invariant measure. Since the dynamics of each type of particle depends on the other, 
the evolution equations are not closed and it is a priori unclear how to derive the limiting equations.  
Nevertheless, according to NLFH, the normal fields  can be identified and predictions can be made 
(see Appendix \ref{sec:mode}, in which this is done in detail).  

Alternatively, to identify the normal fields we can (and will) proceed as follows. 
First, we analyse the action of the infinitesimal generator on the occupation variables for particles of types $A$ and $B$, 
derive their instantaneous current and centre all variables. 
This then will allow us to evaluate the generator on a generic field given by a linear combination of
the centred occupation variables for $A$ and $B$. 
The expression we obtain will display drift terms that blow up in the limit and we force these 
to be zero by passing the generic field to a time moving frame. 
At this point, the derivation of the correct normal fields boils downs to solving a system of two equations 
with two unknowns: the velocity of the moving frame and the constant defining the linear combination of fields (in principle 
the constants are two, but, by linearity, one can always be fixed to be $1$). 
This procedure is carried out in Section~\ref{sec:martn} and delivers the same 
normal fields as those predicted by the NLFH theory. 

Once blowing up terms are removed from the evolution equations,  we are left with higher order terms which, 
in this specific model, are quadratic and can be written as products of the occupation variables for 
particles of type $A$ or $B$. 
The coefficients in front of these quadratic terms are nothing but  the entries the coupling matrices. 

For the multi-species {WASEP} we consider, the special structure of these coupling matrices implies that 
for the normal fields, we always have $G^\alpha_{\beta,\beta}=G^\beta_{\alpha,\alpha}=0$ 
 whatever the choice of the constants $E_\alpha$ is (see the computations in Appendix \ref{sec:mode}). 
This means that the only important contribution comes from self-coupling terms 
(i.e. the entries  $G^\alpha_{\alpha,\alpha}$ and $ G^\beta_{\beta,\beta}$).
 Therefore  the predicted  limit behaviour from NLFH  (in the strong asymmetric regime, i.e. $\gamma=0$) is either diffusive (the EW universality class)  or KPZ behaviour (in the KPZ universality class), depending on whether 
 $G^\alpha_{\alpha,\alpha}$ or $G^\beta_{\beta,\beta}$ is zero or not. 
 See Figure~\ref{fig:MCM} for a summary of the predictions.

\begin{figure}
	\begin{center}	
		\begin{tabular}{ |c | c | c | }
			\hline
			\diagbox{$G^2$}{$G^1$} & $\left(\begin{smallmatrix}\star & \cdots \\ \cdots & 0\end{smallmatrix}\right)$ & $\left(\begin{smallmatrix}0 & \cdots \\ \cdots & 0\end{smallmatrix}\right)$\\
			\hline 
			$^{\textcolor{white}{A}}_{\textcolor{white}{A}}\left(\begin{smallmatrix}0 & \cdots \\ \cdots & \star\end{smallmatrix}\right)^{\textcolor{white}{A}}_{\textcolor{white}{A}}$ & (KPZ,KPZ) & (EW,KPZ) \\
			\hline
			$^{\textcolor{white}{A}}_{\textcolor{white}{A}}\left(\begin{smallmatrix}0 & \cdots \\ \cdots & 0\end{smallmatrix}\right)^{\textcolor{white}{A}}_{\textcolor{white}{A}}$ & (KPZ,EW) & (EW,EW) \\
			\hline
		\end{tabular}
		\caption{Classification of the universal behaviour of the two modes  observed in the ABC model by the structure of the mode coupling matrices. A star denotes a non-zero entry, a dot represents an arbitrary value.}\label{fig:MCM}
	\end{center}
\end{figure}

 \subsection{Our contribution}
 
 The main contribution of this article is twofold. First, we rigorously {determine the large-scale behaviour
of the non-diagonal terms of the coupling matrices. 
In our context, this means that, the quadratic terms due to crossed products of normal modes associated to 
particles of type $A$ and $B$, {\it only interact via their initial total mass} and {\it do not produce any additional non-linear 
term in the limit}}. The result is shown in wide generality as its assumptions only 
impose mild conditions on the moments of the fields, so that 
in particular it holds irrespectively of the specific dynamics one considers 
and even allows for products of more than two fields, 
but we are restricted to the diffusive time scale. 
{Let us stress that, while on the one hand this puts on firm ground the assumption in NLFH 
that off-diagonal terms are negligible (at least in diffusive time scale), on the other we {\it do observe} a contribution 
which is though trivial as it only involves their initial total masses (which are conserved quantities of the system). 
This is something which was not predicted by mode coupling because it only looks at the fluctuations of the normal modes 
at a fixed time.  } 

Second, we show that, in the two component model we consider, if the asymmetry is weak, $\gamma>1/2$, 
then the limit of both fields is again diffusive, so the system falls in the EW universality class. 
More interestingly, we carry out a full analysis of the case $\gamma=1/2$, in which both the EW and the KPZ 
universality classes arise. 
To be precise, we prove that any normal field whose coupling matrices has only non-zero non-diagonal entries, 
has again diffusive behaviour. From this result we then show that in cases \textbf{(I)} and \textbf{(II)} 
one mode is diffusive while the other is KPZ. 
For the case \textbf{(III)}, we show instead that both fields have indeed KPZ behaviour {
(since we are restricted to the choice $\gamma=1/2$, when we say KPZ behaviour, 
we mean that the fluctuations are given in terms of the SB equation).}
In other words, we rigorously prove that, for $\gamma=1/2$, {\it all the predicted results  summarised in Figure~\ref{fig:MCM} hold true}. 
\medskip

Let us stress that since the development of NLFH theory very little advances have been made for multi-component systems (for results on the L\'evy limits with parameters  $3/2$ or $5/3$ see  \cite{BGJ1,BGJ2,BGJSS,BGJS1,BGJS2,Cane,GH,JKO,SSS}). 
For some models (as, for example, chain of oscillators {with} the Fermi--Pasta--Ulam potential), 
even the understanding of the correct linear combinations of fields has not been successfully addressed 
(apart from numerical simulations).
Up to our knowledge, prior to this work, only the case of diffusive behaviour, L\'evy-$3/2$ and L\'evy-$5/3$ 
have been derived rigorously. This is the first time in which diffusion/KPZ, as well as KPZ for both modes has 
been obtained. 
In \cite{BFS} systems of coupled KPZ equations appear but the strict hyperbolicity condition was lacking, 
so that, contrary to our case, the resulting equations are coupled. 
\medskip 

Let us now comment on the technical aspects of our work. 
To obtain the KPZ behaviour we extend the second order Boltzmann--Gibbs principle derived in \cite{GJ14, GJS17} 
to the setting of multi-component systems. The Boltzmann--Gibbs principle allows 
to replace products of occupation variables with uniform averages in big microscopic boxes, 
which then enables us to close the equations for each of the normal fields and 
ultimately identify the limit as energy solutions of the SB equation. 
In order to prove that the contribution from crossed terms is negligible, 
we need a further refinement of the aforementioned principle, according to which, 
instead of taking uniform averages, 
we suitably distribute the mass in such a way that, in the macro-limit, this can be well approximated by a smooth function. 
Once this is in place, the control over the crossed term boils down to rewrite it in Fourier and 
apply a version of Riemann--Lebesgue lemma for stochastic processes which satisfy suitable moment bounds. 
Here, the strict hyperbolicity condition is essential, as if it fails the crossed terms might survive in the limit 
(see e.g. the zero-range model studied in \cite{BFS} where this indeed happens).   
\medskip

To conclude this introduction, we mention future work and possible extensions of our result. 
The first important open problem is to go beyond the  diffusive time scale and to show that in fact the results presented in the table above can be obtained in the whole range of the parameter $\gamma>0$. Moreover, it would also be interesting to push forward our results when  the system is evolving in the infinite lattice or in the open interval but in the presence of a Glauber dynamics at the boundary.

Another possible direction is to consider particle exchange models in which more than one particle per site is allowed. 
We believe that in this case the scenario could be richer and possibly other universal behaviours might appear. 
At last, it would  be very interesting to apply our results to Hamiltonian models 
as the chain of oscillators and that in \cite{BS}. 
Our contributions might provide useful insights which could help to derive the normal fields 
and determine their large-scale behaviour.

\subsection{Outline }
In Section \ref{sec:results}
 we introduce the model and we state our main result.  In Section \ref{sec:martn} we present the computations for the martingales at the discrete level for all the cases defining the jump rates.  In Section \ref{sec:proof_main_theo} we give the proof of the limit of the fluctuation fields. To do so in Subsection  \ref{sec:tightness} we prove that the sequence of fluctuation fields is tight and in Subsection \ref{sec:char_limit} we characterize the limit fields either as solutions to Ornstein-Uhlenbeck equations or as energy solutions to the stochastic Burgers equation.  
In Section \ref{sec:cross} we prove that the contribution of crossed terms are negligible in the asymptotic limit and Section 
\ref{sec:BG} is devoted to the proof of second order Boltzmann Gibbs principles.  In the Appendix \ref{sec:mode} we present the computations of NLFH theory for our specific model and in Appendix \ref{a:RL}
 we present some auxiliary results.

\section{Statement of Results}\label{sec:results}
\subsection{The model}
We consider the discrete ring with $N$ sites, $\mathbb{T}_N=\Z/N\Z$; each site is occupied by exactly one particle, 
and such particle is of {species} $\alpha$ with $\alpha\in\{A,B,C\}$. 
The system evolves by nearest-neighbour exchanges of particles in the presence of an external field that interacts with particles of different species with different strength denoted by $E_A,E_B,E_C$.

The space of configurations is ${\Omega}_{N}=\{A,\,B,\,C\}^{\mathbb{T}_N}$; its elements are denoted with $\eta$; on each site $x\in\mathbb{T}_N$, $\eta(x)\in\{A,\,B,\,C\}$. We define the occupation numbers of the species $\alpha\in\{A,\,B,\,C\}$  as the function $\xi^{\alpha}:{\Omega}_{N}\rightarrow\{0,\,1\}^{\mathbb{T}_N}$ acting on the configurations in the following way
\[
\xi^{\alpha}_x(\eta):=\Id_{\{\alpha\}}(\eta(x))=\begin{cases}
	1 & \eta(x)=\alpha\\
	0 & \textrm{otherwise}.
\end{cases}
\]
Note that $\sum_{\alpha}\xi^{\alpha}_x(\eta)=1$
for all $x\in\mathbb{T}_N$ {and for all $\eta\in\Omega_N$}. We consider  a \textit{weakly asymmetric regime}: for a configuration $\eta\in \Omega_N$ and $x\in\mathbb T_N$ such that $(\eta_x,\eta _{x+1})=(\alpha,\beta)$  for $\alpha,\beta\in\{A,\,B,\,C\} $, $\alpha\neq \beta$, the exchange to $(\eta_x,\eta _{x+1})=(\beta,\alpha)$  occurs at rate  
\begin{equation*}
c^{\alpha,\beta}:=c^{\alpha,\beta}_{N,\gamma}=1+\frac{E_\alpha-E_\beta}{2N^\gamma}.
\end{equation*}
The total number of particles of each species, $N_\alpha(\eta)=\sum_{x\in \mathbb T_N}\xi^\alpha_x(\eta)$, $\alpha\in\{A,\,B,\,C\}$, is conserved and  $N_{A}+N_{B}+N_{C}=N$.

Given two sites $x,\,y\in\T_{N}$ and given a configuration
$\eta\in {\Omega}_{N}$, we define $\eta^{x,y}$ as
the configuration obtained by exchanging particles on sites $x$
{and} $y$, i.e.
\[
\eta^{x,y}(z)=\begin{cases}
\eta(y) & z=x\\
\eta(x) & z=y\\
\eta(z) & \textrm{otherwise}.
\end{cases}
\]
Thus we can define the process as a continuous time Markov chain with state space 
${\Omega}_{N}$ and generator $L_{N}$ acting on the functions 
$f\,:{\Omega}_{N}\rightarrow\mathbb{R}$ as
\begin{equation}\label{generator}
L_{N}f(\eta)=N^a\sum_{x\in\mathbb{T}_N}c_x(\eta)[f(\eta^{x,x+1})-f(\eta)],
\end{equation}
with
\begin{equation}\label{cxeta}
	c_x(\eta)=\sum_{\alpha,\beta}c^{\alpha,\beta}\xi^\alpha_x\xi^\beta_{x+1}
\end{equation}
where the sum runs over $\alpha\neq\beta\in\{A,B,C\}$. The dynamics is summarized in Figure \ref{fig:model}.

In \eqref{generator}, $a>0$, but everywhere in the article we will assume that $a=2$, i.e. the process will be speeded up in the diffusive time scale and this choice will be justified ahead (see e.g. Lemma~\ref{lemmaqv}).

Note that the rates satisfy the pairwise balance\footnote{Observe that this pairwise balance is not present in the classical $ABC$ that we discussed in the first paragraph of Section~\ref{sec:pem} .}. 
\begin{equation}\label{pairwise_balance}
	c^{AB}+c^{BC}+c^{CA}=c^{BA}+c^{CB}+c^{AC}.
\end{equation}
In~\cite{SRB96} a criterion to identify the invariant measure for generalised exclusion processes satisfying 
the above was showed.

\begin{lem}\label{lemma_invariant} Under condition \eqref{pairwise_balance}, any measure $\mu_N$ on $\Omega_N$ such that
	\begin{equation}\label{ivariancebytransp}
		\mu_N(\eta^{x,x+1})=\mu_N(\eta), \,\text{ for all }\,\eta\in \Omega_N \,\text{ and all }\,x\in \mathbb{T}_N,
	\end{equation} 
	is invariant for the dynamics with generator $L_N$.
\end{lem}
From this it follows that, when the number of particles of each species is fixed, $n_A+n_B+n_C=N$, the canonical distribution is uniform on $\{\eta\in\Omega_N: N_A(\eta)=n_A,N_B(\eta)=n_B,N_C(\eta)=n_c\}$ with
\begin{equation}\label{canonical}
	\mu_N(\eta)=\begin{pmatrix}
		N\\ n_A,\, n_B,\, n_C
	\end{pmatrix}^{-1}=\frac{n_A!\, n_B!\, n_C!}{N!}.
\end{equation}
Moreover, for any constant densities $\rho_A$, $\rho_B$ and $\rho_C=1-\rho_A-\rho_B$, the product measure $\nu_\rho$ with $\rho=(\rho_A, \rho_B, \rho_C)$, over $x\in \mathbb T_N$, such that for all $x\in\mathbb T_N$ and $\alpha\in\{A,B,C\}$,
$\nu_\rho(\eta: \eta(x)=\alpha)=\rho_\alpha$,
is an invariant measure on $\Omega_N$. 

In the diffusive time scaling $a=2$ and for $\gamma=1$, the hydrodynamic equations for the densities of particles $A$ and $B$ are given by \footnote{This result is stated in \cite{bertini} but it was not rigorously proved.  For a more complicated dynamics with stochastic reservoirs, see the results in \cite{GMO2023}.}
{\begin{equation}\label{eq:hydro}
\partial_t \begin{pmatrix}\rho^A\\ \rho^B \end{pmatrix}=\Delta \begin{pmatrix}\rho^A\\ \rho^B \end{pmatrix}-\nabla \left(\chi(\rho)\cdot g_E \right),
\end{equation}}
where \begin{equation*}
\chi(\rho)=\begin{pmatrix} \rho^A(1-\rho^A) & -\rho^A\rho^B\\ -\rho^A\rho^B & \rho^B(1-\rho^B) \end{pmatrix}
\end{equation*} is the mobility and $g_E=\begin{pmatrix}E_A-E_C\\E_B-E_C\end{pmatrix}$. The density of particles $C$ can be recovered by $\rho^C=1-\rho^A-\rho^B$.

\subsection{Fluctuation fields}\label{sec:flucfields}

Let us denote by $\mathbb T=[0,1)$ the continuous torus and by $\mcb D(\mathbb T)$ the space of $\R$-valued smooth functions on $\mathbb T$. 
We define the density fluctuation fields ${\mcb Y}_t^N=({\mcb Y}_t^{N,A},{\mcb Y}_t^{N,B},{\mcb Y}_t^{N,C})$ as
\begin{equation}\label{eq:densityfieldNL}
{\mcb Y}_t^{N,\alpha}(du)=\frac{1}{\sqrt N}\sum_{x\in\mathbb{T}_N}\left[\xi^\alpha_x(\eta_t)-\E_{\nu_\rho}[\xi^\alpha_x(\eta_t)]\right]\delta_{\frac{x}{N}}(du),
\end{equation}
where $\alpha\in\{A,B,C\}$ and $\E_{\nu_\rho}[\cdot]$ is the expectation with respect to the invariant product measure associated to constant profiles. 
To simplify the presentation, we assume  $\rho_A=\rho_B=\rho_C=\rho=1/3$ throughout the article 
(so that trivially for any $\alpha \in\{A,B,C\}$, $\E_{\nu_\rho}[\xi^\alpha_x(\eta_t)]=1/3$).
The results for general densities can be derived from ours by taking suitable Galilean transformations, 
which centre the fields with respect to appropriate reference moving frames.

To lighten the notation, for $s\in[0,T]$, we suppress the dependence on $\eta_s$ from $\xi^\alpha$ 
and write $\xi^\alpha_x(s):=\xi^\alpha_x(\eta_s)$.
For $v\in\mathbb R$ and $t\geq 0$, we denote by $T_{v t}$ the translation operator acting on {$f\in \mcb D(\mathbb T)$} as
\begin{equation}\label{eq:translation}
T_{v t}f\left(\frac{x}{N}\right)=f\left(\frac{x-v t}{N}\right).
\end{equation}

{The large scale fluctuations of the density fields depend on the relation among 
the coupling constants $E_A, E_B$ and $E_C$, and in particular on whether they are 
all different or not. 
This leads us to distinguish some cases that we now spell out.}

\quad

{\noindent\textbf{Case (I): $E_{A}=E_B$}}

\quad

\noindent In this first case, the fields that should be considered are
\begin{equation}\label{e:case1}
\begin{aligned}
\mcb Z^{N,+}_t(f)&=\mcb Y^{N,A}_t(T_{ v_ +N^b t}f)-\mcb Y^{N,B}_t(T_{v_+ N^b t}f)\\
\mcb Z_t^{N,-}(f)&=\mcb Y^{N,A}_t(T_{v_-N^b t}f)+\mcb Y^{N,B}_t(T_{v_- N^bt}f)\,,
\end{aligned}
\end{equation} 
with   $v_\pm=\pm\frac {(E_A-E_C)}{3N^\gamma }$ and $b=a$.

\quad 

{\noindent\textbf{Case (II): $E_{B}=E_C$
}}

\noindent Now instead we should look at the fields 
\begin{equation}\label{e:Case2}
\begin{aligned}
\mcb Z_t^{N,+}(f)&=\mcb Y^{N,A}_t(T_{v_+N^bt}f)+2\mcb Y^{N,B}_t(T_{v_+N^bt}f)\,, \\
\mcb Z_t^{N,-}(f)&=\mcb Y^{N,A}_t(T_{v_-N^bt}f)\,,
\end{aligned}
\end{equation}
with   $v_\pm=\pm\frac {(E_C-E_A)}{3N^\gamma }$ and $b=a$.

\quad

{\begin{remark}
The fields for the case $E_A=E_C$ are the same as for the Case \textbf{(II)} with $A$ and $B$ interchanged.
\end{remark}}

\quad

\noindent\textbf{Case (III): $E_{B}-E_{A}\neq E_{C}-E_{A}\neq 0$
}

\noindent In this case, we have 
\begin{equation}\begin{split}\label{eq:gen_fields_n}
		&{\mcb {Z}}^{N,+}_t(f)=\mcb Y^{N,A}_t(T_{v_+ N^b t}f)+\frac{c_-}{E_A-E_{C}}\mcb Y^{N,B}_t(T_{v_+N^bt}f),\\
&\mcb {Z}^{N,-}_t(f)=\mcb Y^{N,A}_t(T_{v_-N^bt}f)+\frac{c_+}{E_A-E_{C}}\mcb Y^{N,B}_t(T_{v_-N^bt}f),\end{split}
\end{equation}
where again $b=a$, 
\begin{equation}~\label{assumption}
v_\pm=\pm\frac {\delta}{2N^\gamma }\qquad\text{and}\qquad c_\pm:=E_A-E_B\pm\tfrac32\delta\,.
\end{equation}
Here $\delta$ is given by
\begin{equation}\label{e:delta}
\delta:= \tfrac23\sqrt{(E_A-E_C)^2+(E_B-E_C)^2-(E_A-E_C)(E_B-E_C)}\,.
\end{equation}
\quad

{\begin{remark} Note that the completely symmetric case $E_A=E_B=E_C$ is contained in Cases \textbf{(I)}  and \textbf{(II)}, and observe that in this case $v_{\pm}=0$.
\end{remark}}

%
%

\begin{remark}\label{rem:(I)(II)} Case \textbf{(II)} can be obtained as a consequence of Case \textbf{(I)}, looking at the pair $(\mcb Y^{N,B}_t, \mcb Y^{N,C}_t)$ instead of $(\mcb Y^{N,A}_t, \mcb Y^{N,B}_t)$ and recalling the relation $\mcb Y^{N,A}_t+\mcb Y^{N,B}_t+\mcb Y^{N,C}_t=0$. 
For this reason, we will only explicitly treat cases \textbf{(I)} and \textbf{(III)}. 
\end{remark}

Before stating the result concerning the fluctuations of the density fields, let us introduce the notion of solution for the stochastic Burgers equation and the Ornstein--Uhlenbeck equation, as they are the equations describing the limiting behaviour. 
For the former, we follow the definition of {\it energy solutions} proposed in~\cite{GJS17}.

\subsection{Energy solution to the stochastic Burgers equation}

Let us recall that the stochastic Burgers equation is the SPDE given by
\begin{equation}\label{eq:SBE}
\partial_{t} \mcb{Z}_{t}= \Delta \mcb{Z}_{t}+{m\nabla\mcb Z_t}+ \lambda\nabla\mcb Z_t^2+\sqrt{2 \sigma^{2}} \nabla d\mathscr{W}_{t},
\end{equation}
where $\lambda\in\mathbb R$, $\sigma\neq 0$ and {$m\in\R$ are constant (and we allow $m$ to be random)},  
and $\mathscr{W}_{t}$ is a  $\mcb D'(\mathbb T)$-valued Brownian motion with covariance given on $f,g\in\mcb D(\mathbb T)$ by \begin{equation}\label{eq:cov_BM}\mathbb E[\mathscr{W}_{t}(\psi)\mathscr{W}_{s}(\varphi)]=(t\wedge s)\langle \psi,\varphi\rangle_{L^2(\mathbb T)}.\end{equation} 
{Above $\mcb D'(\mathbb T)$ denotes the space of $\mathbb R$-valued distributions on $\mathbb T $. } The notion of solution to~\eqref{eq:SBE} we will use, is that of {\it energy solutions}, first proposed in~\cite{GJ14}, 
which allows to make sense of the non-linearity in its expression. To see how this works, 
let $\left\{\imath_{\varepsilon} ; \varepsilon \in(0,1)\right\}$ be an approximation of the identity and $\varphi \in \mcb {D}(\mathbb T)$. We define the process $\left\{\mcb{B}_{t}^{\varepsilon} ; t \in[0, T]\right\}$ as\footnote{ In some cases, it is more convenient to define this process as $\int_{0}^{t} \int_{\mathbb{R}}\left(\mcb{Z}_{s} * \ola{\iota_{\varepsilon}}(u)\left(\mcb{Z}_{s} * \ora{\iota_{\varepsilon}}(u)\right)\right) \nabla \varphi(u) \,\dd u \,\dd s$, where $\ola{\iota_{\varepsilon}}(u)$ and $\ora{\iota_{\varepsilon}}(u)$ are approximations of the identity over intervals to the left or to right of $u$ respectively. Observe that the $\varepsilon\to0$ limits of this process and~\eqref{eq:q_end} coincide. The reason why we need this additional definition will be clear in Section~\ref{sec:char_limit}, see~\eqref{Beps}.}
\begin{equation}\label{eq:q_end}
	\mcb{B}_{t}^{\varepsilon}(\varphi):=\int_{0}^{t} \int_{\mathbb{R}}\left(\mcb{Z}_{s} * \imath_{\varepsilon}(u)\right)^{2} \nabla \varphi(u) \,\dd u \,\dd s,
\end{equation}
where $*$ denotes the convolution operator and $\imath_\eps(u)=\imath_\eps(u)(\cdot)=\imath_\eps(u-\cdot)$. 
In the context of energy solutions, the non-linear term in the stochastic Burgers equation is interpreted as the limit for 
$\varepsilon\to 0$ of the process $\mcb{B}_t^\varepsilon$. 
If the process $\left\{\mcb{Z}_{t} ; t \in[0, T]\right\}$ satisfies an {\it energy estimate}, 
i.e. if there exists a finite constant $C>0$ such that 
for any {$0\leq s \leq t \leq T$}, any $0<\delta \leq \varepsilon<1$ and any $f \in \mcb{D}(\mathbb T)$, we have
\begin{equation}\label{eq:energy_estimate}
	\mathbb{E}\left[\left((\mcb{B}_{s}^{\varepsilon}(\varphi)-\mcb{B}_{t}^{\varepsilon}(\varphi))-(\mcb{B}_{s}^{\delta}(\varphi)-\mcb{B}_{t}^{\delta}(\varphi))\right)^{2}\right] \leq C \varepsilon(t-s)\|\nabla \varphi\|_{2}^{2},
\end{equation}
where $\|\cdot\|_2$ is the usual $L^2(\T)$-norm, then the limit
\begin{equation}\label{eq:limitB}
	\mcb{B}_{t}(\varphi):=\lim _{\varepsilon \rightarrow 0} \mathscr{B}_{t}^{\varepsilon}(\varphi)
\end{equation}
exists in $L^2$ and does not depend on the choice of $\left\{\imath_{\varepsilon} ; \varepsilon \in(0,1)\right\}$.
Hence, we can postulate that the integral in time of the square of the distribution-valued process $\mcb{Z}_{s}$ 
is given by $\mcb{B}_{t}$ and pose the Cauchy problem for the stochastic Burgers equation. 
Nevertheless, the energy estimate~\eqref{eq:energy_estimate} alone does not guarantee uniqueness 
so it will not be sufficient to prove convergence of our sequence of density fields. 
To overcome the issue, in~\cite{gubinelli_jara}, 
the authors imposed an extra condition on the reversed problem, 
which then led to the proof of uniqueness in~\cite{gubinelli_perkowski}.  
This extra condition consists of requiring the reversed process to satisfy the same martingale problem 
as the original process, but with respect to the adjoint generator. 
Let us remark that this is very natural (and easy to check) in the context of interacting particle systems, 
since it simply corresponds to considering the same process but with reversed rates. 

We are now ready to characterise the solution to the stochastic Burgers equation with the following definition. {Let ${C}([0, T], {\mcb D}'(\mathbb T))$ (resp. $D([0,T], \mathcal D'(\mathbb T))$ be the space of continuous (resp. c\`adl\`ag) functions defined on  $[0,T]$ and taking values in ${\mcb D}'(\mathbb T)$.}

\begin{defin} ~\label{def:energy_solution}We say that a process $\left\{\mcb{Z}_{t} ; t \in[0, T]\right\}$ with trajectories in ${C}([0, T], {\mcb D}'(\mathbb T))$ is a \textit{stationary energy solution} of the stochastic Burgers equation  given in \eqref{eq:SBE}
if
	\begin{itemize}
	\item[(i)] for each $t \in[0, T]$ the $\mcb {D}^{\prime}(\mathbb{T})$-valued random variable $\mcb{Z}_{t}$ is a white noise of covariance $\sigma^{2}$;
	
	\item[(ii)] the process $\left\{\mcb{Z}_{t} ; t \in[0, T]\right\}$ satisfies the energy estimate~\eqref{eq:energy_estimate};
	
	\item[(iii)] for any $\varphi \in \mcb{D}(\mathbb T)$ and $t\in[0,T]$, the process
	\begin{equation}\label{eq:mart_prob}
		\mcb M_t(\varphi)=\mcb{Z}_{t}(\varphi)-\mcb{Z}_{0}(\varphi)-\int_{0}^{t} \mcb{Z}_{s}( \Delta\varphi) \dd s {+m\int_0^t \mcb Z_s(\nabla f)\dd s}+\lambda \mcb{B}_{t}(\varphi)
	\end{equation}
	is a continuous martingale with respect to the natural filtration associated to $\mcb{Z}_{\cdot}$, of quadratic variation
	\begin{equation*}
	\langle\mcb M(\varphi)\rangle_t=t2\sigma^2\|\nabla\varphi\|_{L^2(\mathbb T)}.
	\end{equation*}
Above the process $\{\mcb B_t, t\in[0,T]\}$ is obtained as the $L^2$-limit~\eqref{eq:limitB};
	\item[(iv)] the reversed processes $\left\{\left(\mcb{Z}_{T-t}, \mcb{B}_{T-t}-\mcb{B}_{T}\right) ; t \in[0, T]\right\}$ also satisfy (iii) with $\lambda$ replaced by $-\lambda$.
	\end{itemize}
\end{defin}

\begin{defin}\label{def:OU}
	We say that $\mcb{Z}_{t}$ is a stationary solution of the Ornstein--Uhlenbeck equation 
	\begin{equation*}
		\partial_{t} \mcb {Z}_{t}= \Delta \mcb{Z}_{t}+{m\nabla\mcb Z_t}+\sqrt{2 \sigma^{2}} \nabla d\mathscr{W}_{t},
	\end{equation*}
	if it satisfies conditions (i) and (iii) in Definition~\ref{def:energy_solution} with $\lambda=0$. \end{defin}
	
In the next theorem, we state existence and uniqueness for the energy solutions to~\eqref{eq:SBE}, {
whose existence was derived in \cite{GJ14,gubinelli_jara} and the uniqueness in~\cite{gubinelli_perkowski}} . 

\begin{thm}{~\cite[Theorem 2.4]{gubinelli_perkowski}}\label{thm:Uniqueness}
For any value of $\lambda\in\R$, $\sigma\neq 0$ {and (random) constant $m\in\R$}, 
the stochastic Burgers equation~\eqref{eq:SBE} 
has a stationary energy solution, given in Definition~\ref{def:energy_solution}, which is unique up to indistinguishability. 
\end{thm}

{\begin{remark}
With respect to~\cite{GJ14,gubinelli_jara, gubinelli_perkowski}, 
the SBE we are considering features an additional transport term which is linear and 
therefore does not alter the existence/uniqueness theory established therein. 
\end{remark}}

\subsection{Main result}

Now we are ready to state the main results of this article, which are summarised in the next theorem.

\begin{thm}\label{thm:main_new}
In the diffusive time scaling $a=b=2$,  the sequence of processes $(\mcb Z^{N,+},{\mcb Z}^{N,-})_{N\in\mathbb N}$, whose definition depends on the values of the constants $E_A,E_B$ and $E_C$ and is given in Section~\ref{sec:flucfields}, converges in law in { $C([0,T], \mcb {D}'(\mathbb T)^2)$} to $(\mcb Z^+,\mcb Z^-)$, where $\mcb Z^+$ and $\mcb Z^-$ are stationary solutions of stochastic Burgers equations of the form
\begin{equation}	\label{eq:SBE_main}
d\mcb Z^\pm_t=\Delta\mcb Z^\pm_tdt+{m_{\pm}\nabla \mcb Z^{\pm}+}\lambda_\pm\nabla(\mcb Z^\pm)^2_t+\sqrt{2\sigma^2_\pm}\nabla d\mcb W^\pm_t,
\end{equation}
{started with independent initial conditions}, 
where $\mcb{W}^+$ and $\mcb W^-$ are {independent} $\mcb D(\mathbb T)'$-valued Brownian motions 
with covariances given by \eqref{eq:cov_BM} {
and the coefficients $m_{\pm}$}, $\lambda_\pm$, $\sigma^2_\pm$ are given by 
\medskip
	
\noindent\textbf{Case (I-II)}:  $\sigma^2_+=\tfrac 23$ and $\sigma^2_-=\tfrac29$.
\begin{itemize}
	\item[(i)]  when $\gamma>1/2$, $\lambda_+=\lambda_-=0{=m_-=m_+}$, so that $\mcb Z^{\pm}$ are the unique solutions to the corresponding  Ornstein--Uhlenbeck equations,

	\item[(ii)] when $\gamma=\tfrac 12$, $\lambda_+=0$ {and $m_+=(E_A-E_C)\overline{\mcb Z}^{\,-}$ where $\overline{\mcb Z}^{\,-}$ is random and given by the total mass of the initial condition of $\mcb Z^-$}, while $\lambda_-=E_C-E_A$ {and $m_-=0$}.	
\end{itemize}

\noindent\textbf{Case (III)}: $\sigma^2_\pm=\tfrac29 (1+\frac{c_\mp^2}{(E_A-E_C)^2}-\frac {c_\mp}{(E_A-E_C)})$, for $c_\pm$ defined 
	according to~\eqref{assumption}, and 
\begin{itemize}
	\item[(i)]  when $\gamma>1/2$, $\lambda_+=\lambda_-=0{=m_-=m_+}$,
	\item[(ii)]  when $\gamma=1/2$, $\lambda_+=-\tfrac{E_A-E_C}{3\delta}(c_+-E_B+E_C)$ {and  $m_+=-E_A(1-\tfrac{c_+-E_B}{3\delta})\overline{\mcb Z}^{\,-}$}, while  $\lambda_-=-\tfrac{E_A-E_C}{3\delta}(E_B-E_C-c_-)$ {and  $m_-=-E_A(1-\tfrac{E_B-c_-}{3\delta})\overline{\mcb Z}^{\,+}$, where, as above, $\overline{\mcb Z}^{\mp}$ is the total mass of the initial condition}.
\end{itemize}
\end{thm}

{\begin{remark}\label{rem:Mass}
The additional massive transport term in~\eqref{eq:SBE_main} should not come as a surprise as the total mass of the initial 
condition is a conserved quantity of the system (thus independent of time) and 
{\it does not see} the diverging velocity. To witness, let $t\mapsto\eta_t$ be the 
standard periodic SSEP in diffusive time scale and evolving on $\T_N$ started from the  Bernoulli product measure with parameter $\rho\in(0,1)$ and 
$\mcb Y^N_t (du):=\tfrac{1}{\sqrt{N}}\sum_{x\in\T_N}(\eta_{t}(x)-\rho)\delta_{x/N}(du)$ be its density fluctuation  field. 
Let $f$ be a smooth non-negative non-zero periodic function and $\bar f=\int_\T f(u)\dd u$. 
Notice that for any $t>0$ and $\alpha>1$ we have 
\begin{align*}
\int_0^t \mcb Y^N_s (T_{N^\alpha s} f)\dd s&=\int_0^t \mcb Y^N_s (T_{N^\alpha s} (f-\bar f))\dd s+\int_0^t \mcb Y^N_s (T_{N^\alpha s} \bar f)\dd s\\
&=\int_0^t \mcb Y^N_s (T_{N^\alpha s} (f-\bar f))\dd s+\bar f\int_0^t \mcb Y^N_s (1)\dd s
\end{align*}
where $\mcb Y^N_s (1)$ denotes the field $\mcb Y^N_s$ tested against the function constantly equal to $1$. 
Now, by Theorem~\ref{thm:Cross}, the first term converges to $0$. For the other, $\mcb Y^N_s (1)$ is nothing but the 
total mass of the fluctuation field which is a conserved quantity of the system and therefore independent of time. 
In particular, for every $s$, $\mcb Y^N_s (1)=\overline{\mcb Y^N}$ and the latter converges in law to a 
Gaussian random variable $\overline{\mcb Y}$ 
with variance $\rho(1-\rho)$. Therefore, we have showed that 
\begin{equation*}
\lim_{N\to\infty} \int_0^t \mcb Y^N_s (T_{N^\alpha s} f)\dd s= t \bar f \overline{\mcb Y}
\end{equation*}
in probability. 
In other words, even though we look at the process in the wrong time-frame the integral 
in time of the field does not fully vanish but leaves as a trace its total mass. 
\end{remark}

\begin{remark}
Note that the transport term in~\eqref{eq:SBE_main} only appears for $\gamma=1/2$ 
and we could have avoided its presence by an additional Galilean transformation, 
i.e. by modifying the definition of the velocities in Section~\ref{sec:flucfields}. 
More precisely, one can show that upon setting 
\begin{itemize}[noitemsep]
\item in \textbf{Case (I-II)} $\tilde v_+:= v_++m_+\tfrac{\overline{\mcb Z}^{\,-}}{N}$ and $\tilde v_-:=v_-$, 
\item in \textbf{Case (III)} $\tilde v_{\pm}:= v_{\pm}+m_{\pm}\tfrac{\overline{\mcb Z}^\mp}{N}$
\end{itemize}
the sequence $(\tilde{\mcb Z}^{N,+},\tilde{\mcb Z}^{N,-})_{N\in\mathbb N}$ defined as in Section~\ref{sec:flucfields} 
but with $\tilde v_\pm$ instead of $v_\pm$, converges to the couple 
$(\tilde{\mcb Z}^{+},\tilde{\mcb Z}^{-})$ solving~\eqref{eq:SBE_main} with $m_\pm= 0$ and 
the same choice of the parameters $\lambda_\pm$ and $\sigma_\pm$ as in the above statement.  
Further, the limiting processes $(\tilde{\mcb Z}^{+},\tilde{\mcb Z}^{-})$ can be proven to satisfy 
the cylinder martingale problem of~\cite[Section 5.1]{GPGen}, and are therefore independent. 

That said, the velocities $\tilde v_\pm$ are random and their definition looks rather artificial, at least at first sight. 
Moreover, we believe that the observation that these transport terms appear is interesting in its own right 
so that we preferred to present Theorem~\ref{thm:main_new} as stated. 
\end{remark}}

{\begin{remark}
Note, in particular, from Case {\textbf{(I)}} or {\textbf{(II)}}, that in the completely symmetric cases $E_A=E_B=E_C$, we have that $\lambda_{+}=\lambda_-=0{=m_-=m_+}$ and so, as expected, $\mcb Z^{\pm}$ are solutions of the the Ornstein--Uhlenbeck equation, as in the SSEP.
\end{remark}}

The strategy of proof of the previous theorem consists of first showing tightness of the sequence of fields to be analysed, 
with respect to the Skorokhod topoplogy on $D([0,T], \mcb D'(\mathbb T))$ (see Section \ref{sec:tightness}). 
From this together with {Prokhorov's} theorem, the processes converge along subsequences. 
In order to characterise the limit uniquely, by Theorem~\ref{thm:Uniqueness}, it suffices to show that 
any limit point either satisfies Definition~\ref{def:OU} or points (i)-(iv) of Definition~\ref{def:energy_solution}. 
In this latter case, we follow the  strategy of \cite{GJ14}, which has since been applied to many different models, 
but here we generalise it to the setting in which the system has several conservation laws. 
When instead the limit is a solution to the Ornstein-Uhlenbeck equation, there are two cases to distinguish.  
If $\gamma>1/2$, the asymmetry becomes negligible and the limiting fields evolve independently, 
therefore the proof follows by adapting standard arguments. If instead $\gamma=1/2$, 
the evolution of each field depends non-linearly on the other and therefore 
novel tools are needed in order to show that any crossed term ultimately vanishes (see Section~\ref{sec:cross}).

In the next section, via Dynkin's formula, we derive a collection of martingales for each of the cases {\bf (I)-(III)}. 
These martingales provide a weak formulation for the dynamics of our process and represent 
the building blocks of our arguments.

 \section{Associated martingales }\label{sec:martn}

In order to prove Theorem~\ref{thm:main_new}, we study the microscopic dynamics of the fields 
in each of the cases \textbf{ (I)}, \textbf{(II)} and \textbf{(III)}.
To do so in the greatest possible generality, we start by presenting some  algebraic computations which are the basis 
for the derivation of the martingale terms in point (iii) of Definition \ref{def:energy_solution}.
Recall \eqref{eq:densityfieldNL}. 
We consider now a generic field $\mcb Z^N$ given by a linear combination of 
$\mcb Y^{N,A}_t$ and $\mcb Y^{N,B}_t$, i.e.
\begin{equation}\label{eq_gen_field}
	\mcb Z_t^N (f)=D_1\mcb Y^{N,A}_t(T_{v N^b t}f)+D_2\mcb Y^{N,B}_t (T_{v N^b t}f)\,,
\end{equation} 
where $D_1, D_2,b>0$ and  $v\in\mathbb R$ will be fixed later. 
From Dynkin's formula, see e.g.\cite[Appendix A.1.5]{KL},  we know that for $f\in\mcb D(\mathbb T)$
\begin{equation}\label{eq:Dynkin_gen}
\begin{aligned}
\mcb M_t^N(f)=&{\mcb Z}^{N}_t(f)-{\mcb Z}^{N}_0(f)-\int_0^t (L_N+\partial_s) \mcb Z_s^N(f)\dd s 
\end{aligned}
\end{equation}
is a martingale with respect to the natural filtration of the process. 
It is (tedious but) not hard to see that the  quadratic variation of $\mcb M^{N}_t(f)$ is given by
\begin{equation}\label{eq:QV}
	\begin{aligned}
		\langle\mcb {M}^{N}(f)\rangle_t=  & N^{a-3}\int_0^t \sum_x c_x(\eta)(\nabla_N T_{v N^bs}f(\tfrac xN))^2\big[(D_1 \xi^{A}_{x+1}+D_2 \xi^{B}_{x+1})-(D_1\xi^{A}_{x}+D_2\xi^{B}_{x})\big]^2\dd s.
	\end{aligned}
\end{equation}
In the next lemma, we show that the expectation of the quadratic variation converges to a non-zero finite constant
{\it only if}  the scaling chosen is diffusive, which means that $a$ must necessarily be equal to $2$. 

\begin{lem}\label{lemmaqv}
{For $f\in\mcb D(\mathbb T)$, let $\mcb {M}^{N}(f)$} be the martingale in~\eqref{eq:Dynkin_gen}. Then, under the diffusive scaling $a=2$
\begin{equation}\label{eq:lim_QV}
	\lim_{N\to\infty} \mathbb{E}_{\nu_\rho}[\langle \mcb  M^{N}(f)\rangle_t]=\frac{4}{9}(D_1^2+D_2^2-D_1D_2)t\|\nabla f\|_{2}^2\,.
\end{equation}
\end{lem}
\begin{proof} Note first that for any $x\in \mathbb T_N$,
	\begin{equation}\label{calculation_quadratic_variation}
		\mathbb{E}_{\nu_\rho}\big[c_x(\eta)[(D_1\xi^{A}_{x+1}+D_2\xi^{B}_{x+1})-(D_1\xi^{A}_{x}+D_2\xi^{B}_{x})]^2\big]=\frac{4}{9}(D_1^2+D_2^2-D_1D_2)\,.
	\end{equation} 	
	To prove~\eqref{calculation_quadratic_variation},
recall the definition of $c_x(\eta)$ in \eqref{cxeta}, and note that for $\alpha, \beta\in\{A,B,C\}$ such that $\alpha\neq\beta$,  it holds {$c^{\alpha,\beta}+c^{\beta,\alpha}=2.$} The left-hand side of \eqref{calculation_quadratic_variation} can be written as
	\begin{equation}\label{calc_qv2}
		\begin{aligned}
			D_1^2\,\E_{\nu_\rho}\Big[\sum_{\alpha,\beta}c^{\alpha,\beta}\xi^{\alpha}_{x}\xi^{\beta}_{x+1}\xi^A_{x+1}&+\sum_{\alpha,\beta}c^{\alpha,\beta}\xi^{\alpha}_{x}\xi^{\beta}_{x+1}\xi^A_{x}\Big]+D_2^2\,\E_{\nu_\rho}\Big[\sum_{\alpha,\beta}c^{\alpha,\beta}\xi^{\alpha}_{x}\xi^{\beta}_{x+1}\xi^B_{x+1}+\sum_{\alpha,\beta}c^{\alpha,\beta}\xi^{\alpha}_{x}\xi^{\beta}_{x+1}\xi^B_{x}\Big]\\
			&+2D_1D_2\E_{\nu_\rho}\Big[\sum_{\alpha,\beta}c^{\alpha,\beta}\xi^{\alpha}_{x}\xi^{\beta}_{x+1}\xi^A_{x+1}\xi^B_{x+1}+\sum_{\alpha,\beta}c^{\alpha,\beta}\xi^{\alpha}_{x}\xi^{\beta}_{x+1}\xi^A_{x}\xi^B_{x}\Big]\\
			&-2D_1D_2\E_{\nu_\rho}\Big[\sum_{\alpha,\beta}c^{\alpha,\beta}\xi^{\alpha}_{x}\xi^{\beta}_{x+1}\xi^A_{x+1}\xi^B_{x}+\sum_{\alpha,\beta}c^{\alpha,\beta}\xi^{\alpha}_{x}\xi^{\beta}_{x+1}\xi^A_{x}\xi^B_{x+1}\Big]\\
			&-2\E_{\nu_\rho}\Big[D_1^2\sum_{\alpha,\beta}c^{\alpha,\beta}\xi^{\alpha}_{x}\xi^{\beta}_{x+1}\xi^A_{x+1}\xi^A_{x}+D_2^2\sum_{\alpha,\beta}c^{\alpha,\beta}\xi^{\alpha}_{x}\xi^{\beta}_{x+1}\xi^B_{x}\xi^B_{x+1}\Big].
		\end{aligned}
	\end{equation}
	By the exclusion rule $\xi^A_{y}\xi^B_y=0$, so that the second line of \eqref{calc_qv2} is equal to zero.
	The last line of \eqref{calc_qv2} is also equal to zero because $c^{A,A}=c^{B,B}=0$, since there are no exchanges between particles of the same type.
	The third line of \eqref{calc_qv2} is equal to
	\begin{equation*}
		-2D_1D_2 \left(c^{B,A}\E_{\nu_\rho}\Big[\xi^A_{x+1}\xi^B_x\Big]+c^{A,B}\E_{\nu_\rho}\Big[\xi^B_{x+1}\xi^A_x\Big]\right)=-\frac{2}{9}D_1D_2[c^{B,A}+c^{A,B}]=-\frac{4}{9}D_1D_2\,,
	\end{equation*}
	while the first is given by
	\begin{equation*}
		\begin{aligned}
			D_1^2 \,\E_{\nu_\rho}\Big[\sum_{\alpha}c^{\alpha , A} \xi^\alpha_x\xi^A_{x+1}+&\sum_{\beta}c^{ A, \beta} \xi^\beta_{x+1}\xi^A_{x}\Big]+D_2^2\,\E_{\nu_\rho}\Big[\sum_{\alpha}c^{\alpha, B} \xi^\alpha_x\xi^B_{x+1}+\sum_{\beta}c^{ B, \beta} \xi^\beta_{x+1}\xi^B_{x}\Big]\\
			&=\frac{1}{9}D_1^2\sum_{\alpha\in\{B,C\} }(c^{\alpha, A}+c^{A, \alpha})+\frac{1}{9}D_2^2\sum_{\alpha\in\{A,C\}}(c^{\alpha, B}+c^{B, \alpha})={\frac{4}{9}}(D_1^2+D_2^2)\,,
		\end{aligned}
	\end{equation*}
	from which~\eqref{calculation_quadratic_variation} follows. 
	
Using \eqref{eq:QV} and \eqref{calculation_quadratic_variation} a simple computation proves the statement.
\end{proof}

We turn our attention to the right-hand side of~\eqref{eq:Dynkin_gen} and 
we compute explicitly the expression inside the time integral. 
Observe that
\begin{equation}\label{eq:act_der}
\partial_s\mcb Z_s^N(f)=v N^{b-1} \mcb Z_s^N(\nabla f)\,.
\end{equation}
Moreover,  to determine $L_N\mcb Z_s^N(f)$ we need to understand the action of the generator 
on each type of particle, i.e. $L_N \xi^A_x$ and 
$L_N \xi^{B}_x$. It is immediate to see that
$L_N \xi^A_x=N^a(j_{x-1,x}^A-j_{x,x+1}^{A})$, where, 
for particles of species $A$, the infinitesimal current is given by
\begin{equation}\label{eq:inst_current}
	\begin{aligned}
		j^A_{x,x+1}=&\xi^{A}_{x}-\xi^{A}_{x+1}+\frac{E_{A}-E_{B}}{2N^\gamma}(\xi^{A}_{x}\xi^{B}_{x+1}+\xi^{B}_{x}\xi^{A}_{x+1})\\
		&+\frac{E_{A}-E_{C}}{2N^\gamma}(\xi^{A}_{x+1}+\xi^{A}_{x}-2\xi^{A}_{x}\xi^{A}_{x+1}-\xi^{A}_{x}\xi^{B}_{x+1}-\xi^{B}_{x}\xi^{A}_{x+1})\,,
	\end{aligned}
\end{equation}
while $L_N \xi^{B}_x=j_{x-1,x}^{B}-j_{x,x+1}^{B}$
and  the infinitesimal current for particles of species $B$ is given by the same expression with $A$ and $B$ interchanged.

We now look at the infinitesimal current $\bar{j}^\alpha$, $\alpha\in\{A,B\}$ for the centred variables 
$\bar{\xi}^\alpha_x:=\xi^\alpha_x-\E_{\nu_\rho}[\xi^\alpha_x]=\xi^\alpha_x-1/3$, $x\in\T_N$, which is 
\begin{equation}\label{eq:inst_current_A_n}
	\begin{aligned}
		\bar{j}^A_{x,x+1}=&\bar{\xi}^A_{x}-\bar{\xi}^A_{x+1}-\frac{E_{B}-E_{A}}{6N^\gamma}(\bar{\xi}^{A}_{x}+\bar{\xi}^{A}_{x+1})-\frac{E_{B}-E_{C}}{6N^\gamma}(\bar{\xi}^{B}_{x}+\bar{\xi}^{B}_{x+1})
		\\
		&+\frac{E_{C}-E_{A}}{N^\gamma}\bar{\xi}^A_x\bar{\xi}^A_{x+1}-\frac{E_{B}-E_{C}}{2N^\gamma}( \bar{\xi}^{A}_{x}\bar{\xi}^{B}_{x+1}+\bar{\xi}^{B}_{x}\bar{\xi}^{A}_{x+1})
	\end{aligned}
\end{equation}
and 
\begin{equation}\label{eq:inst_current_B_n}
\begin{aligned}
\bar{j}^{B}_{x,x+1}=&\bar{\xi}^{B}_{x}-\bar{\xi}^{B}_{x+1}+\frac{E_{B}-E_{A}}{6N^\gamma}(\bar{\xi}^{B}_{x}+\bar{\xi}^{B}_{x+1})-\frac{E_{A}-E_{C}}{6N^\gamma}(\bar{\xi}^{A}_{x}+\bar{\xi}^{A}_{x+1})\\
&+\frac{E_{C}-E_{B}}{N^\gamma}\bar{\xi}^{B}_x\bar{\xi}^{B}_{x+1}-\frac{E_{A}-E_{C}}{2N^\gamma}( \bar{\xi}^{A}_{x}\bar{\xi}^{B}_{x+1}+\bar{\xi}^{B}_{x}\bar{\xi}^{A}_{x+1})\,.
\end{aligned}
\end{equation}

Thanks to the previous, a simple, but long computation,  shows that  for $f\in \mcb D(\mathbb T)$,
\begin{align}\label{dynkin_case_gen_n}
		&\mcb M_t^{N}(f)={\mcb Z}^{N}_t(f)-{\mcb Z}^{N}_0(f)-N^{a-2}\int_0^t \dd s \mcb Z_s^{N}(\Delta_Nf)\nonumber\\
		&-\frac{N^{a-3/2}}{N^\gamma}\int_0^t \dd s \sum_{x\in\mathbb{T}_N}\nabla_N T_{v N^bs}f\left(\tfrac{x}{N}\right)\Big(D_1(E_{C}-E_A)\bar\xi^A_{x}(s)\bar\xi^A_{x+1}(s)+D_2(E_{C}-E_{B})\bar\xi^{B}_{x}(s)\bar\xi^{B}_{x+1}(s)\Big)\nonumber\\
		&+\frac{N^{a-3/2}}{2N^\gamma}\Big(D_1(E_{B}-E_{C})+D_2(E_{A}-E_{C})\Big)\int_0^t \dd s \sum_{x\in\mathbb{T}_N}\nabla_N T_{v N^bs}f\left(\tfrac{x}{N}\right)\Big(\bar{\xi}^{A}_{x}(s)\bar{\xi}^{B}_{x+1}(s)+\bar{\xi}^{B}_{x}(s)\bar{\xi}^{A}_{x+1}(s)\Big)\nonumber\\
		&+D_1\frac{N^{a-3/2}}{N^\gamma}\int_0^t \dd s \sum_{x\in\mathbb{T}_N}\nabla_N T_{v N^bs}f\left(\tfrac{x}{N}\right)\Big(\frac{E_{B}-E_A}{3}\bar\xi_x^A(s)+\frac{E_{B}-E_{C}}{3}\bar\xi^{B}_{x}(s)\Big)		\\
		&+D_2\frac{N^{a-3/2}}{N^\gamma}\int_0^t \dd s \sum_{x\in\mathbb{T}_N}\nabla_N T_{v N^bs}f\left(\tfrac{x}{N}\right)\Big(\frac{E_{A}-E_{C}}{3}\bar\xi^{A}_{x}(s)-\frac{E_{B}-E_{A}}{3}\bar\xi^{B}_{x}(s)\Big)\nonumber\\
		&-D_1\frac{N^{a-5/2}}{N^\gamma}\int_0^t \dd s \sum_{x\in\mathbb{T}_N}\Delta_N T_{v N^bs}f\left(\tfrac{x}{N}\right)\Big(\frac{E_{B}-E_A}{6}\bar\xi_x^A(s)+\frac{E_{B}-E_{C}}{6}\bar\xi^{B}_{x}(s)\Big)\nonumber\\
		&-D_2\frac{N^{a-5/2}}{N^\gamma}\int_0^t \dd s \sum_{x\in\mathbb{T}_N}\Delta_N T_{v N^bs}f\left(\tfrac{x}{N}\right)\Big(\frac{E_A-E_{C}}{6}\bar\xi^A_{x}(s)-\frac{E_{B}-E_A}{6}\bar\xi^{B}_{x}(s)\Big)\nonumber\\
		&+\frac{v N^{b-1}}{\sqrt N}\int_0^t \dd s\sum_{x\in\mathbb T_N} \nabla T_{v N^bs}f\left(\tfrac{x}{N}\right)\Big(D_1\bar\xi^A_{x}(s)+D_2\bar\xi^{B}_{x}(s)\Big)\,.\nonumber
	\end{align}
Above, the discrete Laplacian and the discrete derivative operators are defined, for $x\in\mathbb  T_N$, by
	\begin{equation}\label{eq:discretederivatives}
	\begin{aligned}
	\Delta_N f\left(\frac{x}{N}\right) &=N^2\left\{f\left(\frac{x+1}{N}\right)-2f\left(\frac{x}{N}\right)+f\left(\frac{x-1}{N}\right)\right\}\\
	\nabla_N f\left(\frac{x}{N}\right) &=N\left\{f\left(\frac{x+1}{N}\right)-f\left(\frac{x}{N}\right)\right\}.
	\end{aligned}
	\end{equation}

\subsection{Finding the fluctuation fields}

In view of Lemma~\ref{lemmaqv}, let us fix the diffusive scaling $a=2$. 
First we observe that, independently of the values of $D_1,D_2,v$ and $b$, 
the variance of the terms in the sixth and seventh lines {in \eqref{dynkin_case_gen_n}} are of order $O(N^{-2\gamma})$, 
so that they can be neglected for any $\gamma>0$. 
On the other hand, one sees that the variance of the terms in the fourth and fifth lines of \eqref{dynkin_case_gen_n} are of order $O(N^{2-2\gamma})$, 
which explodes for any $\gamma\in(0,1)$, while that on the eighth is of order $O(N^{2b-3})$. 
Therefore, we are forced to choose the velocity $v$, the scaling $b$ and the constants $D_1$ and $D_2$, 
in order to annihilate them. 
Note that, {after a replacement of discrete derivatives by continuous ones, which can be done by paying a price of a lower order with respect to $N$} their sum is
\begin{equation}\label{dynkin_gen}
	\begin{aligned}
	&\Big\{N^{1/2-\gamma}\Big(D_1\frac{E_{B}-E_A}{3}+ D_2\frac{E_A-E_{C}}{3}\Big)+N^{b-3/2} D_1 v\Big\}
		\int_0^t \dd s \sum_{x\in\mathbb{T}_N}\nabla T_{v N^bs}f\left(\frac{x}{N}\right)\bar\xi^{A}_{x}(s)\\
		&+\Big\{N^{1/2-\gamma}\Big(D_1\frac{E_{B}-E_{C}}{3}- D_2\frac{E_{B}-E_A}{3}\Big)+N^{b-3/2} D_2 v\Big\}
		\int_0^t \dd s \sum_{x\in\mathbb{T}_N}\nabla T_{v N^bs}f\left(\frac{x}{N}\right)\bar\xi^{B}_{x}(s).\end{aligned}
\end{equation} 
For the coefficients of the integrals to be equal to zero, we first want to pick $b$ in such a way that the summands are of the same order, which imposes $b=a=2$, 
and then need to find constants $D_1, D_2$ and $v$ satisfying the system of equations 
\begin{equation}\label{eq:system}
\begin{cases}
	\frac{1}{N^\gamma}\Big(D_1\frac{E_{B}-E_A}{3}+ D_2\frac{E_A-E_{C}}{3}\Big)+D_1v=0 \\
	\frac{1}{N^\gamma}\Big(D_1\frac{E_{B}-E_{C}}{3}- D_2\frac{E_{B}-E_A}{3}\Big)+D_2v=0.
\end{cases}
\end{equation}
As the system is overdetermined, 
we fix $D_1=1$. Then, we obtain two solutions, which, in the notation of Section~\ref{sec:flucfields}, are given by
\begin{itemize}
\item[a)] $D_2=D_2^+=\frac{c_-}{E_A-E_C}$ for $c_-$ given in~\eqref{assumption} with $v_+=\frac{\delta}{2N^\gamma}$ 
and $\delta$ defined according to~\eqref{e:delta}. 
\item   [b)] $D_2=D_2^-=\frac{c_+}{E_A-E_C}$ for $c_+$ given in~\eqref{assumption} with $v_-=-\frac{\delta}{2N^\gamma}$ and $\delta$ as above.
\end{itemize}
We are therefore led to consider the fields 
 \begin{equation}\begin{split}\label{eq:gen_fields_n}
		&{\mcb {Z}}^{N,+}_t(f)=\mcb Y^{N,A}_t(T_{v_+ N^2 t}f)+\tfrac{c_-}{E_A-E_C}\mcb Y^{N,B}_t(T_{v_+N^2t}f)\,,\\
&\mcb {Z}^{N,-}_t(f)=\mcb Y^{N,A}_t(T_{v_-N^2t}f)+\tfrac{c_+}{E_A-E_C}\mcb Y^{N,B}_t(T_{v_-N^2t}f)\,,\end{split}
\end{equation}
whose {corresponding} martingales can be written as 
\begin{equation}\label{dynkin_case_gen_nn}
	\begin{aligned}
		\mcb M_t^{N,\pm}(f)=&{\mcb Z}^{N,\pm}_t(f)-{\mcb Z}^{N,\pm}_0(f)-\mcb I_t^{N,\pm}(f)- \mcb B_t^{N,\pm}(f){+\mcb R_t^{N,\pm}(f)}, 
	\end{aligned}
\end{equation}
where $$\mcb I_t^{N,\pm}(f):=\int_0^t \dd s \mcb Z_s^{N,\pm}(\Delta_Nf),$$
\begin{equation}\label{eq:quadratic_term_gen}
	\begin{aligned}
	\mcb B_t^{N,\pm}(f):=&N^{\tfrac12-\gamma}(E_{C}-E_A)\int_0^t \dd s \sum_{x\in\mathbb{T}_N}\nabla_N T_{v_\pm N^2s}f\left(\tfrac{x}{N}\right)\bar\xi^A_{x}(s)\bar\xi^A_{x+1}(s)\\
	&-N^{\tfrac12-\gamma}c_\mp\frac{E_{C}-E_{B}}{E_A-E_C}\int_0^t \dd s \sum_{x\in\mathbb{T}_N}\nabla_N T_{v_\pm N^2 s}f\left(\tfrac{x}{N}\right)\bar\xi^{B}_{x}(s)\bar\xi^{B}_{x+1}(s)\\
	&+N^{\tfrac12-\gamma}\frac{E_B-E_C+c_\mp}{2}\int_0^t \dd s \sum_{x\in\mathbb{T}_N}\nabla_N T_{v_\pm N^2s}f\left(\tfrac{x}{N}\right)\Big(\bar{\xi}^{A}_{x}(s)\bar{\xi}^{B}_{x+1}(s)+\bar{\xi}^{B}_{x}(s)\bar{\xi}^{A}_{x+1}(s)\Big)\,,
	\end{aligned}
\end{equation}
{and $\mcb R_t^{N,\pm}(f)$ contains the terms that vanish in $L^2(\mathbb P_{\nu_\rho})$ as $N\to+\infty$.}
We now investigate the different cases that arise by varying the 
relation among the constants $E_\alpha$ for {$\alpha\in\{A,B,C\}$}.

\subsection{Case (I): $E_{A}-E_{C}=E_{B}-E_{C}=E$}\label{sec:case1}

Under the hypothesis on the rates, the system in~\eqref{eq:system} with $D_1=1$, 
is solved by the following values of the parameters 
\begin{itemize}
\item[a)] $D_2=D_2^+=-1$, $v=v_+=\frac{E}{3N^\gamma}$
\item[b)] $D_2=D_2^-=1$, $v=v_-=-\frac{E}{3N^\gamma}$
\end{itemize}
which give the fields in~\eqref{e:case1}.
{From~\eqref{dynkin_case_gen_nn}, 
the martingales read}
\begin{equation}\label{dynkin_case_gen_nnn}
	\begin{aligned}
		\mcb M_t^{N,\pm}(f)=&{\mcb Z}^{N,\pm}_t(f)-{\mcb Z}^{N,\pm}_0(f)-\mcb I_t^{N,\pm}(f)- \mcb B_t^{N,\pm}(f)+\mcb R_t^{N,\pm}(f), 
	\end{aligned}
\end{equation}
where $$\mcb I_t^{N,\pm}(f):=\int_0^t \dd s \mcb Z_s^{N,\pm}(\Delta_Nf),$$
and, by~\eqref{eq:quadratic_term_gen}
 \begin{equation}\label{eq:quad_case1}
	\begin{aligned}
	\mcb B_t^{N,\pm}(f):=&-N^{\tfrac12-\gamma}E\int_0^t \dd s \sum_{x\in\mathbb{T}_N}\nabla_N T_{v_\pm N^2s}f\left(\tfrac{x}{N}\right)\bar\xi^A_{x}(s)\bar\xi^A_{x+1}(s)\\
	&\mp N^{\tfrac12-\gamma}E\int_0^t \dd s \sum_{x\in\mathbb{T}_N}\nabla_N T_{v_\pm N^2 s}f\left(\tfrac{x}{N}\right)\bar\xi^{B}_{x}(s)\bar\xi^{B}_{x+1}(s)\\
	&+N^{\tfrac12-\gamma}\frac{E\mp E}{2}\int_0^t \dd s \sum_{x\in\mathbb{T}_N}\nabla_N T_{v_\pm N^2s}f\left(\tfrac{x}{N}\right)\Big(\bar{\xi}^{A}_{x}(s)\bar{\xi}^{B}_{x+1}(s)+\bar{\xi}^{B}_{x}(s)\bar{\xi}^{A}_{x+1}(s)\Big)\,,
	\end{aligned}
\end{equation}
and $\mcb R_t^{N,\pm}(f)$ is a term whose {$L^2(\mathbb P_{\nu_\rho})$}-norm vanishes as $N\to\infty$.

Observe that this choice of the fluctuation fields matches that of  Appendix \ref{sub:special_case}. 
Moreover the  prediction is that, in the strong asymmetric regime (i.e. for $\gamma=0$), 
the first field should have diffusive behaviour while the second  {should have} KPZ behaviour.  
To see this (at least in the weakly asymmetric case), it remains to analyse the term $\mcb B_t^{N,\pm}(f)$. 

\subsubsection{The  field $\mcb Z_t^{N,-} $}

From~\eqref{eq:quadratic_term_gen}, we have
\begin{equation}\label{eq:DynkinZ1_diff}
\begin{aligned}
\mcb B_t^{N,-}(f)&=-N^{\tfrac12-\gamma}E\int_0^t \dd s \sum_{x\in\mathbb{T}_N}\nabla_N T_{-\frac{E}{3}N^{2-\gamma}s}f\left(\tfrac{x}{N}\right)\Big\{\bar\xi^A_{x}(s)\bar\xi^A_{x+1}(s)+\bar\xi^{B}_{x}(s)\bar\xi^{B}_{x+1}(s)\\&\quad\quad\quad\quad\quad\quad\quad \quad \quad\quad\quad\quad\quad\quad\quad\quad\quad+\bar\xi^{A}_{x}(s)\bar\xi^{B}_{x+1}(s)+\bar\xi^{B}_{x}(s)\bar\xi^{A}_{x+1}(s)\Big\}.
\end{aligned}
\end{equation}
From the second-order Boltzmann-Gibbs Principle, namely Theorem~\ref{thm:BG}, for $\gamma>1/2$ the  term  $\mcb B_t^{N,-}(f)$  vanishes in $L^2(\mathbb P_{\nu_{\rho}})$ as $N\to\infty$. For $\gamma=1/2$ instead, 
it converges to a non-trivial limit.  
To see what this limit is, note that, for $x\in\Z$, $\alpha, \beta\in\{A,B\}$ and $\varepsilon>0$,  
Theorem~\ref{thm:BG} allows us to replace $\bar\xi^{\alpha}_x$ and $\bar\xi^{\beta}_{x+1}$ with 
$\overleftarrow\xi^{\varepsilon N,\alpha}_{x}$ and $\overrightarrow\xi^{\varepsilon N,\beta}_{x}$ respectively, 
which are the centred averages of $\bar\xi^{\alpha}_x$ and $\bar\xi^{\beta}_{x+1}$
on a box of size $\varepsilon N$ to the left and to the right of $x$, see~\eqref{e:localave}. 
Hence, modulo terms whose $L^2(\mathbb P_{\nu_{\rho}})$-norm vanish as $N\to\infty$, $\mcb B^{N,-}(f)$ becomes
\begin{equation}\label{eq:DynkinZ1_diff_2}
\begin{aligned}
\mcb B_t^{N,-}(f)&=-E\int_0^t \dd s \sum_{x\in\mathbb{T}_N}\nabla_N T_{-\frac{E}{3}N^{3/2}s}f\left(\tfrac{x}{N}\right)\Big(	\ola{\xi}^{A,\varepsilon N}_x(s)\ora{\xi}^{A,\varepsilon N}_{x}(s)+\ola{\xi}^{B,\varepsilon N}_x(s)\ora{\xi}^{B,\varepsilon N}_{x}(s)\\&\quad \quad \quad \quad \quad \quad \quad  \quad\quad \quad\quad \quad \quad \quad\quad \quad+\ola{\xi}^{A,\varepsilon N}_x(s)\ora{\xi}^{B,\varepsilon N}_{x}(s)+\ola{\xi}^{B,\varepsilon N}_x(s)\ora{\xi}^{A,\varepsilon N}_{x}(s)\Big)\\
&=-E\int_0^t \dd s \sum_{x\in\mathbb{T}_N}\nabla_N T_{-\frac{E}{3}N^{3/2}s}f\left(\tfrac{x}{N}\right)\left(	\ola{\xi}^{A,\varepsilon N}_x(s)+\ola{\xi}^{B,\varepsilon N}_x(s)\right)\left(\ora{\xi}^{B,\varepsilon N}_{x}(s)+\ora{\xi}^{A,\varepsilon N}_{x}(s)\right).
\end{aligned}
\end{equation}
Let us now define for $u,v\in\mathbb T$ the functions 
\begin{equation*}
\ora{i_\varepsilon}(u)(v)=\frac{1}{\varepsilon}\Id_{(u,u+\varepsilon]}(v),\quad\ola{i_\varepsilon}(u)(v)=\frac{1}{\varepsilon}\Id_{[u-\varepsilon,u)}(v)\,.
\end{equation*}
Note that for any $d\in\mathbb R$, we have that  $T_d
\ora{i_\varepsilon}(u-d)(v)=\ora{i_\varepsilon}(u)(v)$ and $T_d\ola{i_\varepsilon}(u-d)(v)=\ola{i_\varepsilon}(u)(v)$.

Since $\mcb{Z}^{N,-}_s $ has a velocity $v_-$ in its definition, we get that 
\begin{equation}\label{eq:closing_fields}
\begin{aligned}
\mcb{Z}^{N,-}_s\Big(\ora{i_\varepsilon}\Big(\tfrac{x+\frac E3 N^{3/2}s}{N}\Big)\Big)&=\frac{1}{\sqrt N}\sum_{y\in\mathbb T_N} T_{-\frac E3N^{3/2}s}\ora{i_\varepsilon}\Big(\tfrac{x+\frac E3 N^{3/2}s}{N}\Big)\Big(\tfrac yN\Big)\Big(\bar\xi_x^A(s)+\bar\xi_x^{B}(s)\Big)\\
&=\frac{1}{\sqrt N}\sum_{y\in\mathbb T_N} \ora{i_\varepsilon}\Big(\tfrac{x+\frac E3 N^{3/2}s}{N}\Big)\Big(\tfrac {y+\frac E3 N^{3/2}s}{N}\Big)\Big(\bar\xi_x^A(s)+\bar\xi_x^{B}(s)\Big)\\&=\frac{1}{\sqrt N}\sum_{y}\ora{i_\varepsilon}\Big(\tfrac{x}{N}\Big)\Big(\tfrac yN\Big)\Big(\bar\xi_x^A(s)+\bar\xi_x^{B}(s)\Big)\\&=\sqrt N\Big(\ora{\xi}^{A,\varepsilon N}_x(s)+\ora{\xi}^{B,\varepsilon N}_{x}(s)\Big)\,.
\end{aligned}
\end{equation}
The same identity above holds replacing $\ora{i_\varepsilon}$ by $\ola{i_\varepsilon}$
and $\ora{\xi}^{\alpha,\varepsilon N}_\cdot$ by $\ola{\xi}^{\alpha,\varepsilon N}_\cdot$, $\alpha\in\{A,B\}$.

From the last identities we rewrite \eqref{eq:DynkinZ1_diff_2} as
\begin{equation}\label{eq:martingaleZ3}
\begin{aligned}
\mcb  B_t^{N,-}(f)=-E\int_0^t \dd s{\frac 1N} \sum_{x\in\mathbb{T}_N}\nabla_NT_{-\frac E3 N^{3/2}s} f\left(\tfrac{x}{N}\right)\mcb{Z}^{N,-}_s\Big(\ola{i_\varepsilon}\Big(\tfrac{x+\frac E3 N^{3/2}s}{N}\Big)\Big)\mcb{Z}^{N,-}_s\Big(\ora{i_\varepsilon}\Big(\tfrac{x+\frac E3 N^{3/2}s}{N}\Big)\Big). 
\end{aligned}
\end{equation}
In conclusion, by the change of variables $z=x+\frac E3 N^{3/2}s$, we write $\mcb M^{N,-}(f)$ as
\begin{equation}\label{eq:martingaleZ4}
\begin{aligned}
\mcb M_t^{N,-}(f)=&{\mcb Z}^{N,-}_t(f)-{\mcb Z}^{N,-}_0(f)-\int_0^t \dd s \mcb Z_s^{N,-}(\Delta_N  f){+\mcb R_t^{N,\pm}(f)} \\
&-E\int_0^t \dd s\frac 1N \sum_{z\in\mathbb{T}_N}\nabla_N f\left(\tfrac{z}{N}\right)\mcb{Z}^{N,-}_s(\ola{i_\varepsilon}(\tfrac zN))\mcb{Z}^{N,-}_s(\ora{i_\varepsilon}(\tfrac zN))\,.
\end{aligned}
\end{equation}
This last expression suggests that, in the limit $N\to\infty$ and $\varepsilon\to0$, the martingale $\mcb M^{N,-}(f)$ 
converges to that in \eqref{eq:mart_prob} with $\lambda=E$. Hence, ${\mcb Z}^{N,-}$ displays KPZ behaviour {when $\gamma=1/2$}.

\subsubsection{The  field $\mcb Z_t^{N,+} $}\label{sec:case1_2}

{From~\eqref{eq:quadratic_term_gen}, we have
\begin{equation}\label{eq:quad_case1_2}
\begin{aligned}
\mcb B_t^{N,+}(f)&=-E N^{\frac12-\gamma}\int_0^t \dd s \sum_{x\in\mathbb{T}_N}\nabla_N T_{ \frac E 3 N^{3/2}s}f\left(\tfrac{x}{N}\right)\Big(\bar\xi^{A}_{x}(s)\bar\xi^{A}_{x+1}(s)-\bar\xi^{B}_{x}(s)\bar\xi^{B}_{x+1}(s)\Big)\\
&=-\frac{E}2 N^{\frac12-\gamma}\int_0^t \dd s \sum_{x\in\mathbb{T}_N}\nabla_N T_{ \frac E 3 N^{3/2}s}f\left(\tfrac{x}{N}\right)\Big\{\Big(\bar\xi^{A}_{x}(s)+\bar\xi^{B}_{x}(s)\Big)\Big(\bar\xi^{A}_{x+1}(s)-\bar\xi^{B}_{x+1}(s)\Big)\\
&\qquad\qquad\qquad\qquad\qquad\qquad\qquad\qquad\qquad +\Big(\bar\xi^{A}_{x}(s)-\bar\xi^{B}_{x}(s)\Big)\Big(\bar\xi^{A}_{x+1}(s)+\bar\xi^{B}_{x+1}(s)\Big)\Big\}\,.
\end{aligned}
\end{equation}
Once again, for $\gamma>1/2$, {$\mcb B_t^{N,+}(f)$} vanishes in view of the second-order Boltzmann-Gibbs principle {and for $\gamma>1/2$ this field has diffusive behaviour.}

For $\gamma=1/2$, we will need instead the second version of the principle, Theorem~\ref{thm:BG_v2}. 
To see how to apply it, 
we introduce the 
centred empirical measure $\bar\pi^{N,\alpha}_s$, $\alpha\in\{A,B,C\}$, which is defined as 
\begin{equation}\label{eq:emp_BG}
\bar\pi^{N,\alpha}_s(du)=\frac{1}{N}\sum_{x\in\mathbb{T}_N}[\xi^\alpha_x(s)-\rho_\alpha]\delta_{\frac xN}(du).
\end{equation}
Observe that this is nothing but the density field $\mcb Y^{N,\alpha}$ but in a different scaling regime. 

With this definition at hand, we have 
\begin{equation}\label{eq:rel_em_mea_ave}
\begin{aligned}
\langle\bar\pi^{N,\alpha}_s,\ola{\rho_\ell}(\cdot-\tfrac xN)\rangle=\frac{1}{{N}\ell}\sum_{y}\ola\rho\Big(\frac{y-x}{N}\Big)\bar\xi^\alpha_y(s),\quad \textrm{and}\quad
\langle\bar\pi^{N,\alpha}_s,\ora{\rho_\ell}(\cdot-\tfrac xN)\rangle=\frac{1}{{N}\ell}\sum_{y
}\ora\rho\Big(\frac{y-x}{N}\Big)\bar\xi^\alpha_y(s)
\end{aligned}
\end{equation}
where $\ola{\rho_\ell}$ and $\ora{\rho_\ell}$ are defined according to~\eqref{e:rhol} (see also~\eqref{eq:smooth}). 

For $x\in\Z$, $\alpha, \beta\in\{A,B\}$ and $\varepsilon>0$, 
Theorem~\ref{thm:BG_v2} allows us to replace $\bar\xi^{\alpha}_{x}\bar\xi^{\beta}_{x+1}$
by 
\begin{equation*}
\langle\bar\pi^{N,\alpha}_s,\ola{\rho_\varepsilon}(\cdot-\tfrac xN)\rangle  \langle\bar\pi^{N,\beta}_s,\ora{\rho_\varepsilon}(\cdot-\tfrac xN)\rangle\,,
\end{equation*} 
at a price whose second moment is $O(\varepsilon+N^{-1})$, and thus 
vanishes as $N\to+~\infty$ and $\varepsilon\to0$.
Therefore, modulo negligible terms, $\mcb B_t^{N,+}$ can be written as 
\[
\begin{aligned}
\mcb B^{N,+}_t(f)&=\frac E2\int_0^t \dd s \sum_{x\in\mathbb{T}_N}\nabla_N T_{ \frac E3 N^{3/2}s}f\left(\tfrac{x}{N}\right)\times\\
&\quad\quad\quad\times\Big\{\Big(\langle\bar\pi^{N,A}_s,\ola{\rho_\varepsilon}(\cdot-\tfrac xN)\rangle+\langle\bar\pi^{N,B}_s,\ola{\rho_\varepsilon}(\cdot-\tfrac xN)\rangle\Big)\Big(\langle\bar\pi^{N,A}_s,\ora{\rho_\varepsilon}(\cdot-\tfrac xN)\rangle-\langle\bar\pi^{N,B}_s,\ora{\rho_\varepsilon}(\cdot-\tfrac xN)\rangle\Big)\\
&\quad\quad\quad\quad+\Big(\langle\bar\pi^{N,A}_s,\ola{\rho_\varepsilon}(\cdot-\tfrac xN)\rangle-\langle\bar\pi^{N,B}_s,\ola{\rho_\varepsilon}(\cdot-\tfrac xN)\rangle\Big)\Big(\langle\bar\pi^{N,A}_s,\ora{\rho_\varepsilon}(\cdot-\tfrac xN)\rangle+\langle\bar\pi^{N,B}_s,\ora{\rho_\varepsilon}(\cdot-\tfrac xN)\rangle\Big)\Big\}\,.
\end{aligned}
\]
To identify the density fields in the previous display, let us focus on the first summand, 
and note that by a change of variables
\begin{equation*}
\begin{aligned}
&\sum_{x\in\mathbb{T}_N}\nabla_N T_{ \frac E3 N^{3/2}s}f\left(\tfrac{x}{N}\right)\Big(\langle\bar\pi^{N,A}_s,\ola{\rho_\varepsilon}(\cdot-\tfrac xN)\rangle+\langle\bar\pi^{N,B}_s,\ola{\rho_\varepsilon}(\cdot-\tfrac xN)\rangle\Big)\Big(\langle\bar\pi^{N,A}_s,\ora{\rho_\varepsilon}(\cdot-\tfrac xN)\rangle-\langle\bar\pi^{N,B}_s,\ora{\rho_\varepsilon}(\cdot-\tfrac xN)\rangle\Big)\\
&=\sum_{x\in\mathbb{T}_N}\nabla_N f\left(\tfrac{x}{N}\right)\bigg(\left\langle\bar\pi^{N,A}_s,\ola{\rho_\varepsilon}\left(\cdot-\tfrac{x+\frac E3 N^{3/2}s}{N}\right)\right\rangle+\left\langle\bar\pi^{N,B}_s,\ola{\rho_\varepsilon}\left(\cdot-\tfrac{x+\frac E3 N^{3/2}s}{N}\right)\right\rangle\bigg)\times\\
&\hspace{3cm}\times \bigg(\left\langle\bar\pi^{N,A}_s,\ora{\rho_\varepsilon}\left(\cdot-\tfrac{x+\frac E3 N^{3/2}s}{N}\right)\right\rangle-\left\langle\bar\pi^{N,B}_s,\ora{\rho_\varepsilon}\left(\cdot-\tfrac{x+\frac E3 N^{3/2}s}{N}\right)\right\rangle\bigg)\\
&=\frac1N\sum_{x\in\mathbb{T}_N}\nabla_N f\left(\tfrac{x}{N}\right){\mcb {Z}}^{N,-}_s\left(\ola{ \rho_{\varepsilon}}\left(\tfrac{x+\frac 23E N^{3/2}s}{N}\right)\right){\mcb {Z}}^{N,+}_s(\ora{ \rho_{\varepsilon}}(\tfrac xN))
\end{aligned}
\end{equation*}
where the last passage follows arguing as in~\eqref{eq:closing_fields} and using the fact that fields are 
defined with opposite velocities. 
Applying the same strategy for the second summand, we get
\begin{equation}\label{eq:DynkinZ2b}
\begin{aligned}
\mcb B^{N,+}_t(f)=\frac E2\int_0^t \dd s \frac1N\sum_{x\in\mathbb{T}_N}\nabla_N f\left(\tfrac{x}{N}\right)\Big\{&{\mcb {Z}}^{N,-}_s\left(\ola{ \rho_{\varepsilon}}\left(\tfrac{x+\frac 23E N^{3/2}s}{N}\right)\right){\mcb {Z}}^{N,+}_s(\ora{ \rho_{\varepsilon}}(\tfrac xN))\\
&+{\mcb {Z}}^{N,+}_s(\ola{ \rho_{\varepsilon}}(\tfrac xN)){\mcb {Z}}^{N,-}_s\left(\ora{ \rho_{\varepsilon}}\left(\tfrac{x+\frac 23E N^{3/2}s}{N}\right)\right)\Big\}\,.
\end{aligned}
\end{equation}	

\begin{remark}\label{rem:NewBG}
The reason why we need to consider the novel version of the second-order Boltzmann-Gibbs principle is to control terms 
as those in~\eqref{eq:DynkinZ2b}, i.e. terms displaying products of fields living in moving frames with different velocities. 
Indeed, in order to determine their limiting behaviour we will need to invoke Theorem~\ref{thm:Cross} 
which only applies if the fields are tested against sufficiently smooth functions 
(and this would not be the case if $\ora{\rho_\varepsilon}$ and $\ola{\rho_\varepsilon}$ were replaced by 
$\ora{i_\varepsilon}$ and $\ola{i_\varepsilon}$).
%
%
%
%
\end{remark}

In conclusion, we obtain 
\begin{equation}\label{eq:DynkinZ2bb}
\begin{aligned}
{\mcb M}_t^{N,+}(f)=&\mcb Z^{N,+}_t(f)-\mcb Z^{N,+}_0(f)-\int_0^t \dd s {\mcb  Z}_s^{N,+}(\Delta_Nf)\\
&-\frac E2\int_0^t \dd s \frac1N\sum_{x\in\mathbb{T}_N}\nabla_N f\left(\tfrac{x}{N}\right)\Big\{{\mcb {Z}}^{N,-}_s\left(\ola{ \rho_{\varepsilon}}\left(\tfrac{x+\frac 23E N^{3/2}s}{N}\right)\right){\mcb {Z}}^{N,+}_s(\ora{ \rho_{\varepsilon}}(\tfrac xN))\\
&\hspace{4.5cm}+{\mcb {Z}}^{N,+}_s(\ola{ \rho_{\varepsilon}}(\tfrac xN)){\mcb {Z}}^{N,-}_s\left(\ora{ \rho_{\varepsilon}}\left(\tfrac{x+\frac 23E N^{3/2}s}{N}\right)\right)\Big\}+\mcb R_t^{N,+}(f)\,.
\end{aligned}
\end{equation}	
Now, Theorem \ref{thm:Cross} with $\varphi_1(\cdot)=\varphi_2(\cdot)=\ora{\rho_\varepsilon}(\cdot)$ 
suggests that the last term vanishes in  $L^2(\mathbb P_{\nu_{\rho}})$ as $N\to\infty$, so that 
the martingale ${\mcb M}^{N,+}(f)$ should converge to that in \eqref{eq:mart_prob} with $\lambda=0$. 
Therefore this field has a diffusive behaviour {(for any $\gamma>1/2$)}. 

}

\subsection{Case (II): $E_{B}-E_{A}=E_{C}-E_{A}=E$}\label{sec:case2}

{As previously observed in Remark~\ref{rem:(I)(II)}, 
this case can be obtained from case \textbf{(I)} by a simple trasformation of the fields, so we omit the details and present only the final formulation of the martingale problems.}
{Recall~\eqref{e:Case2}, i.e. 
\begin{equation*}
\begin{aligned}
\mcb Z_t^{N,+}(f)&=\mcb Y^{N,A}_t(T_{\frac E3 N^{3/2}t}f)+2\mcb Y^{N,B}_t(T_{\frac E3 N^{3/2}t}f)\,, \\
\mcb Z_t^{N,-}(f)&=\mcb Y^{N,A}_t(T_{-\frac E3 N^{3/2}t}f)\,.
\end{aligned}
\end{equation*}
}
For the first field, when $\gamma=1/2$, we have for $f\in\mcb D(\mathbb T)$,
{
\begin{equation*}
\mcb  M_t^{N,-}(f)={\mcb Z}^{N,-}_t(f)-{\mcb Z}^{N,-}_0(f)-\int_0^t \dd s \mcb Z_s^{N,-}(\Delta_N  f)+E\int_0^t \dd s \frac 1N \sum_{z\in\mathbb{T}_N}\nabla_N f\left(\tfrac{z}{N}\right)\left(\mcb{Z}^{N,-}_s(\ora{i_\varepsilon}(\tfrac zN))\right)^2.
\end{equation*} 
So we expect that any limit point $\mcb Z^{-}_t$ of the sequence  $\{\mcb Z^{N,-}_t\}_{N}$ 
satisfies~\eqref{eq:mart_prob} with $\lambda_-=E$ {and so in this case, i.e. $\gamma=1/2$, the field has KPZ behaviour.} 
As before, for $\gamma>1/2$, any crossed term in the occupation variables vanishes 
in the $N\to\infty$ limit and we observe diffusive behaviour. 
}
{
For the second field, we proceed exactly as in Section~\ref{sec:case1_2} and get to the following expression for the associated martingale
\begin{equation}\label{eq:Mart_II+}
\begin{aligned}
&\mcb M^{N,+}_t(f)=\mcb Z^{N,+}_t(f)-\mcb Z^{N,+}_0(f)-\int_0^t \dd s \mcb Z^{N,+}_s(\Delta_Nf)\\
&\hspace{2cm}+\frac E2\int_0^t \dd s \frac 1N\sum_{z\in\mathbb{T}_N}\nabla_Nf\left(\tfrac{z}{N}\right)\Big\{\mcb {Z}^{N,-}_s\Big(\ola{\rho_{\varepsilon}}\Big(\frac{z+\frac 23 E N^{3/2}s}{N}\Big)\Big){\mcb {Z}}^{N,+}_s(\ora{ \rho_{\varepsilon}}(\tfrac zN))\\
&\hspace{5cm}+{\mcb {Z}}^{N,+}_s(\ola{ \rho_{\varepsilon}}(\tfrac zN))\mcb {Z}^{N,-}_s\Big(\ora{\rho_{\varepsilon}}\Big(\frac{z+\frac 23 E N^{3/2}s}{N}\Big)\Big)\Big\}+\mcb R_t^{N,+}(f).
\end{aligned}
\end{equation}
{Theorem \ref{thm:Cross} implies that, as $N\to\infty$, the last term only contributes via the limiting mass of 
${\mcb {Z}}^{N,-}$},
which suggests that~\eqref{eq:Mart_II+} will converge to~\eqref{eq:mart_prob} with $\lambda=0$.
}

\subsection{Case (III): $E_{C}-E_{A}\neq E_{C}-E_{B}\neq 0$}\label{sec:case3}


In this case, we consider the generic fields given in \eqref{eq:gen_fields_n}. 
Let us start with $\mcb {Z}^{N,-}_t(f)$ and then we analyze $\mcb {Z}^{N,+}_t(f)$, 
keeping in mind that 
both modes are predicted to have KPZ behaviour. 
To make the presentation simpler, in this subsection we make
the choice $E_C=0$, as one can recover the general case by replacing every instance of $E_A$ and $E_B$ below 
with $E_A-E_C$ and $E_B-E_C$, respectively.

\subsubsection{The field $\mcb {Z}^{N,-}_t(f)$ }

Plugging the values $D_2=D_2^-=c_+/E_A$, $D_1=1$ and $v=v_-=-\tfrac\delta{2N^\gamma}$, for $\delta$ in~\eqref{e:delta}, 
into the formula~\eqref{dynkin_case_gen_n} 
and ignoring both the terms that are negligible in the limit as $N\to\infty$ and 
those killed with the choice of constants and velocities, we see that for any $f\in \mcb D(\mathbb T)$ we have 
\begin{equation}\label{dynkin_case_gen_n_nn}
	\begin{aligned}
		\mcb M_t^{N,-}(f)=&{\mcb Z}^{N,-}_t(f)-{\mcb Z}^{N,-}_0(f)-\int_0^t \dd s \mcb Z_s^{N,-}(\Delta_Nf)\\
		&+N^{1/2-\gamma}E_A\int_0^t \dd s \sum_{x\in\mathbb{T}_N}\nabla_N T_{-\frac\delta2 N^{2-\gamma}s}f\left(\tfrac{x}{N}\right)\bar\xi^A_{x}(s)\bar\xi^A_{x+1}(s)\\
		&+N^{1/2-\gamma}\frac{c_+E_{B}}{E_A}\int_0^t \dd s \sum_{x\in\mathbb{T}_N}\nabla_N T_{-\frac\delta2 N^{2-\gamma} s}f\left(\tfrac{x}{N}\right)\bar\xi^{B}_{x}(s)\bar\xi^{B}_{x+1}(s)\\
		&+N^{1/2-\gamma}\frac{E_B+c_+}{2}\int_0^t \dd s \sum_{x\in\mathbb{T}_N}\nabla_N T_{-\frac\delta2 N^{2-\gamma} s}f\left(\tfrac{x}{N}\right)\Big(\bar\xi^{A}_{x}(s)\bar\xi^{B}_{x+1}(s)+\bar\xi^{B}_{x}(s)\bar\xi^{A}_{x+1}(s)\Big).
	\end{aligned}
\end{equation}
The second order Boltzmann-Gibbs principle says that the last three lines in the last display vanish as $N\to\infty$ for $\gamma>1/2$, {and in this case the field has diffusive behaviour.}
When $\gamma=1/2$, in order  to close the equations in terms of the fields $\mcb {Z}^{N,-}_s$ and  $\mcb {Z}^{N,+}_s$ 
we first note that {the last three lines in last display} can be written as
\begin{equation}\label{eq:gen_field}
	\begin{aligned}
		&\int_0^t \dd s \sum_{x\in\mathbb{T}_N}\nabla_N T_{-\frac\delta2 N^{3/2}s}f\left(\tfrac{x}{N}\right)\Big\{E_A\bar\xi^A_{x}(s)\bar\xi^A_{x+1}(s)+ \frac{E_A+\tfrac32\delta}{2}\Big(\bar{\xi}^{A}_{x}(s)\bar{\xi}^{B}_{x+1}(s)+\bar{\xi}^{B}_{x}(s)\bar{\xi}^{A}_{x+1}(s)\Big)\\ &\quad\quad\quad\quad\quad\quad\quad\quad\quad\quad\quad\quad\quad\quad\quad\quad\quad\quad\quad +\frac{c_+E_{B}}{E_A}\bar\xi^{B}_{x}(s)\bar\xi^{B}_{x+1}(s)\Big\}\,.
	\end{aligned}
\end{equation} 
Let us look at the term in {parentheses.} Note that
\begin{equation}\label{e:linsys}
\begin{aligned}
\hspace{-.5cm}	E_A\bar\xi^A_{x}\bar\xi^A_{x+1}+ \frac{E_A+\tfrac32\delta}{2}\Big(\bar{\xi}^{A}_{x}\bar{\xi}^{B}_{x+1}+\bar{\xi}^{B}_{x}\bar{\xi}^{A}_{x+1}\Big)+\frac{c_+E_{B}}{E_A}\bar\xi^{B}_{x}\bar\xi^{B}_{x+1}&=\mathfrak c_1  \Big(\bar\xi_x^A+\frac{c_+}{E_A}\bar\xi_{x}^{B}\Big) \Big( \bar\xi_{x+1}^A+\frac{c_+}{E_A}\bar\xi_{x+1}^{B}\Big)\\
&+\mathfrak c_2 \Big( \bar\xi_x^A+\frac{c_+}{E_A}\bar\xi_{x}^{B}\Big) \Big( \bar\xi_{x+1}^A+\frac{c_-}{E_A}\bar\xi_{x+1}^{B}\Big)\\
&+ \mathfrak c_3 \Big( \bar\xi_x^A+\frac{c_-}{E_A}\bar\xi_{x}^{B}\Big)\Big( \bar\xi_{x+1}^A+\frac{c_+}{E_A}\bar\xi_{x+1}^{B}\Big)
\end{aligned}
\end{equation}
for $\mathfrak c_2=\mathfrak c_3=(E_A-\mathfrak c_1)/2$ and 
\begin{equation*}\begin{aligned}
\mathfrak c_1= \frac{E_A}{3\delta}(2E_{B}-E_{A}+\tfrac32\delta)=\frac{E_A}{3\delta}(E_B-c_-)
\end{aligned}
\end{equation*}
where $c_\pm$ and $\delta$ are defined in~\eqref{assumption} and~\eqref{e:delta}. 
From this we conclude that 
\eqref{eq:gen_field} can be rewritten as

\begin{equation*}
\begin{aligned}
&\int_0^t \dd s \sum_{x\in\mathbb{T}_N}\nabla_N T_{v_- N^2s}f\left(\tfrac{x}{N}\right)\Big\{\frac{E_A}{3\delta}(E_B-c_-)
\Big(\bar\xi_x^A+\frac{c_+}{E_A}\bar\xi_{x}^{B}\Big) \Big( \bar\xi_{x+1}^A+\frac{c_+}{E_A}\bar\xi_{x+1}^{B}\Big)\\&\quad\quad\quad\quad\quad\quad\quad\quad\quad\quad\quad\quad\quad\quad+
E_A\Big(\frac{1}2-\frac{E_B-c_-}{6\delta}\Big)
\Big( \bar\xi_x^A+\frac{c_+}{E_A}\bar\xi_{x}^{B}\Big) \Big( \bar\xi_{x+1}^A+\frac{c_-}{E_A}\bar\xi_{x+1}^{B}\Big)
\\&\quad\quad\quad\quad\quad\quad\quad\quad\quad\quad\quad\quad\quad\quad+ 	E_A\Big(\frac{1}2-\frac{E_B-c_-}{6\delta}\Big) \Big( \bar\xi_x^A+\frac{c_-}{E_A}\bar\xi_{x}^{B}\Big)\Big( \bar\xi_{x+1}^A+\frac{c_+}{E_A}\bar\xi_{x+1}^{B}\Big)
\Big\}.
\end{aligned}
\end{equation*}
Now, in the first line we apply the standard version of the second order Boltzmann--Gibbs principle, and  in the second and  third we use the  new formulation proposed in Theorem~\ref{thm:BG_v2}, so that overall the previous display equals 
\begin{equation*}
\begin{aligned}
&\int_0^t \dd s \sum_{x\in\mathbb{T}_N}\nabla_N T_{v_- N^2s}f\left(\tfrac{x}{N}\right)\Big\{{\frac{E_A}{3\delta}(E_B-c_-)}
\mcb Z_s^{N,-}(\overleftarrow{\iota_\varepsilon}(\tfrac xN))\mcb Z_s^{N,-}(\overrightarrow{\iota_\varepsilon}(\tfrac xN))\\&\quad\quad\quad\quad\quad\quad\quad\quad\quad\quad\quad\quad\quad\quad+
E_A\Big(\frac{1}2-\frac{E_B-c_-}{6\delta}\Big)
\mcb Z_s^{N,-}(\overleftarrow{\rho_\varepsilon}(\tfrac xN)) \mcb Z_s^{N,+}\Big(\overrightarrow{\rho_\varepsilon}\Big(\frac{x-2\frac E3N^{3/2}s}{N}\Big)\Big) 
\\&\quad\quad\quad\quad\quad\quad\quad\quad\quad\quad\quad\quad\quad\quad+ 	E_A\Big(\frac{1}2-\frac{E_B-c_-}{6\delta}\Big)\mcb Z_s^{N,+}\Big(\overleftarrow{\rho_\varepsilon}\Big(\frac{x-2\frac E3N^{3/2}s}{N}\Big)\Big) \mcb Z_s^{N,-}\Big(\overrightarrow{\rho_\varepsilon}\Big(\tfrac{x}{N}\Big)\Big) 
\Big\}.
\end{aligned}
\end{equation*}
As in the previous cases, when $N\to\infty$, the terms in the last two lines of the last display 
{only produce a massive term} 
and the only term that survives is the first, so that $\mcb {Z}^{N,-}_t(f)$ has KPZ behaviour {(for $\gamma=1/2$)}.

\subsubsection{The field $\mcb {Z}^{N,+}_t(f)$ }

Dynkin's formula for the field $\mcb {Z}^{N,+}_t$ is identical to the expression in~\eqref{dynkin_case_gen_n_nn}, 
with $c_-$ instead of $c_+$ and $\delta$ replaced by $-\delta$. For $\gamma>1/2$ we know that terms in $\mcb B_t^{N,+}$ vanish as $N\to\infty$ while for $\gamma=1/2$ they have a non-trivial limit. As above, we want to write the current in terms of the fluctuation fields 
$\mcb {Z}^{N,-}_s$ and  $\mcb {Z}^{N,+}_s$. To do so, we solve the same linear system as in~\eqref{e:linsys} 
but in which we exchange $c_+$ and $c_-$. This leads us to 
\begin{equation}
\begin{aligned}
&\mcb B_t^{N,+}(f)=\int_0^t \dd s \sum_{x\in\mathbb{T}_N}\nabla_N T_{v N^2s}f\left(\tfrac{x}{N}\right)\Big\{\frac{E_A}{3\delta}(c_+-E_B)
\Big( \bar\xi_x^A+\frac{c_-}{E_A}\bar\xi_{x}^{B}\Big) \Big( \bar\xi_{x+1}^A+\frac{c_-}{E_A}\bar\xi_{x+1}^{B}\Big)\\&\quad\quad\quad\quad\quad\quad\quad\quad\quad\quad\quad\quad\quad\quad+
E_A\Big(\frac 12-\frac{c_+-E_B}{6\delta}\Big)
\Big(\bar\xi_x^A+\frac{c_+}{E_A}\bar\xi_{x}^{B}\Big) \Big( \bar\xi_{x+1}^A+\frac{c_-}{E_A}\bar\xi_{x+1}^{B}\Big)
\\&\quad\quad\quad\quad\quad\quad\quad\quad\quad\quad\quad\quad\quad\quad+ 	E_A\Big(\frac 12-\frac{c_+-E_B}{6\delta}\Big) \Big(\bar\xi_x^A+\frac{c_-}{E_A}\bar\xi_{x}^{B}\Big)\Big(\bar\xi_{x+1}^A+\frac{c_+}{E_A}\bar\xi_{x+1}^{B}\Big)
\Big\}.
\end{aligned}
\end{equation}
Arguing exactly as above, we ultimately get
\begin{equation}\label{eq:Z2Z1}
\begin{aligned}
&\mcb B_t^{N,+}(f)=\int_0^t \dd s \sum_{x\in\mathbb{T}_N}\nabla_N T_{v N^2s}f\left(\tfrac{x}{N}\right)\Big\{\frac{E_A}{3\delta}(c_+-E_B)
\mcb Z_s^{N,+}(\overleftarrow{\iota_\varepsilon}(\tfrac xN))\mcb Z_s^{N,+}(\overrightarrow{\iota_\varepsilon}(\tfrac xN))\\&\quad\quad\quad\quad\quad\quad\quad\quad\quad\quad\quad\quad\quad\quad+
E_A\Big(\frac 12-\frac{c_+-E_B}{6\delta}\Big)
\mcb Z_s^{N,+}(\overleftarrow{\rho_\varepsilon}(\tfrac xN)) \mcb Z_s^{N,-}\Big(\overrightarrow{\rho_\varepsilon}\Big(\frac{x-2\frac E3N^{3/2}s}{N}\Big)\Big) 
\\&\quad\quad\quad\quad\quad\quad\quad\quad\quad\quad\quad\quad\quad\quad+ 	E_A\Big(\frac 12-\frac{c_+-E_B}{6\delta}\Big) \mcb Z_s^{N,-}\Big(\overleftarrow{\rho_\varepsilon}\Big(\frac{x-2\frac E3N^{3/2}s}{N}\Big)\Big) \mcb Z_s^{N,+}\Big(\overrightarrow{\rho_\varepsilon}\Big(\tfrac{x}{N}\Big)\Big) 
\Big\},
\end{aligned}
\end{equation}
which suggests that also this field has KPZ behaviour {for $\gamma=1/2$.}

\section{Proof of Theorem~\ref{thm:main_new}}\label{sec:proof_main_theo}

{In this section, we prove the main result of the paper. 
We will first show tightness of the fluctuation fields and then uniquely characterise the limit points by verifying 
they satisfy either Definition~\ref{def:energy_solution} or~\ref{def:OU}. }

\subsection{Tightness}\label{sec:tightness}

{In this section we prove not only tightness of the sequences of processes $\{\mcb Z^{N,\pm}_t,t\in[0,T]\}_{N\in\N}$ with respect to the uniform topology of $D([0,T],\mcb D'(\mathbb T))$, in each of the cases {\bf (I)}-{\bf(III)} above, 
but further show some properties the fluctuation fields enjoy and that will be necessary to control 
terms containing their cross product. 
The main result of the section is the following proposition. 

\begin{prop}\label{prop:Tightness}
In any of the cases {\bf\textrm{(I)}-\bf\textrm{(III)}} in Section~\ref{sec:flucfields}, 
the sequences of processes $\{\mcb Z^{N,\pm}_t)\,:\,t\in[0,T]\}_{N\in\mathbb N}$ are tight in $D([0,T]; \mcb D'(\mathbb T))$. 

Furthermore, there exists a constant $C>0$ such that for any $f\in \mcb D(\T)$
and $N\in\N$, the following bounds hold 
\begin{align}
&\E_{\nu_\rho}\left[\sup_{s\leq T} |\mcb{Z}^{N,\pm}_s(f)|^2\right]\lesssim \Big(\|\Delta f\|_2^2\vee\|\nabla f\|_2^2+\frac{1}{N^3}\|f\|_{\infty}^2\Big)\,,\label{e:supZ}\\
&\E_{\nu_\rho}\left[\Big(\mcb{Z}^{N,\pm}_t(f)-\mcb{Z}^{N,\pm}_s(f)\Big)^2\right]\lesssim (t-s)\Big(\|\Delta f\|_2^2\vee\|\nabla f\|_2^2\Big)\,,\qquad \textrm{for all} \,\,0\leq s<t\leq T,\label{e:HolZ}
\end{align}
where $\|\cdot\|_2$ denotes the usual $L^2(\T)$-norm.
\end{prop}


\begin{proof}\label{sec:tightZ}
By Mitoma's criterion~\cite{mitoma}, tightness for $\{\mcb Z^{N,\pm}_t,t\in[0,T]\}_{N\in\N}$ follows 
upon showing that, for any $f\in\mcb D (\T)$, the sequences $\{\mcb Z^{N,\pm}_t(f),t\in[0,T]\}_{N\in\N}$ are tight. 

Now, notice that, in each of the cases {\bf (I)}-{\bf(III)} of the previous section, 
the fields $\mcb Z^{N,\pm}(f)$ can be written, in their most general form, 
as in~\eqref{dynkin_case_gen_n} where $a,b,D_1,D_2$ and $v$ are chosen in such a way 
that the fourth, fifth and eighth lines sum up to $0$. 
In other words, we look at the generic field $\mcb Z^{N}_t(f)$ in~\eqref{eq_gen_field}, 
which is given by 
\begin{equation}\label{e:generic}
\begin{aligned}
{\mcb Z}^{N}_t(f)&={\mcb Z}^{N}_0(f)+\mcb I_t^{N}(f)+\mcb B_t^{N}(f)+\mcb R_t^{N}(f)+{\mcb M}^{N}_t(f)
\end{aligned}
\end{equation}
where ${\mcb M}^{N}_t(f)$ is the martingale whose quadratic variation is given in~\eqref{eq:QV}, 
\begin{align}
\mcb I^{N}_t(f)&:=\int_0^t \dd s \mcb Z^{N}_s(\Delta_N  f)\label{eq:INt}\\
\mcb B^{N}_t(f)&:=N^{\frac12-\gamma}\int_0^t \dd s \sum_{x\in\mathbb{T}_N}\nabla_N T_{vN^2 s}f\left(\tfrac{x}{N}\right)\Big(\mathfrak C_1\bar\xi^A_{x}(s)\bar\xi^A_{x+1}(s)+\mathfrak C_2\bar\xi^{B}_{x}(s)\bar\xi^{B}_{x+1}(s)\nonumber\\
&\quad \quad \quad\quad \quad \quad \quad \quad \quad  \quad\quad \quad\quad \quad \quad \quad\quad \quad+\mathfrak C_3(\bar\xi^{A}_{x}(s)\bar\xi^{B}_{x+1}(s)+\bar\xi^{B}_{x}(s)\bar\xi^{A}_{x+1}(s))\Big)\,.\label{eq:KNt}
\end{align}
for suitable choices of the constants $\mathfrak C_i$, $i=1,2,3$, and $\mcb R_t^{N}(f)$ consists of the 
sixth and seventh lines in~\eqref{dynkin_case_gen_n}, so that its second moment goes to $0$ 
as $N\to\infty$. 
We will separately analyse the terms on the right-hand side of~\eqref{e:generic} and show that each of them is 
tight and satisfies~\eqref{e:supZ} and~\eqref{e:HolZ}. 
\medskip 

The process $\{\mcb I^{N}_t,t\in[0,T]\}_{N\in\N}$ can be explicitly written as 
\begin{equation}\label{e:IHol}
\begin{aligned}
{\mcb I}^{N}_{t}(f)=\int_{0}^{t}\dd s\frac{1}{\sqrt N}\sum_{x\in\mathbb T_N}\Big(D_1\bar{\xi}^{A}_x(s)+D_2\bar{\xi}^{B}_x(s)\Big)\Delta_N T_{vN^{2}s}f\left(\tfrac xN\right)\,.
\end{aligned}
\end{equation}
We consider its increment and apply the Cauchy--Schwarz inequality to the time integral thus obtaining
\begin{equation}\label{e:IHol}
\begin{aligned}
\E_{\nu_\rho}\left[({\mcb I}^{N}_{t}(f)-{\mcb I}^{N}_s(f))^2\right]&=\E_{\nu_\rho}\Bigg[\Bigg(\int_{s}^{t}\dd r\frac{1}{\sqrt N}\sum_{x\in\mathbb T_N}\Big(D_1\bar{\xi}^{A}_x(r)+D_2\bar{\xi}^{B}_x(r)\Big)\Delta_N T_{vN^2r}f\left(\tfrac xN\right)\Bigg)^2\Bigg]\\
&\leq  (t-s)\int_s^t\dd r\E_{\nu_\rho}\Bigg[\Bigg(\frac{1}{\sqrt N}\sum_{x\in\mathbb T_N}\Big(D_1\bar{\xi}^{A}_x(r)+D_2\bar{\xi}^{B}_x(r)\Big)\Delta_N T_{vN^2r}f\left(\tfrac xN\right)\Bigg)^2\Bigg]\\
&= (t-s)\int_s^t\dd r\frac{1}{N}\sum_{x\in\mathbb T_N}\left[\Delta_N T_{vN^{2 }r}f\left(\tfrac xN\right)\right]^2\E_{\nu_\rho}\left[\left(D_1\bar{\xi}^{A}_x+D_2\xi^{B}_x\right)^2\right]\\
& \lesssim   (t-s)^2 \|\Delta f\|^2_2\,,
\end{aligned}
\end{equation}
where we used the fact that  the product measure $\nu_\rho$ is stationary 
and that the expected value in the second to last line is bounded by a constant. 

The previous bound guarantees, on the one hand that~\eqref{e:HolZ} holds, and on the other, 
by the Kolmogorov criterion, that the sequence of processes 
$\{\mcb I^{N}_t,t\in[0,T]\}_{N\in\N}$ is tight with respect to
the uniform topology of $\mathcal C([0,T],\R)$, satisfies~\eqref{e:supZ} 
and any limit point has $\alpha$-H\"older-continuous trajectories for any $\alpha<1/2$,
\medskip 

We now turn to the sequence of quadratic terms $\{\mcb B^{N}_t,t\in[0,T]\}_{N\in\N}$ in~\eqref{eq:KNt}.
It suffices to prove tightness for each of the summands, which are all of the form 
\begin{align}\label{eq:quad_term}
\mcb K^{N,\alpha,\beta}_t(f)&:=N^{\frac12-\gamma}\int_0^t \dd s \sum_{x\in\mathbb{T}_N}\nabla_N T_{vN^2 s}f\left(\frac{x}{N}\right)\bar\xi^\alpha_{x}(s)\bar\xi^\beta_{x+1}(s)\,
\end{align}
where $\alpha,\beta\in\{A,B\}$. We present the proof only in the case $\alpha=\beta$, 
being it the hardest since it involves an extra bound on the variance of the occupation variables. 
As $\alpha=\beta$ and $\alpha$ will be fixed throughout, we shorten the notation 
by omitting the corresponding index from $\mcb K^{N,\alpha,\beta}=\mcb K^{N}$. 

In order to estimate the second moment of the increments of $\mcb K^{N}(f)$, 
we will use the second order Boltzmann--Gibbs principle, 
which allows to replace the single-species degree-two term $\bar\xi^\alpha_x\bar\xi^\alpha_{x+1}$ 
with the local function of the configuration $(\ora{\xi}^{\alpha,L}_x)^2$, for some $L$ to be determined. 
If we sum and subtract in the integral in \eqref{eq:quad_term} $\big(\ora{\xi}^{\alpha,L}_x(sN^a)\big)^2+\frac{\chi_\alpha}{L}$, with $\chi_\alpha={\rm Var}(\xi^\alpha_x)$ and use the convex  inequality $(x+y)^2\leq 2x^2+2y^2$, then,  for any $0\leq s<t\leq T$ it holds 
\begin{equation}\label{eq:Ktight}
\begin{aligned}
\E_{\nu_\rho}&\left[({\mcb K}^{N}_{t}(f)-{\mcb K}^{N}_{s}(f))^2\right]\\
\leq& 2N^{1-2\gamma}\E_{\nu_\rho}\Bigg[\Bigg(\int_s^t\dd r\sum_{x\in\mathbb T_N}\nabla_NT_{v N^2 r}f\left(\tfrac xN\right)\Big\{\bar\xi^{\alpha}_{x}(r)\bar\xi^{\alpha}_{x+1}(r)-\big(\ora{\xi}^{\alpha,L}_x(r)\big)^2+\frac{\chi_\alpha}{L}\Big\}\Bigg)^2\Bigg]\\
&+2N^{1-2\gamma}\E_{\nu_\rho}\Bigg[\Bigg(\int_s^t\dd r\sum_{x\in\mathbb T_N}\nabla_NT_{v N^2 r}f\left(\tfrac xN\right)\Big\{\big(\ora{\xi}^{\alpha,L}_x(r)\big)^2-\frac{\chi_\alpha}{L}\Big\}\Bigg)^2\Bigg].
\end{aligned}
\end{equation}
Theorem~\ref{thm:BG} implies that the first term is bounded above by 
\begin{equation}\label{eq:Ktight1}
\begin{aligned}
N^{1-2\gamma}\left\{\frac{L}{N}+\frac{(t-s)N}{L^2}\right\}\int_s^t\dd r\frac1N&\sum_{x\in\mathbb T_N}\left[\nabla_NT_{v N^2r}f\left(\tfrac xN\right)\right]^2\lesssim N^{1-2\gamma}\left\{\frac{(t-s)L}{N}+\frac{(t-s)^2N}{L^2}\right\}\|\nabla f\|^2_2\,,
\end{aligned}
\end{equation}
while, from the  Cauchy--Schwarz inequality, the second term can be bounded above by 
\begin{equation}\label{eq:Ktight2}
\begin{aligned}
(t-s)N^{1-2\gamma}\int_s^t \dd r\sum_{x,y\in\mathbb T_N}&\nabla_NT_{vN^2r}f\left(\tfrac xN\right)\nabla_NT_{vN^2r}f\left(\tfrac yN\right)\times\\
&\times\E_{\nu_\rho}\Bigg[\left(\left(\ora{\xi}^{\alpha,L}_x(r)\right)^2-\frac{\chi_\alpha}{L}\right)\left(\left(\ora{\xi}^{\alpha,L}_y(r)\right)^2-\frac{\chi_\alpha}{L}\right)\Bigg]\,.
\end{aligned}
\end{equation}
Note that for $|x-y|>L$ the variables $\ora{\xi}^{\alpha,L}_x(r)$ and $\ora{\xi}^{\alpha,L}_y(r)$ are uncorrelated,  and 
therefore we can restrict the double sum to $|x-y|\leq L$. We then apply again the convex inequality 
$2xy\leq x^2+y^2$, to estimate the previous display by a constant times
\begin{equation}\label{eq:Ktight3}
\begin{aligned}
(t-s)L&N^{1-2\gamma}\int_s^t \dd r\sum_{x\in\mathbb T_N}\Big(\nabla_NT_{vN^2 r}f\left(\tfrac xN\right)\Big)^2\E_{\nu_\rho}\Bigg[\left(\left(\ora{\xi}^{\alpha,L}_x(r)\right)^2-\frac{\chi_\alpha}{L}\right)^2\Bigg]
\leq \frac{(t-s)^2}{L}N^{2-2\gamma} \|\nabla f\|^2_{2}\,.
\end{aligned}
\end{equation}
Putting together~\eqref{eq:Ktight1} and~\eqref{eq:Ktight2}, and choosing $L=\sqrt{(t-s)} N^{2\gamma}$, we get
\begin{equation}\label{eq:Ktight_last}
\begin{aligned}
&\E_{\nu_\rho}\left[({\mcb K}^{N}_{t}(f)-{\mcb K}^{N}_{s}(f))^2\right]\lesssim (t-s)^{3/2}\|\nabla f\|^2_{2}.
\end{aligned}
\end{equation}
Since the previous computations can be  derived for any $\alpha, \beta\in\{A,B\}$, we conclude that 
\begin{equation}\label{eq:HolB_q}
\begin{aligned}
\E_{\nu_\rho}\left[({\mcb B}^{N}_{t}(f)-{\mcb B}^{N}_{s}(f))^2\right]\lesssim (t-s)^{3/2}\|\nabla f\|^2_{2}\,.
\end{aligned}
\end{equation}
We can now argue as for $\{\mcb I^{N}_t,t\in[0,T]\}_{N\in\N}$ to conclude that tightness,~\eqref{e:supZ} 
and~\eqref{e:HolZ} hold for $\{\mcb B^{N}_t\colon t\in[0,T]\}_{N\in\N}$. 
\medskip

At last, we consider the sequence of martingales $\{{\mcb M}^{N}(f)\}_N$. We will prove a much stronger statement, 
namely that the sequence indeed converges (so that in particular is tight) and that 
the limit is a Brownian motion with an explicit covariance. To do so, we exploit~\cite[Theorem VIII, 3.12]{lib} 
that we state below for the reader's convenience. 

\begin{thm}
	\label{Shir}
	Let $\{M^N\}_{N \in\mathbb N}$ be a sequence of martingales belonging to the space $ D ([0,T]; \mathbb R)$ and denote by $\langle M^N\rangle$ the quadratic variation of $M^N$, for any $N\in\mathbb N$. 
	Let $c\colon [0,T]\to[0,\infty)$ be a deterministic continuous function. Assume that
	\begin{enumerate}
		\item for any $N\in\mathbb N$, the quadratic variation process $\langle M^N\rangle_t$ has continuous trajectories almost surely,
		\item the following limit holds
		\begin{equation}
			\lim_{N \rightarrow \infty}\mathbb{E}_{\nu_{\rho}}\Big[\sup_{0\leq s \leq T}\Big|M_s^N-M_{s^-}^N\Big|\Big]=0,
			\label{maxjump}
		\end{equation}
		\item for any $t \in [0,T]$ the sequence of random variables $\{\langle M^N\rangle_t\}_{N> 1}$ converges in probability to  $c(t)$. 
	\end{enumerate}
	Then, the sequence $\{M^N\}_{N}$ converges in law in $D([0,T]; \R)$, as $N$ goes to infinity,  to a mean zero Gaussian process $M$ which is a martingale on $[0,T]$ with continuous trajectories and whose quadratic variation is given by $c$.
\end{thm}

Point 1 in the previous statement is trivially satisfied by our sequence $\{{\mcb M}^{N}(f)\}_N$ 
as the quadratic variation in~\eqref{eq:QV} is clearly continuous. Now, for point 2, note that
\begin{equation}\label{eq:sup_mart}
\sup_{0\leq s\leq T}\Big|\mcb M_s^{N}(f)-\mcb M_{s^-}^{N}(f)\Big|
=\sup_{0\leq s\leq T}\Big|\mcb Z_s^{N}(f)-\mcb Z_{s^-}^{N}(f)\Big|\,.
\end{equation}
To bound the right-hand side, we consider a particle exchange happening at time $s$ on the bond $\{x,x+1\}$. 
If at time $s^-$, we observe the pair $(A,B)$, then at time $s$ we will have the pair $(B,A)$. Therefore, 
at site $x$,
\[
\Big(D_1\bar{\xi}^{A}_x(s)+D_2\bar{\xi}^{B}_x(s)\Big)-\Big(D_1\bar{\xi}^{A}_x(s^-)+D_2\bar{\xi}^{B}_x(s^-)\Big)=D_2-D_1\,,
\]
while at site $x+1$,
\[
\Big(D_1\bar{\xi}^{A}_{x+1}(s)+D_2\bar{\xi}^{B}_{x+1}(s)\Big)-\Big(D_1\bar{\xi}^{A}_{x+1}(s^-)+D_2\bar{\xi}^{B}_{x+1}(s^-)\Big)=D_1-D_2\,.
\]
From this we get 
\begin{equation*}
\Big|\mcb Z_s^{N,\pm}(f)-\mcb Z_{s^-}^{N,\pm}(f)\Big|=\frac{1}{\sqrt N} \Big(f(\tfrac xN)-f(\tfrac {x+1}{N})\Big)(D_2-D_1)\leq \frac{1}{N^{3/2}}\|\nabla  f\|_{\infty}
\end{equation*}
and, since the right-hand side does not depend on $s$, point 2 trivially follows. 
For point 3, we first state a crucial lemma whose proof is provided at the end of the section. 

\begin{lem}\label{lem:conv_QV}
For any $f\in\mcb D(\T)$, it holds 
\begin{equation*}
	\begin{aligned}
		\lim_{N\to\infty}\mathbb E_{\nu_\rho}\left[\left(\langle\mcb M^{N}(f)\rangle_t- \mathbb E_{\nu_\rho}[\langle\mcb M^{N}(f)\rangle_t]\right)^2\right]=0.
	\end{aligned}
\end{equation*}
\end{lem}

The previous statement, together with Lemma~\ref{lemmaqv} completes the verification of the 
assumptions of Theorem~\ref{Shir}. Hence, the sequence $\{{\mcb M}^{N}(f)\}_N$ is tight and converges in 
law to the Gaussian process whose quadratic variation is given by the right hand side of~\eqref{eq:lim_QV}. 

It remains to prove the validity of~\eqref{e:supZ} and~\eqref{e:HolZ} for $\{{\mcb M}^{N}(f)\}_N$. 
For the first, by Burkholder--Davis--Gundy inequality, \eqref{eq:lim_QV} and  \eqref{eq:sup_mart} we obtain 
\begin{equation}\label{e:SupM}
\begin{aligned}
\E_{\nu_\rho}\left[\sup_{t\leq T}\Big({\mcb M}^{N}_{t}(f)\Big)^2\right]\leq&\E_{\nu_\rho}\left[\langle{\mcb M}^{N}(f)\rangle_T\right] +
\E_{\nu_\rho}\left[\sup_{t\leq T}\Big|{\mcb M}^{N}_t(f)-{\mcb M}^{N}_{t^-}(f)\Big|^2\right] \\
\lesssim& T \|\nabla f\|^2_2 + N^{-3} \|\nabla f\|_\infty^2\,.
\end{aligned}
\end{equation}
For the latter instead, recalling the expression of the quadratic variation of the martingale given in \eqref{eq:QV}, 
we see that for any $0\leq s<t\leq T$, we have 
\begin{equation}\label{e:HolM}
\begin{aligned}
&\E_{\nu_\rho}\left[({\mcb M}^{N}_{t}(f)-{\mcb M}^{N}_{s}(f))^2\right]\\
&=\E_{\nu_\rho}\Bigg[\int_s^t \dd r\frac1N \sum_x c_x(\eta)(\nabla_N T_{v N^2r}f(\tfrac xN))^2\big[(D_1\xi^{A}_{x+1}(r)+D_2\xi^{B}_{x+1}(r))-(D_1\xi^{A}_{x}(r)+D_2\xi^{B}_{x}(r))\big]^2\Bigg]\\
&\lesssim   (t-s) \|\nabla f\|^2_2\,.
\end{aligned}
\end{equation}

To summarise, each of the terms $\mcb I^{N}(f)$, $\mcb B^{N}(f)$, $\mcb R^{N}(f)$ and $\mcb M^{N}(f)$ is tight, 
which implies tightness for $\mcb Z^{N}(f)$. Moreover,~\eqref{e:IHol},~\eqref{e:HolM}, and~\eqref{eq:HolB_q} 
give~\eqref{e:HolZ}, while~\eqref{e:supZ} follows by~\eqref{e:SupM} and the uniform (in $N$) continuity 
of $\mcb I^{N}(f)$, $\mcb B^{N}(f)$. 
\end{proof}

\begin{proof}[Proof of Lemma~\ref{lem:conv_QV}]
Note that
the expectation in the statement of the lemma is equal to
\begin{equation*}
	\begin{aligned}
		\mathbb E_{\nu_\rho}\bigg[\bigg(\int_0^t \dd s\frac1N\sum_x (\nabla_N T_{v N^2s}f(\tfrac xN))^2\Big[c_x(\eta)\Big\{\Big(D_1\xi^{A}_{x+1}+D_2\xi^{B}_{x+1})&-(D_1\xi^{A}_{x}+D_2\xi^{B}_{x})\Big\}^2\\
		&-\frac{4}{9}(D_1^2+D_2^2-D_1D_2)\Big]\bigg)^2\bigg].
	\end{aligned}
\end{equation*}
From the computations {in the proof of Lemma \ref{lemmaqv}}  we know that
\begin{equation}\label{eq:manyterms}
	\begin{aligned}
	c_x(\eta)\Big\{\Big(D_1\xi^{A}_{x+1}+D_2\xi^{B}_{x+1})-(D_1\xi^{A}_{x}+&D_2\xi^{B}_{x})\Big\}^2-\frac{4}{9}(D_1^2+D_2^2-D_1D_2)\\
	=&D_1^2\Big\{\sum_{\alpha}c^{\alpha, A}\xi^{\alpha}_{x}\xi^{A}_{x+1}+\sum_\beta c^{A, \beta}\xi^{A}_{x}\xi^{\beta}_{x+1}-\frac 49\Big\}\\
	&+
D_2^2\Big\{\sum_{\alpha}c^{\alpha, B}\xi^{\alpha}_{x}\xi^{B}_{x+1}+\sum_\beta c^{B,\beta}\xi^{B}_{x}\xi^{\beta}_{x+1}-\frac 49\Big\}\\&-2D_1D_2\Big\{c^{B,A}\xi^B_x\xi^A_{x+1}+c^{A,B}\xi_x^A\xi_{x+1}^B-\frac 29\Big\}.
\end{aligned}
\end{equation}
The proof of the lemma is complete once we show that the second moment of the time integral of each 
of the previous summands goes to $0$ as $N\to\infty$. As they can all be analysed similarly, we 
limit ourselves to treat the first. 
By the Cauchy-Schwarz inequality, we have
\begin{equation}\label{eq:exp_mart_conv_m}
	\begin{aligned}
		&\mathbb E_{\nu_\rho}\bigg[\bigg(\int_0^t \dd s\frac1N\sum_x (\nabla_N T_{v N^2s}f(\tfrac xN))^2\Big\{\sum_{\alpha}c^{\alpha, A}\xi^{\alpha}_{x}\xi^{A}_{x+1}+\sum_\beta c^{A, \beta}\xi^{A}_{x}\xi^{\beta}_{x+1}-\frac 49\Big\}\bigg)^2\bigg]\\
		&\leq\frac{t^2}{N^2}\sum_{x,y}\nabla_N f(\tfrac xN)^2\nabla_N f(\tfrac yN)^2\times\\
		&\qquad\times
		\mathbb E_{\nu_\rho}\bigg[\Big\{\sum_{\alpha}c^{\alpha, A}\xi^{\alpha}_{x}\xi^{A}_{x+1}+\sum_\beta c^{A, \beta}\xi^{A}_{x}\xi^{\beta}_{x+1}-\frac 49\Big\}\Big\{\sum_{\alpha}c^{\alpha, A}\xi^{\alpha}_{y}\xi^{A}_{y+1}+\sum_\beta c^{A, \beta}\xi^{A}_{y}\xi^{\beta}_{y+1}-\frac 49\Big\}\bigg]\,.
	\end{aligned}
\end{equation}
Note now that for $|x-y|>1$ the variables $\xi^{\alpha}_{x}\xi^{\beta}_{x+1}$ and 
$\xi^{\alpha'}_{y}\xi^{\beta'}_{y+1}$ are independent for any value 
of $\alpha,\alpha',\beta,\beta'\in\{A,B\}$. Since they are also centred, the expectation 
at the right hand side is $0$. 
%
When instead $x\in\{y,y+1\}$ the expectation can be easily seen to be bounded 
so that overall \eqref{eq:exp_mart_conv_m} is of order $N^{-1}$ and therefore vanishes in the limit. 
%
%
%
\end{proof}

\subsection{Characterisation of the limit points}\label{sec:char_limit}

From the  results in the previous section, 
we know that the sequences $\{ \mcb{Z}^{N,\pm} \}_{ N \in \mathbb{N} } $, $\{ \mcb{I}^{N,\pm} \}_{ N \in \mathbb{N} } $, 
$\{ \mcb {B}^{N,\pm} \}_{ N \in \mathbb{N} } $, $\{ \mcb {M}^{N,\pm} \}_{ N \in \mathbb{N} } $ and 
$\{ \mcb{R}^{N,\pm} \}_{ N \in \mathbb{N} } $, the latter vanishing,  
are tight with respect to the uniform topology on $ D([0,T]; \mcb D'(\mathbb{T}))$.
Prohorov's theorem ensures the existence of a converging subsequence that, abusing notation, we will still denote by $N$. 
Let $\mcb{Z}^\pm$, $\mcb{I}^{\pm}$, $\mcb {B}^{\pm} $ and $ \mcb{M}^{\pm}$ be 
their almost surely continuous limits.  

We are left to show that $\mcb{Z}^\pm$ are either solutions to the Ornstein--Uhlenbeck equation or 
stationary energy solutions of the stochastic Burgers  equation in the sense of Definitions~\ref{def:OU} and \ref{def:energy_solution}. Once this is proven, since any of these solutions is unique by Theorem~\ref{thm:Uniqueness}, 
the convergence of the whole sequence follows. 
\medskip

We begin with item (i) in Definition \ref{def:energy_solution}. 
Note that, by computing the characteristic function at any time $t$ of the field $\mcb Z_t^{N,\pm}$ 
we obtain that $\mcb Z_t^{N,\pm}$ converges in distribution to a spatial white noise of variance  
\begin{equation}
\textrm{Var}(D_1\xi_x^A+D_2\xi_x^{B})=\frac 29(D_1^2+D_2^2-D_1D_2)\,,
\end{equation}
where the previous equality can be verified by arguing as in the proof of Lemma~\ref{lemmaqv}. 
{Above for a random variable $X$, we denoted by $\textrm{Var}(X)$ its variance with respect to $\nu_\rho$.}

Now, we turn to items (ii)-(iii). 
The initial field $\{ \mcb Z^{N,\pm}_0\}_{N\in\mathbb N}$ can be treated as in the previous point. 
The martingale terms $\{ \mcb{M}^{N,\pm}\}_{ N \in \mathbb{N} }$ were thoroughly analysed in the 
previous section, where it was shown that $\mcb{M}^{\pm} $ is a mean-zero Gaussian process 
with quadratic variation given by the right-hand side of~\eqref{eq:lim_QV}. 
For the linear term $\mcb I^\pm$, recall that
$$\mcb I_t^{N,\pm}(f)=\int_0^t \dd s \mcb Z_s^{N,\pm}(\Delta_Nf)\,.$$
Since $\{\mcb Z^{N,\pm}\}_N$ converges in the space $ C([0,T]; \mcb D'(\mathbb{T}))$, the limit 
must be such that for any $f\in \mcb D(\mathbb T)$ 
$$\mcb I_t^{\pm}(f)=\int_0^t \dd s \mcb Z_s^{\pm}(\Delta f)\,.$$
At this point, it remains to study the quadratic term $\mcb B^\pm$, which is the most delicate. 
Recall \eqref{eq:quadratic_term_gen}. 

In all the cases that we distinguished in Section~\ref{sec:martn}, 
from the second-order Boltzmann-Gibbs principles, Theorems \ref{thm:BG} and~\ref{thm:BG_v2}, 
\eqref{eq:quadratic_term_gen} vanishes in $L^2(\mathbb P_{\nu_\rho})$ if $\gamma>1/2$, 
while for $\gamma=1/2$ it can be written as the sum of terms that can be either of the form
 \begin{equation}\label{eq:quadraticKPZ}
	\begin{aligned}
		&\lambda_\pm\int_0^t \dd s \frac1N\sum_{x\in\mathbb{T}_N}\nabla_N f\left(\tfrac{x}{N}\right)
		\mcb Z_s^{N,\pm}({\iota^1_\varepsilon}(\tfrac xN))\mcb Z_s^{N,\pm}({\iota^2_\varepsilon}(\tfrac {x}N))
	\end{aligned}
\end{equation}
where $\lambda_\pm$ is a constant and $\iota^i_\varepsilon \in\{\ola{\iota_\varepsilon}, \ora{\iota_\varepsilon}\}$, $i=1,2$, 
or 
\begin{equation}\label{eq:crossVanish}
\mathfrak A^\pm\int_0^t \dd s  \frac1N\sum_{x\in\mathbb{T}_N}\nabla_N f\left(\tfrac{x}{N}\right)
		\mcb Z_s^{N,\pm}({\rho^1_\varepsilon}(\tfrac xN)) \mcb Z_s^{N,\mp}\Big({\rho^2_\varepsilon}\left(\tfrac{x+2v_\pm N^{2}s}{N}\right)\Big) 
\end{equation}
for some constants $\mathfrak A^\pm$, where 
$\rho^i_\varepsilon \in\{\ola{\rho_\varepsilon}, \ora{\rho_\varepsilon}\}$, $i=1,2$. 
To control this latter integral we { would like to} apply Theorem~\ref{thm:Cross} {but, 
while condition~\ref{itm:2} clearly holds with $\alpha=1/2$ as can be seen 
by a simple application of Cauchy-Schwarz inequality together with~\eqref{e:supZ} and~\eqref{e:HolZ}, 
condition~\ref{itm:1}. is violated for the points $k_1=k\in\Z\setminus\{0\}$ and $k_2=0$. 
Indeed, for these values of $k_1,k_2$ we have $k_1+k_2=k\neq 0$, 
$\widehat{\rho^1_\varepsilon}(k_1)\widehat{\rho^2_\varepsilon}(0)\neq 0$ but 
$v_1^Nk_1+v_2^N k_2= 0 \times k +2v_\pm N^{2} \times 0=0$. This is the place at which 
the linear transport term in~\eqref{eq:SBE_main} comes about. 
To see this, note first that, by definition, $\widehat{\rho^2_\varepsilon}(0)=\int \rho^2_\varepsilon(u)\dd u=1$ and 
\begin{equation}\label{eq:Mass}
\mcb Z_s^{N,\mp}(1)=\mcb Z_0^{N,\mp}(1)=\overline{\mcb Z}^{N,\mp}
\end{equation}
where we used that the total mass is conserved and denoted the time-independent quantity $\overline{\mcb Z}^{N,\mp}$.
Hence, we rewrite~\eqref{eq:crossVanish} as
\begin{equation}\label{eq:crossVanish2}
\begin{aligned}
\mathfrak A^\pm\int_0^t \dd s  \frac1N\sum_{x\in\mathbb{T}_N}\nabla_N &f\left(\tfrac{x}{N}\right)
		\mcb Z_s^{N,\pm}({\rho^1_\varepsilon}(\tfrac xN)) \mcb Z_s^{N,\mp}\Big({\rho^2_\varepsilon}\left(\tfrac{x+2v_\pm N^{2}s}{N}\right)-1\Big) \\
		&+\mathfrak A^\pm \overline{\mcb Z}^{N,\mp}\int_0^t \dd s  \frac1N\sum_{x\in\mathbb{T}_N}\nabla_N f\left(\tfrac{x}{N}\right)
		\mcb Z_s^{N,\pm}({\rho^1_\varepsilon}(\tfrac xN))\,. 
\end{aligned}
\end{equation}
We can now apply Theorem~\ref{thm:Cross} which ensures that, 
as $\varepsilon>0$ is fixed, the first summand vanishes in the limit $N\to\infty$. 
The second instead is the product of two terms which are tight (it is {\it not} the product of fields anymore!) 
and therefore it converges, in the double limit $N\to\infty$ and $\eps\to 0$, to 
\begin{equation}\label{e:LimitMass}
\mathfrak A^\pm \overline{\mcb Z}^{\mp}\int_0^t \dd s \mcb Z_s^{\pm}(\nabla f)\,.
\end{equation}}

We turn to~\eqref{eq:quadraticKPZ}, for which, by tightness of $\{\mcb B^{N,\pm}\}_{N\in\mathbb N}$, 
we get
 \begin{equation}\label{Beps}
	\begin{aligned}
		\mcb B^{\varepsilon,\pm}_t(f)=\lim_{N\to\infty} \int_0^t \dd s  \frac1N\sum_{x\in\mathbb{T}_N}\nabla_N f\left(\tfrac{x}{N}\right)
		\mcb Z_s^{N,\pm}({\iota^1_\varepsilon}(\tfrac xN))\mcb Z_s^{N,\pm}({\iota^2_\varepsilon}(\tfrac {x}N))\,.
	\end{aligned}
\end{equation}
The process at the left hand side coincides with that in \eqref{eq:q_end}. 
Indeed, even if neither $\overleftarrow{\iota_\varepsilon}$ nor $\overrightarrow{\iota_\varepsilon}$ 
belong to $\mcb D(\mathbb T)$, they can be approximated by elements in the latter space 
(see also the footnote before~\eqref{eq:q_end}). 
For further details, we refer the reader to \cite[Section~5.3]{GJ14}.
Now, by computations similar to those performed in the proof of~\eqref{eq:HolB_q}, it is not hard to see that 
for any $0<\delta\leq\varepsilon<1$, we have 
\begin{equation*}
\mathbb{E}_n \bigg[\bigg|
(\mcb {B}^{N,\pm}_{t}(f) 
- \mcb{B}^{N,\pm}_{s}(f))
- (\mcb B^{\delta,\pm}_{t}(f)-\mcb B^{\delta,\pm}_{s}(f))\bigg|^2 \bigg] \lesssim \bigg((t-s)\varepsilon  
+ \frac{(t-s)^2}{\varepsilon^2 N}  \bigg)
\| \nabla f\|^2_{2} 
\end{equation*}
so that, by~\eqref{Beps} and taking the limit in $N$, we obtain
\begin{equation}
\label{eq:energy_condition_estimate}
\mathbb{E} \big[\big| 
(\mcb B^{\varepsilon,\pm}_{t}(f)-\mcb B^{\varepsilon,\pm}_{s}(f))
- (\mcb B^{\delta,\pm}_{t}(f)-\mcb B^{\delta,\pm}_{s}(f))
\big|^2 \big] 
\lesssim  \varepsilon(t-s)
\| \nabla f \|^2_{2} 
\end{equation}
which, on the one hand guarantees that $\mcb B^{\varepsilon,\pm}$ has unique limit in $\varepsilon\to 0$, and on the other 
provides the energy estimate~\eqref{eq:energy_estimate}. 
Hence, both (ii) and (iii) of Definition \ref{def:energy_solution} hold. 

At last, we note that all the arguments above hold {\it mutatis mutandis} for the reversed process 
$\{ \mcb {Z}^{N,\pm}_{T-t} : t \in [0,T]\}$ whose dynamics is generated by the adjoint operator $L_N^*$, 
so that also item (iv) is satisfied. 
%
}

At last, we prove that $\mcb M^+$ and $\mcb M^-$ are independent, which amounts to verify that 
for all $f,g\in\mcb D(\mathbb T)$, $(\mcb M^+_t(f))_t $ and $(\mcb M^-_t(g))_t$ are independent. 
This in turn follows once we show that the cross-variation $\langle \mcb M^+(f), \mcb M^-(g)\rangle_t=0.$
By polarization, for any $N\geq 1$, we have
\[
2\langle \mcb M^{N,+}(f), \mcb M^{N,-}(g)\rangle_t=\langle \mcb M^{N,+}(f)+ \mcb M^{N,-}(g)\rangle_t-\langle \mcb M^{N,+}(f)\rangle_t-\langle \mcb M^{N,-}(g)\rangle_t\,,
\]
and we will prove
$$\lim_{N\to+\infty}\mathbb E_{\nu_\rho}\Big[\Big(\langle \mcb M^{N,+}(f)+ \mcb M^{N,-}(g)\rangle_t-\langle \mcb M^{N,+}(f)\rangle_t-\langle \mcb M^{N,-}(g)\rangle_t\Big)^2\Big]=0,$$ 
which completes the proof.

A simple computation similar to that in \eqref{eq:QV} shows that
\begin{equation}\label{eq:QV_of_sum}
	\begin{aligned}
		\langle \mcb M^{N,+}(f)+ \mcb M^{N,-}(g)\rangle_t=  & N^{a-3}\int_0^t \sum_x c_x(\eta)(\nabla_N T_{v_+ N^2s}f(\tfrac xN))^2\big[\xi^{A}_{x+1}-\xi^{A}_{x}+D_2^+ (\xi^{B}_{x+1}-\xi^{B}_{x})\big]^2\dd s\\
		&+ N^{a-3}\int_0^t \sum_x c_x(\eta)(\nabla_N T_{v_- N^2s}g(\tfrac xN))^2\big[\xi^{A}_{x+1}-\xi^{A}_{x}+D_2^- (\xi^{B}_{x+1}-\xi^{B}_{x})\big]^2\dd s\\
		&+ N^{a-3}\int_0^t \sum_x c_x(\eta)\nabla_N T_{v_+ N^2s}f(\tfrac xN)\nabla_N T_{v_- N^2s}g(\tfrac xN)\times\\&\qquad\qquad\times 2\big[ \xi^{A}_{x+1}-\xi^{A}_{x}+D_2^+ (\xi^{B}_{x+1}-\xi^{B}_{x})\big]\big[ \xi^{A}_{x+1}-\xi^{A}_{x}+D_2^- (\xi^{B}_{x+1}-\xi^{B}_{x})\big]\dd s\,.
	\end{aligned}
\end{equation}
Now, the first two terms on the right-hand side equal $\langle \mcb M^{N,+}(f)\rangle_t$ and $\langle \mcb M^{N,-}(g)\rangle_t$, respectively, so that we only need to show that the variance of the last vanishes as $N\to+\infty$. 
To this end, we first note that its mean is zero since
\begin{equation}\label{eq:last}
	\begin{aligned}
	\mathbb E_{\nu_\rho}\Big[\big( \xi^{A}_{x+1}-\xi^{A}_{x}+D_2^+ (\xi^{B}_{x+1}-\xi^{B}_{x})\big)\big( \xi^{A}_{x+1}-\xi^{A}_{x}+D_2^- (\xi^{B}_{x+1}-\xi^{B}_{x})\big)\Big]=\frac{2}{9} (2-(D_2^++D_2^-)+2D_2^+D_2^-)=0
	\end{aligned}
\end{equation}
where the last equality holds provided $D_2^+, D_2^-$ are chosen as in cases \textbf{(I-III)}. 
Here, we are crucially relying on the choice of constants dictated by mode coupling theory. 
Turning to the second moment, we perform computations similar to those 
in the proof of Lemma \ref{lem:conv_QV}. By the Cauchy-Schwarz inequality, for $a=2$, we have
\begin{equation*}
	\begin{aligned}
&\mathbb E_{\nu_\rho}\Big[\Big(N^{a-3}\int_0^t \sum_x c_x(\eta)\nabla_N T_{v_+ N^2s}f(\tfrac xN)\nabla_N T_{v_- N^2s}g(\tfrac xN)\times\\
&\quad \quad \quad \quad \quad \quad \quad \quad\times\big[ \xi^{A}_{x+1}-\xi^{A}_{x}+D_2^+ (\xi^{B}_{x+1}-\xi^{B}_{x})\big]\big[ \xi^{A}_{x+1}-\xi^{A}_{x}+D_2^- (\xi^{B}_{x+1}-\xi^{B}_{x})\big]\dd s\Big)^2\Big]\\
&\leq \frac{t}{N^2}\int_0^t\dd s \sum_{x,y} \nabla_N T_{v_+ N^2s}f(\tfrac xN)\nabla_N T_{v_- N^2s}g(\tfrac xN)\nabla_N T_{v_+ N^2s}f(\tfrac yN)\nabla_N T_{v_- N^2s}g(\tfrac yN)\times\\
&\quad \quad \quad \quad\quad \quad \quad \quad\times\mathbb E_{\nu_\rho}\Big[c_x(\eta)c_y(\eta)\big( \xi^{A}_{x+1}-\xi^{A}_{x}+D_2^+ (\xi^{B}_{x+1}-\xi^{B}_{x})\big)\big( \xi^{A}_{x+1}-\xi^{A}_{x}+D_2^- (\xi^{B}_{x+1}-\xi^{B}_{x})\big)\times\\
&\quad \quad \quad \quad\quad \quad \quad \quad\quad\quad\times\big( \xi^{A}_{y+1}-\xi^{A}_{y}+D_2^+ (\xi^{B}_{y+1}-\xi^{B}_{y})\big)\big( \xi^{A}_{y+1}-\xi^{A}_{y}+D_2^- (\xi^{B}_{y+1}-\xi^{B}_{y})\big)\Big]\\
	\end{aligned}
\end{equation*}
and, arguing as in~\eqref{eq:manyterms}-\eqref{eq:exp_mart_conv_m}, since the variables are independent and centred, 
the terms in the last sum that survive are only those for which $|x-y|\leq 1$. Therefore, 
the last term vanishes as $N\to+\infty$.

Finally we prove that for every $t\in[0,T]$ the processes $\mcb Z^{+}_t$ and $\mcb Z^{-}_t$  are uncorrelated. To this end we note that for any $N$ and for any $f,g\in\mcb D(\mathbb T)$, we have that  $\mcb Z^{N,+}_t(f)$ and $\mcb Z^{N,-}_t(g)$  are uncorrelated i.e.  
\begin{equation}\label{eq:unc}
\mathbb E_{\nu_\rho}[\mcb Z^{N,+}_t(f)\mcb Z^{N,-}_t(g)]=0,
\end{equation}
and this is a consequence of the fact that for any $x,y\in\mathbb T_N$
\begin{equation}
\mathbb E_{\nu_\rho}[(\bar\xi^A_x+D_2^+\bar\xi^B_x)(\bar\xi^A_y+D_2^-\bar\xi^B_y)]=0
\end{equation}
for any choice of $D_2^\pm$ in all the cases \textbf{(I-III)}.  This is an easy computation that is similar to the one in \eqref{eq:last}. Again we note that in {the} last result it is crucial to use the choice of the constants fixing the fields that we obtained from mode coupling theory.  
With this the proof of Theorem~\ref{thm:main_new} is completed.

 
%
%

\section{Crossed fields}\label{sec:cross}

The goal of this section is to show that those integrals displaying the product of fluctuation fields evolving on 
different time frames, vanish in the large $N$ limit. We state the main result in wider generality as we believe 
its scope goes beyond the model treated in the present work, so that the theorem below might be of independent interest. 

\begin{thm}\label{thm:Cross}
Let $T>0$, $n\in\N$ and, for every $N\in\N$, let $\{\mcb X^{i,N}\colon i=1,\dots,n\}$ 
be $n$ time-dependent random fields taking values in $D([0,T];\mcb D'(\T))$ {and defined} on the same probability space.
Let $\{v_i^N\}_N$, $i=1,\dots,n$ be $n$ sequences of constants 
{ and $\varphi_i$, $i=1,\dots,n$, be smooth functions on $\T$.}

We make the following assumptions 
\begin{enumerate}[noitemsep]
\item\label{itm:1} the constants $\{v_i^N\}_N$ {and the functions $\varphi_i$ are such that if for $k_1,\dots,k_n\in\Z$ with  
$\sum_{i=1}^n k_i\neq 0$, $\Pi_i\hat\varphi_i(k_i)\neq 0$, then}
\begin{equation}\label{e:seqConst}
\lim_{N\to\infty} \frac1N\left|\sum_{i=1}^n k_i v_i^N\right|=\infty\,,
\end{equation}
\item\label{itm:2} there exist $\alpha\in(0,1)$, and a constant $C>0$ such that 
for any $f_1,\dots,f_n\in \mcb D(\T)$, the following bounds hold uniformly in $N$
\begin{align}
&\E\bigg[\sup_{s\leq T} |\mcb P^N_s|\bigg]\leq C \prod_{i=1}^n\|\Delta f_i\|_2 \vee\|f_i\|_\infty\,,\label{e:supP}\\
&\E\Big[|\mcb P^N_t-\mcb P^N_s|\Big]\leq C (t-s)^{\alpha} \prod_{i=1}^n\|\Delta f_i\|_2\,,\qquad\text{for all $0\leq s\leq t\leq T$,}\label{e:HolP}
\end{align}
where, for $s\in[0,T]$, we denoted by $\mcb P^N_s:=\prod_{i=1}^n\mcb X^{i,N}_s(f_i)$. 
\end{enumerate}

Then, for any $t\in[0,T]$ and $f\in \mcb D(\T)$, we have  
\begin{equation}\label{e:CrossLim}
\lim_{N\to\infty}\E\left[\left|\int_0^t\dd s\frac 1N\sum_{z\in\mathbb{T}_N}\nabla_Nf\left(\tfrac{z}{N}\right)\prod_{i=1}^n\mcb X^{i,N}_s \left(\varphi_i\left(\tfrac{z+v_i^N s}{N}\right)\right)\right|\right]=0\,. 
\end{equation}
\end{thm}

The proof of the previous result relies on a version of {the} Riemann-Lebesgue lemma which is stated in Appendix~\ref{a:RL}. 
To be able to use it, we need to introduce some notation. 
Let $\phi$ be a real-valued function on $\T$. We define the {\it discrete Fourier transform} of $\phi$ as 
\begin{equation}\label{e:DFT}
	\cF^N\phi(k):=\frac{1}{N}\sum_{x\in\T_N}\phi(\tfrac xN)e_{-k}(\tfrac xN)\,,\qquad \text{for $k\in\N$, }
\end{equation}
where $e_k(x):=e^{2\pi\iota k x}$, for $x\in\T$, and $\iota:=\sqrt{-1}$, the complex unit. 
The usual properties of the Fourier transform hold in the discrete context as well. For the reader's convenience, 
we recall them here without proof. The following inversion holds
\begin{equation}\label{e:InverseDFT}
	\phi(\tfrac xN):=\sum_{k=0}^{N-1}\cF^N\phi(k)e_{k}(\tfrac xN)\,,\qquad \text{for $x\in\T_N$, }
\end{equation}
For any $a\in \R$ the discrete Fourier transform of $\phi(a)(\cdot):=\phi(\cdot-a)$ is given by 
		\begin{equation}\label{e:Translate}
			\cF^N[\phi(a)](k)=\cF^N\phi(k) e_{-k}\left(\frac{\lfloor aN\rfloor}{N}\right)\,,\qquad \text{for all $k\in\N$,}
		\end{equation}
		where $\lfloor a\rfloor$ denotes the integer part of $a$. 
Plancherel's identity holds, i.e. 
		\begin{equation}\label{e:Plancherel}
			\frac{1}{N}\sum_{x\in\T_N}\phi(\tfrac xN)\psi(\tfrac xN)=\sum_{k=0}^{N-1}\cF^N\phi(k)\cF^N\psi(-k)\,.
		\end{equation}
If $\phi\in L^2(\T)$, then for every $k\in\N$
		\begin{equation}\label{e:ConvDFTft}
			\lim_{N\to\infty} \cF^N\phi(k)=\hat\phi(k):=\int_\T \phi(x) e_{-k}(x)\dd x\,.
		\end{equation}
We are now ready to prove the result. 

\begin{proof}[Proof of Theorem~\ref{thm:Cross}]
At first, we want to write the integrand in~\eqref{e:CrossLim}, in terms of its discrete Fourier transform. 
To do so, notice that, for $i=1,\dots,n$ and any $w\in\T_N$, applying first~\eqref{e:InverseDFT} 
and then~\eqref{e:Translate} 
to $\varphi_i(\tfrac{w}{N})(\cdot)$, we get 
\begin{equation*}
{\mcb X}^{i,N}_s\Big(\varphi_i(\tfrac wN)\Big)=\sum_{k=0}^{N-1}{\mcb X}^{i,N}_s(e_{-k}) \mcb F^N[ \varphi_i(\tfrac wN)](-k)=\sum_{k=0}^{N-1}{\mcb X}^{i,N}_s(e_{-k}) \mcb F^N \varphi_i^\eps(-k)e_k(\tfrac wN)\,. 
\end{equation*} 
Hence, using the above and~\eqref{e:Plancherel}, 
we see that the integrand in~\eqref{e:CrossLim} satisfies
\begin{align*}
\frac 1N\sum_{z\in\mathbb{T}_N}&\nabla_Nf\left(\tfrac{z}{N}\right)\prod_{i=1}^n\mcb X^{i,N}_s \left(\varphi_i\left(\tfrac{z+v_i^N s}{N}\right)\right)\\
&=\sum_{k_1,\dots,k_n=0}^{N-1}\mcb F^N(\nabla_Nf)(-k_{[1:n]})\bigg(\prod_{i=1}^n \cF^N \varphi_i(-k_i){\mcb X}^{i,N}_s(e_{-k_i})\bigg) e^{\iota 2\pi \sum_{i=1}^n k_i \frac{\lfloor v_i^N s\rfloor}{N}}\\
&=\sum_{k_1,\dots,k_n=0}^{N-1}F^N(k_1,\dots,k_n,k_{[1:n]})  \bigg(\prod_{i=1}^n {\mcb X}^{i,N}_s(e_{-k_i})\bigg) e^{\iota 2\pi \sum_{i=1}^n k_i \frac{\lfloor v_i^N s\rfloor}{N}}\,,
\end{align*}
where we introduced the shorthand notation $k_{[1:n]}:=\sum_{i=1}^n k_i$ and defined $F^N$ according to
\begin{equation}\label{e:FN}
F^N(k_1,\dots,k_n,k_{[1:n]}):=\mcb F^N(\nabla_Nf)(-k_{[1:n]})\bigg(\prod_{i=1}^n \cF^N \varphi_i(-k_i)\bigg).
\end{equation}
Now, the term we need to control is given by 
\begin{align}
\E&\left[\left|\int_0^t\dd s\frac 1N\sum_{z\in\mathbb{T}_N}\nabla_Nf\left(\tfrac{z}{N}\right)\prod_{i=1}^n\mcb X^{i,N}_s \left(\varphi_i\left(\tfrac{z+v_i^N s}{N}\right)\right)\right|\right]\nonumber\\
&=\E\bigg[\bigg|\int_0^t\dd s \sum_{k_1,\dots,k_n=0}^{N-1}F^N(k_1,\dots,k_n,k_{[1:n]})  \bigg(\prod_{i=1}^n {\mcb X}^{i,N}_s(e_{-k_i})\bigg) e^{\iota 2\pi \sum_{i=1}^n k_i \frac{\lfloor v_i^N s\rfloor}{N}}\bigg|\bigg]\nonumber\\
&= \E\bigg[\bigg| \sum_{k_1,\dots,k_n=0}^{N-1}F^N(k_1,\dots,k_n,k_{[1:n]}) \int_0^t \bigg(\prod_{i=1}^n {\mcb X}^{i,N}_s(e_{-k_i})\bigg) e^{\iota 2\pi \sum_{i=1}^n k_i \frac{\lfloor v_i^N s\rfloor}{N}}\dd s\bigg|\bigg]\nonumber\\
&\leq \sum_{k_1,\dots,k_n=0}^{\infty}|F^N(k_1,\dots,k_n,k_{[1:n]}) |\E\bigg[\bigg| \int_0^t \bigg(\prod_{i=1}^n {\mcb X}^{i,N}_s(e_{-k_i})\bigg) e^{\iota 2\pi \sum_{i=1}^n k_i \frac{\lfloor v_i^N s\rfloor}{N}}\dd s\bigg|\bigg]\,.\label{e:sum}
\end{align}
At this point, we want to take the limit in $N$ at both sides of the previous expression. 
Note first that by~\eqref{e:supP} and in view of~\eqref{e:ConvDFTft}, we have the bound
\begin{align}
|F^N(k_1,\dots,k_n,k_{[1:n]})& |\E\bigg[\bigg| \int_0^t \bigg(\prod_{i=1}^n {\mcb X}^{i,N}_s(e_{-k_i})\bigg) e^{\iota 2\pi \sum_{i=1}^n k_i \frac{\lfloor v_i^N s\rfloor}{N}}\dd s\bigg|\nonumber\\
&\lesssim t |k_{[1:n]}||\hat f(k_{[1:n]})\bigg(\prod_{i=1}^n|\hat\varphi_i(k_i)|\bigg)\E\bigg[\sup_{s\leq T}\bigg| \prod_{i=1}^n {\mcb X}^{i,N}_s(e_{-k_i})\bigg|\bigg]\nonumber\\
&\lesssim t |k_{[1:n]}||\hat f(k_{[1:n]})|\prod_{i=1}^n|\hat\varphi_i(k_i)||k_i|^4\,.\label{e:Summable}
\end{align}
Now, the functions $\varphi_i$, $i=1,\dots,n$ are infinitely differentiable (and compactly supported, as they live on $\T$), 
which implies that their Fourier transform decays faster than any polynomial. 
Hence,~\eqref{e:Summable} is summable in $k_1,\dots, k_n$. 
Furthermore, invoking~\eqref{e:ConvDFTft} once more, it is immediate to see that 
the limit in $N$ of $F^N$ is given by  {
\begin{equation}\label{e:F}
F(k_1,\dots,k_n,k_{[1:n]}):=-2\pi\iota k_{[1:n]} \hat f(-k_{[1:n]})\bigg(\prod_{i=1}^n \hat\varphi_i(-k_i)\bigg)
\end{equation} 
which is nothing but the continuum counterpart of~\eqref{e:FN}}. 
By the dominated convergence theorem, we obtain 
\begin{align}
\lim_{N\to+\infty}&\sum_{k_1,\dots,k_n=0}^{\infty}|F^N(k_1,\dots,k_n,k_{[1:n]}) |\E\bigg[\bigg| \int_0^t \bigg(\prod_{i=1}^n {\mcb X}^{i,N}_s(e_{-k_i})\bigg) e^{\iota 2\pi \sum_{i=1}^n k_i \frac{\lfloor v_i^N s\rfloor}{N}}\dd s\bigg|\bigg]\nonumber\\
&=\sum_{k_1,\dots,k_n=0}^{\infty}|F(k_1,\dots,k_n,k_{[1:n]})| \lim_N\E\bigg[\bigg| \int_0^t \bigg(\prod_{i=1}^n {\mcb X}^{i,N}_s(e_{-k_i})\bigg) e^{\iota 2\pi \sum_{i=1}^n k_i \frac{\lfloor v_i^N s\rfloor}{N}}\dd s\bigg|\bigg]\,.\nonumber
\end{align}
{Now, if $k_1,\dots,k_n$ are such that $\sum_{i=1}^n k_i\neq 0$ or 
$\Pi_i\hat\varphi_i(k_i)\neq 0$, then the corresponding summand is $0$ by~\eqref{e:F}. 
In all other cases,~\eqref{e:seqConst} holds and this together with the fact that  
the processes $\{{\mcb X}^{i,N}(e_{-k_i})\colon i=1,\dots,n\}$ satisfy condition~\ref{itm:2}, 
imply that Proposition~\ref{p:RL} is applicable and therefore the inner limit is $0$, thus the proof is concluded.} 
\end{proof}

\section{Second order Boltzmann--Gibbs principle}\label{sec:BG}

In this section we present two formulations of the second order order Boltzmann--Gibbs principle for degree-two terms of type $\bar\xi^\alpha_x\bar\xi^{\beta}_{x+1}$, involving one or two species, i.e. for $\alpha, \beta\in\{A,B,C\}$. 

\subsection{A second order Boltzmann--Gibbs principle I}

The version of the Boltzmann--Gibbs principle here proposed is a multi-species generalization of Theorem 1 in~\cite{GJS17}.  The goal is to replace terms of the form $\bar\xi^\alpha_x\bar\xi^{\beta}_{x+1}$ by their respective centred averages over boxes of microscopically big (and macroscopically small) size. For this reason we introduce the left and right averages over boxes of size $\ell$ to be chosen appropriately. For $\ell\in\N$ and $x\in\Z$ we introduce the averages on boxes of size $\ell$, one to the right of $x$, one to the left:
\begin{equation}\label{e:localave}
	\ora{\xi}^{\alpha,\ell}_x=\frac{1}{\ell}\sum_{y=x+1}^{x+\ell}\xi^\alpha_y,\qquad \ola{\xi}^{\alpha,\ell}_x=\frac{1}{\ell}\sum_{y=x-\ell}^{x-1}\xi^\alpha_y.
\end{equation}
\begin{rem}
	This version of the principle is used in the analysis of the quadratic terms for the KPZ fields in cases \textbf{(I)} and \textbf{(II)}.
\end{rem}
We first observe that our dynamics satisfies Assumptions 2.1, 2.2 and 2.3 of~\cite{GJS17}. In fact, the invariant measure $\nu_{\rho}$ is of product form and it is such that $\int_{\Omega_N}\xi^{\alpha}_x(\eta)\nu_{\rho}(\dd\eta)=\rho_\alpha$ for each $x$ and each $\alpha$. \\
The associated Dirichlet form, defined on local functions $f\in L^2(\nu_{\rho})$, is given by
\begin{equation}\label{eq:DF}
D_N(f)=-\int_{\Omega_N}f(\eta)L_Nf(\eta)\nu_\rho(\dd\eta),
\end{equation}
and, using the fact that $c^{\alpha,\beta}_x(\eta^{x,x+1})=c^{\alpha,\beta}_x(\eta)\frac{\nu_{\rho}(\eta^{x,x+1})}{\nu_\rho(\eta)}$, it can be decomposed as
\begin{equation}\label{eq:DF2}
D_N(f)=\sum_{x\in\mathbb{T}_N}I_{x,x+1}^N(f),
\end{equation}
where
\begin{equation}\label{eq:DFdec}
I^N_{x,x+1}(f)=\frac{N^2}{2}\int_{\Omega_N}c_x^{\alpha,\beta}(\eta)(f(\eta^{x,x+1})-f(\eta))^2\nu_{\rho}(\dd\eta).
\end{equation}
And finally we recall that the instantaneous current is given by~\eqref{eq:inst_current}. 

Now we can state the theorem. For a function $v\in\ell^2(\Z)$ denote
\begin{equation*}
\|v\|^2_{2,N}:=\frac 1N\sum_{x\in\Z}v(x)^2<\infty.
\end{equation*}
For a function $\phi:\Omega_N\to\R$, let $\|\phi\|^2_{L^2_{\nu_\rho}}$ be the $L^2(\nu_{\rho})$-norm
\begin{equation*}
\|\phi\|^2_{L^2_{\nu_\rho}}=\int_{\Omega_N}\phi(\eta)^2\nu_{\rho}(\dd\eta).
\end{equation*}

\begin{thm}\label{thm:BG}
	Let $\alpha,\beta\in\{A,B\}$. Uniformly over $L\in\N$, $t>0$, and function $v\in\ell^2(\Z)$, we have
	\begin{multline}\label{eq:BG}
	\E_{\nu_\rho}\Bigg[\Bigg(\int_0^t\sum_{x\in\Z}v(x)\Big\{\bar\xi^{\alpha}_{x}(s)\bar\xi^{\beta}_{x+1}(s)-\ora{\xi}^{\beta,L}_x(s)\big[\ola{\xi}^{\alpha,L}_x(s)(1-\delta_{\alpha\beta})+\ora{\xi}^{\alpha,L}_x\delta_{\alpha\beta}\big]+\frac{\chi_\alpha}{L}\Big\}\dd s\Bigg)^2\Bigg]\\
	\lesssim t\left[\frac{L}{N^{a-1}}+\frac{tN}{L^2}\right]\|v\|^2_{2,N},
	\end{multline}
where $\chi_\alpha=\delta_{\alpha,\beta}\E_{\nu_{\rho}}[(\xi^\alpha_x-\xi^{\alpha}_{x+1})^2]$.
\end{thm}
\begin{proof}
Let's start with the case $\alpha\neq\beta$. Let $\ell_0\in\mathbb N$. We can split the integrand in~\eqref{eq:BG} in the following terms.
	\begin{align}
	\bar\xi^{\alpha}_{x}\bar\xi^{\beta}_{x+1}-\ola{\xi}^{\alpha,L}_x\ora{\xi}^{\beta,L}_x=&\bar\xi^{\alpha}_{x}(\bar\xi^{\beta}_{x+1}-\ora{\xi}^{\beta,\ell_0}_x)\label{eq:BG1}\\
	&+\ora{\xi}^{\beta,\ell_0}_x(\bar\xi^{\alpha}_{x}-\ola{\xi}^{\alpha,\ell_0}_x) \label{eq:BG2}\\
	&+\ola{\xi}^{\alpha,\ell_0}_x(\ora{\xi}^{\beta,\ell_0}_x-\ora{\xi}^{\beta,L}_x)\label{eq:BG3}\\
	&+\ora{\xi}^{\beta,L}_x(\ola{\xi}^{\alpha,\ell_0}_x-\ola{\xi}^{\alpha,L}_x) \label{eq:BG4}
	\end{align}
To obtain a bound on~\eqref{eq:BG1} and \eqref{eq:BG2} we apply the ``One-block estimate'', Proposition 5 of~\cite{GJS17}. For~\eqref{eq:BG3} and~\eqref{eq:BG4} we need a ``multi-scale argument'', shown in Proposition 8 of~\cite{GJS17}.
	
Applying Proposition 5 of~\cite{GJS17} to~\eqref{eq:BG1} with $\phi(\tau_x\xi^\alpha_0)=\bar\xi^\alpha_x$, we obtain
\begin{equation}\label{eq:BG7}
\E_{\mu_N}\Bigg[\Bigg(\int_0^t\sum_{x\in\Z}v(x)\Big\{\bar\xi^{\alpha}_{x}(\bar{\xi}^{\beta}_{x+1}-\ora{\xi}^{\beta,\ell_0}_x)\Big\}\dd s\Bigg)^2\Bigg]\leq Ct\frac{\ell_0^2}{N^{a-1}}\|v\|^2_{2,N},
\end{equation}
and we do the same for~\eqref{eq:BG2} taking $\psi(\tau_x\xi^{\beta}_0)=\ora\xi^{\beta,\ell_0}_x$:
{
\begin{equation}\label{eq:BG8}
\E_{\mu_N}\Bigg[\Bigg(\int_0^t\sum_{x\in\Z}v(x)\Big\{\ora\xi^{\beta,\ell_0}_{x}(\bar{\xi}^{\alpha}_{x}-\ola{\xi}^{\alpha,\ell_0}_x)\Big\}\dd s\Bigg)^2\Bigg]\leq Ct\frac{\ell_0^2}{N^{a-1}}\|v\|^2_{2,N}.
\end{equation}
For~\eqref{eq:BG3} and~\eqref{eq:BG4} we apply Proposition 8 of~\cite{GJS17}  respectively:}
\begin{align*}
&\E_{\mu_N}\Bigg[\Bigg(\int_0^t\sum_{x\in\Z}v(x)\Big\{\ola{\xi}^{\alpha,\ell_0}_x(\ora{\xi}^{\beta,\ell_0}_x-\ora{\xi}^{\beta,L}_x)\Big\}\dd s\Bigg)^2\Bigg]\leq C{t}\frac{1}{N^{a-1}}\left(L+\frac{\ell^2_0}{L}\right)\|v\|^2_{2,N}\\
&\E_{\mu_N}\Bigg[\Bigg(\int_0^t\sum_{x\in\Z}v(x)\Big\{\ora\xi^{\beta,L}_{x}(\bar{\xi}^{\alpha}_{x}-\ola{\xi}^{\alpha,\ell_0}_x)\Big\}\dd s\Bigg)^2\Bigg]\leq C{t}\frac{1}{N^{a-1}}\left(L+\frac{\ell^2_0}{L}\right)\|v\|^2_{2,N}
\end{align*}

We conclude the proof for $\alpha\neq\beta$ putting together all the aforementioned estimates.

When $\alpha=\beta$, we have more terms in the initial decomposition. In fact,
 \begin{align}
 \bar\xi^\alpha_{x}\bar\xi^\alpha_{x+1}-(\ora{\xi}_x^{\alpha,L})^2+\frac{\chi_\alpha}{L}
 =& \bar\xi^\alpha_{x}(\bar\xi^\alpha_{x+1}-\ora{\xi}^{\alpha,\ell_0}_x)\label{eq:BG1b}\\
 &+\ora{\xi}^{\alpha,\ell_0}_x(\bar\xi^\alpha_{x}-\ola{\xi}^{\alpha,\ell_0}_x) \label{eq:BG2b}\\
 &+\ola{\xi}^{\alpha,\ell_0}_x(\ora{\xi}^{\alpha,\ell_0}_x-\ora{\xi}^{\alpha,L}_x)\label{eq:BG3b}\\
 &+\ora{\xi}^{\alpha,L}_x(\ola{\xi}^{\alpha,\ell_0}_x-\bar{\xi}^\alpha_x) \label{eq:BG4b}\\
 &+\bar{\xi}^\alpha_x\ora{\xi}^{\alpha,L}_x-(\ora{\xi}^{\alpha,L}_x)^2+\frac{(\bar\xi^\alpha_{x}-\bar\xi^\alpha_{x+1})^2}{2L} \label{eq:BG5b}\\
 &-\frac{(\bar\xi^\alpha_{x}-\bar\xi^\alpha_{x+1})^2}{2L}+\frac{\chi_\alpha}{L}. \label{eq:BG6b}
 \end{align}
 The first four terms are analogous to the terms we saw above: to estimate~\eqref{eq:BG1b},~\eqref{eq:BG2b} and~\eqref{eq:BG4b} we use Proposition 5 of~\cite{GJS17}; for~\eqref{eq:BG3b} we apply the multiscale argument, Proposition 8 of~\cite{GJS17}. For the last two terms,~\eqref{eq:BG5b} and~\eqref{eq:BG6b} we need different arguments. The estimate for~\eqref{eq:BG5b} is a consequence of Proposition 7 of~\cite{GJS17}.  
For the last term~\eqref{eq:BG6b} we apply Proposition 9 of~\cite{GJS17}.
\end{proof}

\subsection{A second-order Boltzmann--Gibbs principle II}

Here we propose an alternative version of the second order Boltzmann--Gibbs principle, where instead of replacing the occupation variable $\bar\xi^\alpha_x$ with left/right averages over a box around $x$ of size $\ell=\ell(N)$, we use smooth functions supported over $[x,x+\ell]$ or $[x-\ell,x]$.

\begin{rem}
This version of the principle is used in the analysis of the quadratic terms of the diffusive fields: here we need that the fields are evaluated on smooth functions in order to apply  Theorem \ref{thm:Cross}. We also observe that, since in this case the fields have different velocities, they never live on the same frame and consequently they are never correlated, so we do not need to study the variance terms even for the same species.
\end{rem}

{
For $\ell>0$, we defined the functions 
\begin{equation}\label{e:rhol}
	\ola{\rho_\ell}(\cdot)=\frac{1}{\ell}\ola \rho\quad \textrm{and}\quad \ora{\rho_\ell}(\cdot)=\frac{1}{\ell}\ora \rho\,,
\end{equation}
where $\ola\rho$ and  $\ora\rho$ are smooth functions supported on $[-\ell,0]$ and $[0,\ell]$, respectively, 
such that
\begin{equation}\label{eq:smooth}
	\smallint_{-\ell}^0 \ola{\rho_\ell}(u)du=1 \quad \textrm{and}\quad  \smallint_0^\ell \ora{\rho_\ell}(u)du=1.
\end{equation}
}

Since the final goal is to write Dynkin's formula in terms of the fluctuation density fields $\mcb Y^{N,A}_s$ and $\mcb Y^{N,B}_s$, the first step is to replace the occupation variable $ \bar\xi^\alpha_x$ and $\bar\xi^\beta_{x+1}$, for $\alpha,\beta\in\{A,B\}$, with the integral of {$\ola{\rho_\ell}$} and {$\ora{\rho_\ell}$} with respect to centred empirical measure $\bar\pi^{N,\alpha}_s$ and $\bar\pi^{N,\beta}_s$ respectively, where $\bar\pi^{N,\alpha}_s$ is  given in \eqref{eq:emp_BG}. Recall \eqref{eq:rel_em_mea_ave}. 
We can rephrase the Boltzmann-Gibbs principle as follows.
\begin{thm}\label{thm:BG_v2}
	Let $\alpha,\beta\in\{A,B\}$. {For any $\varepsilon>0$, any $t>0$, and for any function $v\in\ell^2(\Z)$
	\begin{equation}\label{eq:BG_v2}
	\E_{\nu_\rho}\Bigg[\Bigg(\int_0^tds \sum_{x\in\Z}v(x)\Big\{\bar\xi^{\alpha}_{x}(s)\bar\xi^{\beta}_{x+1}(s)-\langle\bar\pi^{N,\alpha}_s,\ola{\rho_\varepsilon}(\cdot-\tfrac xN)\rangle\langle\bar\pi^{N,\beta}_s,\ora{\rho_\varepsilon}(\cdot-\tfrac xN)\rangle\Big\}\Bigg)^2\Bigg]
	\lesssim t\left[{\varepsilon}+\frac{t}{N\varepsilon^2}\right]\|v\|^2_{2,N}.
	\end{equation}}
\end{thm}
The proof follows the same steps as in the original version, i.e. we need to find estimates for each term in~\eqref{eq:BG1},\eqref{eq:BG2},\eqref{eq:BG3},\eqref{eq:BG4}, where now in place of $\ola{\xi}^{\alpha,\ell}_x$ and $\ora{\xi}^{\beta,\ell}_x$ we have 
$$\bar\pi^{N,\alpha}_x(\ola{\rho_\ell}):=\langle\bar\pi^{N,\alpha}_s,\ola{\rho_\ell}(\cdot-\tfrac xN)\rangle \quad \text{ and } \quad \bar\pi^{N,\beta}_x(\ora{\rho_\ell}):=\langle\bar\pi^{N,\beta}_s,\ora{\rho_\ell}(\cdot-\tfrac xN)\rangle.$$ 
This implies that we have to reformulate Proposition 5 and 6 of~\cite{GJS17} in terms of the functions $\ola{\rho_\ell}$ and $\ora{\rho_\ell}$.
\begin{prop}[One-block estimate]\label{prop:1BE}
{	Fix $N$ and $\ell_0\in\N$. Let $\phi,\psi:\Omega_N\to\R$ be local functions with zero mean with respect to the invariant measure  $\nu_\rho$ and such that the support of $\phi$ does not intersect the set of points $\{0,\dots,\ell_0\}$, and the support of $\psi$ does not intersect the set of points $\{-\ell_0,\dots,-1\}$. Then, for any $t>0$,  $v\in\ell^2(\Z)$ and $\alpha,\beta\in\{A,B\}$
	\begin{equation}\label{eq:1BE1}
	\E_{\nu_\rho}\Bigg[\Bigg(\int_0^t\sum_{x\in\Z}v(x)\phi(\tau_x\xi^\alpha_0(s))\Big[\bar\xi^{\beta}_{x+1}(s)-\bar\pi^{N,\beta}_x(\ora{\rho}_{\ell_0/N})\Big]\dd s\Bigg)^2\Bigg]\lesssim \frac{t\ell_0^2}{N}\|\phi\|^2_{L^2_{\nu_\rho}}\|v\|^2_{2,N},	\end{equation}
	\begin{equation}\label{eq:1BE2}
	\E_{\nu_\rho}\Bigg[\Bigg(\int_0^t\sum_{x\in\Z}v(x){\psi}(\tau_x\xi^{\beta}_1(s))\Big[\bar\xi^{\alpha}_{x}(s)-\bar\pi^{N,\alpha}_x(\ola{\rho}_{\ell_0/N})\Big]\dd s\Bigg)^2\Bigg]\lesssim \frac{t\ell_0^2}{N}\|\psi\|^2_{L^2_{\nu_\rho}}\|v\|^2_{2,N}.
	\end{equation}}
\end{prop}
\begin{proof}
	We show the proof only of~\eqref{eq:1BE1}, since the other is completely analogous. {By  \cite[Lemma 2.4]{KLO} we bound the expectation in~\eqref{eq:1BE1} from above by
	\begin{equation*}
		Ct\Big\|\sum_{x\in \Z}v(x)\phi(\tau_x\xi^\alpha_0)\Big[\bar\xi^{\beta}_{x+1}-\bar\pi^{N,\beta}_x(\ora{\rho}_{\ell_0/N})\Big] \Big\|_{-1}^2
	\end{equation*}
	where the $H_{-1}$-norm is defined in \cite[Equation (1.23)]{KLO} thorough a variational formula, and in particular the previous expression is equal to
	\begin{equation}\label{eq:1BEvar}
		C_1t\sup_{f\in L^2({\nu_\rho})}\left\{2\int_{\Omega_N}\sum_{x\in\Z}v(x)\phi(\tau_x\xi^\alpha_0)\Big[\bar\xi^{\beta}_{x+1}-\bar\pi^{N,\beta}_x(\ora{\rho}_{\ell_0/N})\Big]f(\eta)\nu_\rho(\dd\eta)-D_N(f)\right\},
	\end{equation}}
	where $D_N(f)$ is the Dirichlet form~\eqref{eq:DF}.
	First we observe that {
	\begin{equation*}
		\begin{aligned}
		\bar\xi^{\beta}_{x+1}-\bar\pi^{N,\beta}_x(\ora{\rho}_{\ell_0/N})=\bar{\xi}^{\beta}_{x+1}-\frac{1}{\ell_0}\sum_{y=x+1}^{x+\ell_0}\ora\rho(\tfrac{y-x}{N})\bar{\xi}^\beta_y
		\end{aligned}		
	\end{equation*}
	 Now the key point is to use the fact that $\frac{1}{\ell_0}\sum_{y=x+1}^{x+\ell_0}\ora\rho(\tfrac{y-x}{N})\sim 1$ to write the previous display as  
	\begin{equation}\label{eq:inter}
	\bar\xi^{\beta}_{x+1}\Big(1-\frac{1}{\ell_0}\sum_{y=x+1}^{x+\ell_0}\ora\rho(\tfrac{y-x}{N})\Big)+\frac{1}{\ell_0}\sum_{y=x+1}^{x+\ell_0}\ora\rho(\tfrac{y-x}{N})(\xi^{\beta}_{x+1}-\xi^{\beta}_{y}).
	\end{equation}
	Since $\bar\rho(\cdot)$ is smooth (in fact, $C^1$ is enough), an elementary computation shows that
	\begin{equation*}\begin{split}
	\Big|1-\frac{1}{\ell_0}\sum_{y=x+1}^{x+\ell_0}\ora\rho(\tfrac{y-x}{N})\Big|&= \Big|\int_0^{\ell_0/N}\frac{N}{\ell_0}\ora\rho(u)du-\frac{1}{\ell_0}\sum_{y=1}^{\ell_0}\ora\rho(\tfrac{y}{N})\Big|\\
	&= \Big|\int_0^{\ell_0}\frac{1}{\ell_0}\ora\rho(\tfrac{v}{N})dv-\frac{1}{\ell_0}\sum_{y=1}^{\ell_0}\ora\rho(\tfrac{y}{N})\Big|\\\leq &
	\frac{1}{\ell_0}\sum_{y=1}^{\ell_0 }\int_{y-1}^{y}\Big|\ora\rho(\tfrac vN )-\ora\rho(\tfrac{y}{N})\Big|dv\\\leq &
	\frac{1}{\ell_0 N}\sum_{y=1}^{\ell_0 }\|\ora{\rho}'\|_\infty=O(N^{-1}).
	\end{split}
	\end{equation*}

This means that we can sum and subtract the  leftmost term of \eqref{eq:inter} to   \eqref{eq:1BE1} and we note that, by stationary, it holds  
		\begin{equation*}
	\E_{\nu_\rho}\Bigg[\Bigg(\int_0^t\sum_{x\in\Z}v(x)\phi(\tau_x\xi^\alpha_0(s))\bar\xi^{\beta}_{x+1}\Big[1-\frac{1}{\ell_0}\sum_{y=x+1}^{x+\ell_0}\ora\rho(\tfrac{y-x}{N})\Big]\dd s\Bigg)^2\Bigg]\lesssim \frac{t}{N}\|\phi\|^2_{L^2_{\nu_\rho}}\|v\|^2_{2,N},	\end{equation*}
	Now since the rightmost term of \eqref{eq:inter}
can be written as 	
	\begin{equation*}
	\frac{1}{\ell_0}\sum_{y=x+2}^{x+\ell_0}\sum_{z=x+1}^{y-1}\ora\rho(\tfrac{y-x}{N})(\xi^{\beta}_{z}-\xi^{\beta}_{z+1}).
	\end{equation*}
	we can proceed exactly as in the proof of Proposition 5 of~\cite{GJS17} to obtain the desired estimate.}
\end{proof}

\begin{prop}[Two-blocks estimate]\label{prop:2Block}
	Fix $\ell_k\in\N$. Define $\ell_{k+1}=2\ell_k$. Let $\phi,\psi:\Omega_N\to\R$ be  local functions with zero mean with respect to the invariant measure $\nu_\rho$ and such that the support of $\phi$ does not intersect the set of points $\{0,\dots,\ell_k\}$, and the support of $\psi$ does not intersect the set of points  $\{-\ell_k,\dots,-1\}$. {Then, for any $t>0$,  $v\in\ell^2(\Z)$ and $\alpha,\beta\in\{A,B\}$
	\begin{equation}\label{eq:2B1}
	\E_{\nu_\rho}\Bigg[\Bigg(\int_0^t\sum_{x\in\Z}v(x)\phi(\tau_x\xi^\alpha_0(s))\Big[\bar\pi^{N,\beta}_x(\ora\rho_{\ell_{k}/N})-\bar\pi^{N,\beta}_x(\ora\rho_{\ell_{k+1}/N})\Big]\dd s\Bigg)^2\Bigg]\lesssim t\ell_k^2\frac{1}{N}\|\phi\|^2_{L^2_{\nu_\rho}}\|v\|^2_{2,N},
	\end{equation}
	\begin{equation}\label{eq:2B2}
	\E_{\nu_\rho}\Bigg[\Bigg(\int_0^t\sum_{x\in\Z}v(x)\psi(\tau_x\xi^{\beta}_1(s))\Big[\bar\pi^{N,\alpha}_x(\ola\rho_{\ell_{k}/N})-\bar\pi^{N,\alpha}_x(\ola\rho_{\ell_{k+1}/N})\Big]\dd s\Bigg)^2\Bigg]\lesssim t\ell_k^2\frac{1}{N}\|\psi\|^2_{L^2_{\nu_\rho}}\|v\|^2_{2,N}.
	\end{equation}}
\end{prop}

\begin{proof}
The first steps follow closely the one block estimate and Proposition 6 of~\cite{GJS17}. The difference lies in the estimate of $\bar\pi^{N,\beta}_x(\ora\rho_{\ell_{k}})-\bar\pi^{N,\beta}_x(\ora\rho_{\ell_{k+1}})$.
We have that
	{\begin{align*}
		\bar\pi^{N,\beta}_x(\ora\rho_{\ell_{k}/N})-\bar\pi^{N,\beta}_x(\ora\rho_{\ell_{k+1}/N})&=\frac{1}{\ell_k}\sum_{y=x+1}^{x+\ell_k}\ora\rho(\tfrac{y-x}{N})\bar\xi^\beta_y-\frac{1}{2\ell_k}\sum_{y=x+1}^{x+2\ell_k}\ora\rho(\tfrac{y-x}{N})\bar\xi^\beta_y \nonumber\\
		&=\frac{1}{2\ell_k}\sum_{y=x+1}^{x+\ell_k}\ora\rho(\tfrac{y-x}{N})\bar\xi^\beta_y-\frac{1}{2\ell_k}\sum_{y=x+\ell_k+1}^{x+2\ell_k}\ora\rho(\tfrac{y-x}{N})\bar\xi^\beta_y\nonumber\\
		&=\frac{1}{2\ell_k}\sum_{y=x+1}^{x+\ell_k}\ora\rho(\tfrac{y-x}{N})\bar\xi^\beta_y-\frac{1}{2\ell_k}\sum_{y=x+1}^{x+\ell_k}{\ora\rho(\tfrac{y-x+\ell_k}{N})\bar\xi^\beta_{y+\ell_k}}\\
		&=\frac{1}{2\ell_k}\sum_{y=x+1}^{x+\ell_k}\big[\ora\rho(\tfrac{y-x}{N})\bar\xi^\beta_y-\ora\rho(\tfrac{y-x+\ell_k}{N})\bar\xi^\beta_{y+\ell_k}\big]\nonumber\\
		&=\frac{1}{2\ell_k}\sum_{y=x+1}^{x+\ell_k}\ora\rho(\tfrac{y-x}{N})[\bar\xi^\beta_y-\bar\xi^\beta_{y+\ell_k}]\nonumber\\
		&+\frac{1}{2\ell_k}\sum_{y=x+1}^{x+\ell_k}\big[\ora\rho(\tfrac{y-x}{N})-\ora\rho(\tfrac{y-x+\ell_k}{N})\big]\bar\xi^\beta_{y+\ell_k}\nonumber\\
		&=\frac{1}{2\ell_k}\sum_{y=1}^{\ell_k}\ora\rho(\tfrac{y}{N})\big[\bar\xi^\beta_{y+x}-\bar\xi^\beta_{y+x+\ell_k}\big]\nonumber\\
		&+\frac{1}{2\ell_k}\sum_{y=x+1}^{x+\ell_k}\big[\ora\rho(\tfrac{y-x}{N})-\ora\rho(\tfrac{y-x+\ell_k}{N})\big]\bar\xi^\beta_{y+\ell_k}.\nonumber
	\end{align*}
We observe that by the Cauchy--Schwarz the contribution of the  last term can be bounded by 
\begin{equation*}
	\E_{\nu_\rho}\Bigg[\Bigg(\int_0^t\sum_{x\in\Z}v(x)\phi(\tau_x\xi^\alpha_0(s))
	\frac{1}{2\ell_k}\sum_{y=x+1}^{x+\ell_k}\big[\ora\rho(\tfrac{y-x}{N})-\ora\rho(\tfrac{y-x+\ell_k}{N})\big]\bar\xi^\beta_{y+\ell_k}
	\dd s\Bigg)^2\Bigg]\lesssim t\ell_k\frac{1}{N}\|\phi\|^2_{L^2_{\nu_\rho}}\|v\|^2_{2,N},
	\end{equation*}
At this point we can proceed as in the proof of Proposition 6 of~\cite{GJS17} and conclude.}
\end{proof}

\begin{appendix}

\section{Mode coupling theory}\label{sec:mode}

	In this section we give a sketch of the application of the mode coupling theory developed in \cite{Schuetz17KPZ} to our specific model. Note that 
since the system has two conserved quantities, $\rho^A$ and $\rho^B$ (because as we mentioned above $\rho^C=1-\rho^A-\rho^B$), it is possible to have an estimate on the scaling exponent and form of the expected fluctuation limits for quantities $A$ and $B$. We refer the interested reader to~\cite{Popkov2015,Schuetz17KPZ}.  

Recall \eqref{eq:inst_current}. Then
\begin{equation}\label{eq:exp_inst_current}
	\begin{aligned}
		\mathbb E[j^A_{x,x+1}]&=\frac{E_{A}-E_{B}}{N^\gamma}\rho^{A}\rho^{B}+\frac{E_{A}-E_{C}}{2N^\gamma}\Big(2\rho^A(1-\rho^{A})-2\rho^{A}\rho^{B}\Big)\,,\\
		\mathbb E[j^B_{x,x+1}]&=\frac{E_{B}-E_{A}}{N^\gamma}\rho^{A}\rho^{B}+\frac{E_{B}-E_{C}}{2N^\gamma}\Big(2\rho^B(1-\rho^{B})-2\rho^{A}\rho^{B}\Big)\,.
	\end{aligned}
\end{equation} 
From this we get  that the current of the system  is equal to
\begin{equation}\label{eq:current}
\begin{aligned}
j(\rho)&
=\begin{pmatrix}
-\nabla\rho^{A}+\frac{1}{N^\gamma}\Big(\rho^{A}(1-\rho^{A})(E_{A}-E_{C})-\rho^{A}\rho^{B} (E_{B}-E_{C})\Big)\\
-\nabla\rho^{B}+\frac{1}{N^\gamma}\Big(\rho^{B}(1-\rho^{B})(E_{B}-E_{C})-\rho^{A}\rho^{B}(E_{A}-E_{C})\Big)
\end{pmatrix}.
\end{aligned}
\end{equation}
From last identity, the Jacobian matrix is thus given by
\begin{equation}\label{eq:jacobian}
J=\frac{1}{N^\gamma}\begin{pmatrix}
(1-2\rho^{A})(E_{A}-E_{C})-\rho^{B}(E_{B}-E_{C}) & -\rho^{A}(E_{B}-E_{C})\\
-\rho^{B}(E_{A}-E_{C}) & -\rho^{A}(E_{A}-E_{C})+(1-2\rho^{B})(E_{B}-E_{C})
\end{pmatrix}\,.
\end{equation}
To obtain the normal modes we need to find the matrix $R$ that diagonalizes $J$, $RJR^{-1}={\rm diag}(v_\pm)$, where  
\begin{equation}\label{eq:eigenvalues}
v_{\pm}=\frac{1}{2N^\gamma}[(E_{A}-E_{C})(1-3\rho^{A})+(E_{B}-E_{C})(1-3\rho^{B})\pm\delta],\\
\end{equation}
are the eigenvalues of the matrix $J$. Above
\begin{equation}\label{eq:determinant}
\delta=\sqrt{[(E_{A}-E_{C})(1-\rho^{A})-(E_{B}-E_{C})(1-\rho^{B})]^2+4(E_{A}-E_{C})(E_{B}-E_{C})\rho^{A}\rho^{B}}.
\end{equation}
The  corresponding eigenvectors are given by 
\begin{equation}\label{eq:eigenvectors}
\tau_{\pm}=\begin{pmatrix}
-\frac{1}{2(E_{A}-E_{C})\rho^{B}}[(E_{A}-E_{C})(1-\rho^{A})-(E_{B}-E_{C})(1-\rho^{B})\pm\delta]\\1
\end{pmatrix}.
\end{equation}

Let us now rewrite the previous expressions adapted to the case (I). 
As observed in Remark~\ref{rem:(I)(II)}, case (II), $E_B-E_A=E_C-E_A=E\neq 0$, can be deduced by the previous one.

\subsection{Case (I): $E_{A}-E_{C}=E_{B}-E_{C}=E$}\label{sub:special_case}

In this case the expressions above simplify. We get the jacobian matrix 
\begin{equation*}
J=\frac{1}{N^\gamma}\begin{pmatrix}E (1 - 2 \rho_A) - E\rho_{B} & -E \rho_{A}\\
-E \rho_{B}& E (1 - 2 \rho_{B}) - E \rho_{A}
\end{pmatrix}
\end{equation*} and the eigenvalues are 
$$v_+:=-\frac{1}{N^\gamma}E(\rho_{A}+\rho_{B}-1)\quad \textrm{and}\quad v_-:=-\frac{1}{N^\gamma}E(2\rho_{A}+2\rho_{B}-1)$$
and the corresponding eigenvectors are $\tau_+=(-1,1)$ and $\tau_-=(\frac{\rho_{A}}{\rho_{B}},1)$.
The matrix $R^{-1}$  and $R$ are given by
\begin{equation*}
R^{-1}=\begin{pmatrix}-1& \frac{\rho_{A}}{\rho_{B}}\\
1& 1
\end{pmatrix}
\quad \textrm{and}\quad
R=\frac{1}{\rho_{A}+\rho_{B}}\begin{pmatrix}-\rho_{B}& \rho_{A}\\
{\rho_{B}}& \rho_{B}
\end{pmatrix}
\end{equation*} 
The hessians are 
\begin{equation*}
\mcb H^1=\frac{1}{N^\gamma}\begin{pmatrix}-2E& -E\\
-E& 0
\end{pmatrix}
\quad\textrm{and}\quad
\mcb H^2=\frac{1}{N^\gamma}\begin{pmatrix}0& -E\\
-E& -2E
\end{pmatrix}\,.
\end{equation*} 
A simple computation shows that $(R^{-1})^{T}H^1 R$ is equal to 
\begin{equation*}
(R^{-1})^{T}\mcb H^1 R=\frac{1}{ \rho_{A}+ \rho_{B}}\begin{pmatrix}0&\frac{E( \rho_{A}+ \rho_{B})}{2N^\gamma}\\
\frac{-E \rho_{A}}{2N^\gamma}- \rho_{B}(E(\frac{ \rho_{A}}{2N^\gamma  \rho_{B}})-\frac{E}{2N^\gamma})& \rho_{A}(-E(\frac{ \rho_{A}}{2N^\gamma  \rho_{B}})-\frac{E}{2N^\gamma})-\frac{E \rho_{A}}{2N^\gamma}
\end{pmatrix}
\end{equation*} 
and 
\begin{equation*}
(R^{-1})^{T}\mcb H^2 R=\frac{1}{ \rho_{A}+ \rho_{B}}\begin{pmatrix}0& -\frac{E( \rho_{A}+ \rho_{B})}{2N^\gamma}\\
\frac{E \rho_{B}}{2N^\gamma}+ \rho_{B}(-E(\frac{ \rho_{A}}{2N^\gamma  \rho_{B}})-\frac{E}{N^\gamma})& \rho_{B}(-E(\frac{ \rho_{A}}{2N^\gamma  \rho_{B}})-\frac{E}{N^\gamma})-\frac{E \rho_{A}}{2N^\gamma}
\end{pmatrix}\,.
\end{equation*} 
The coupling matrices $G^1$ and $G^2$ are such that $G^1_{1,1}=0=G^2_{1,1}$, but also $G^2_{2,2}\neq 0$. This means that the first mode is  OU and the second is KPZ and this is independent from the choice of the  densities $\rho^A, \rho^B$ and $\rho^C$. 

We note that the predictions from  ~\cite{Schuetz17KPZ} are done for the strong asymmetric regime,  corresponding to the choice $\gamma=0$. But  the available techniques still do not allow us to go down the regime $\gamma=1/2$. So what is expected to hold for this model in the KPZ regime is exactly the same behaviour as in the case of a single type of particle i.e. the WASEP, for which one gets for $\gamma>1/2$ diffusive fluctuations and for $\gamma=1/2$ the stochastic Burgers equation. Below the regime $\gamma<1/2$ little is known, but we expect that the same process as for ASEP should appear, i.e. the KPZ-fp.

\subsection{The equal density case}
Throughout the paper, for sake of simplicity, we have made the assumption $\rho_A=\rho_{B}=1/3$: in this case, the formula \eqref{eq:determinant} simplifies and we obtain
$$\delta=\frac 23\sqrt{(E_{A}-E_{C})^2+(E_{B}-E_{C})^2-(E_{A}-E_{C})(E_{B}-E_{C})}$$
and the eigenvectors \eqref{eq:eigenvectors} simplify to 
\begin{equation*}
	\tau_{\pm}=\begin{pmatrix}
		-\frac{c_\pm}{E_{A}{-E_{C}}}\\1
	\end{pmatrix},
\end{equation*}
with eigenvalues
\begin{equation*}
	v_{\pm}=\pm\frac{\delta}{2N^\gamma}.
\end{equation*}
Above we used the notation in \eqref{assumption}.
To obtain the linear combination of the fields that one should look at,  we need to find the matrix $R$ that diagonalizes $J$. Observe that   $R^{-1}$ is the matrix whose columns are the eigenvectors of $J$ so that 
\begin{equation*}
	R^{-1}=\begin{pmatrix}-\frac{c_+}{E_{A}{-E_{C}}}&-\frac{c_-}{E_{A}{-E_{C}}} \\
		1& 1
	\end{pmatrix}
\end{equation*}
and its inverse is equal to  
\begin{equation*}
	R=-\frac{E_{A}-E_{C}}{3\delta}\begin{pmatrix}1&\frac{c_-}{E_{A}{-E_{C}}}\\
		-1&\frac{c_+}{E_{A}{-E_{C}}}
	\end{pmatrix}\,.
\end{equation*}
According to the nonlinear fluctuating hydrodynamic theory, the quantities that we should look at are equal to $R (\bar \xi^A_x, \bar \xi^{B}_x)^T,$
which gives 
\begin{equation}\label{normal_modes}
	\begin{aligned}
		\phi_x^+&=-\frac{E_A-E_{C}}{3\delta}\bar{\xi}^{A}_x-\frac{c_-}{3\delta}\bar{\xi}^{B}_x,\\
		\phi_x^-&=\frac{E_A-E_{C}}{3\delta}\bar{\xi}^{A}_x+\frac{c_+}{3\delta}\bar{\xi}^{B}_x. 
	\end{aligned}
\end{equation}

Therefore, the quantities $\phi_x^+
$ and $\phi_x^-$ are the conserved quantities that we should look at and on a frame with velocity $v_+$ and $v_-$, respectively.  
Since we can multiply our fields with suitable constants, we take  
\begin{equation*}\label{eq:KPZ-KPZ}
	\begin{split}
		&{\mcb{Z}}^{N,+}_t(f)=\mcb Y^{N,A}_t(T_{v_+ N^b t}f){+\frac{c_-}{E_A-E_{C}}}\mcb Y^{N,B}_t(T_{v_+N^bt}f),\\
		&\mcb{Z}^{N,-}_t(f)=\mcb Y^{N,A}_t(T_{v_-N^bt}f)+{\frac{c_+}{E_A-E_{C}}}\mcb Y^{N,B}_t(T_{v_-N^bt}f).
	\end{split}
\end{equation*}
The constant $b$ is fixed depending on the time scale. 

Now that we have the fluctuations fields fixed,  let us see the predictions on the form of the fluctuations for each one of these quantities. 
To do that we look now at the corresponding Hessians of the entries of the jacobian matrix $J$:
\begin{equation}\label{eq:hessians}
	\mcb H^1=\frac{1}{N^\gamma}\begin{pmatrix}
		-2(E_A-E_{C}) & -(E_{B}-E_{C})\\
		-(E_{B}-E_{C}) & 0
	\end{pmatrix} \quad \textrm{and }\quad 
	\mcb H^2=\frac{1}{N^\gamma}\begin{pmatrix}
		0& -(E_{A}-E_{C})\\
		-(E_{A}-E_{C}) & 
		-2(E_{B}-E_{C}) 
	\end{pmatrix} .
\end{equation}
The coupling constants, which are  determined by the above matrices, are given on $i\in\{1,2\}$ by
$G^i=\tfrac 12\sum_{j=1}^2 R_{i,j}[(R^{-1})^T \mcb H^j R^{-1}]
$
where $R_{i,j}$ is the entry of the matrix $R$.
A simple computation shows that 
$$[(R^{-1})^T \mcb H^1 R^{-1}]=\frac{1}{(E_{A}-E_{C})N^\gamma}\begin{pmatrix}
	-c_+^2+(E_{B}-E_{C})c_+&-c_+c_-+\frac{c_++c_-}{2}(E_{B}-E_{C})\\
	-c_+c_-+\frac{c_++c_-}{2}(E_{B}-E_{C})& -c_-^2+(E_{B}-E_{C})c_-
\end{pmatrix} .
$$
Moreover, 

$$[(R^{-1})^T \mcb H^2 R^{-1}]=\frac{1}{N^\gamma}\begin{pmatrix}
	c_+-(E_{B}-E_{C})&\frac{c_++c_-}{2}-(E_{B}-E_{C})\\
	\frac{c_++c_-}{2}-(E_{B}-E_{C})& c_--(E_{B}-E_{C})
\end{pmatrix} \,.
$$   
From this  we get 
$$G^1=\begin{pmatrix}
	2g_1 & g_2\\g_2& 0
\end{pmatrix} \quad \textrm{and}\quad G^2=\begin{pmatrix}
	0 & g_1\\g_1 & 2g_2
\end{pmatrix},
$$where $$g_1:=-\frac{1}{12N^\gamma}\big\{-c_+^2+(E_{B}-E_{C})(c_+-c_-)+c_-c_+\big\}$$
$$g_2:=\frac{1}{12N^\gamma}\big\{-c_-^2+(E_{B}-E_{C})(c_--c_+)+c_-c_+\big\}.$$
Depending on the fact that the constants $g_1$ and $g_2$ are null or not we obtain the limits predicted in Figure~\ref{fig:MCM}, we expect to observe diffusive and/or KPZ behaviour.

In most general case  $E_{A}-E_{C}\neq E_{B}-E_{C}$, both modes~\eqref{eq:KPZ-KPZ} are predicted to have KPZ behaviour.
Let us now rewrite the previous expressions in the several particular  cases that we have  explored above. 
We start with the case (I): $E_{A}-E_{C}=E_{B}-E_{C}=E$. In this case the expressions above simplify. In the equal density case $\rho^A= \rho^B=\rho^C=\rho$, we get
\begin{equation}
\delta=\frac 23E.
\end{equation}
The eigenvalues are
\begin{equation}\label{mc:velo}
v_{\pm}=\pm \frac{E}{3N^\gamma},
\end{equation}
with eigenvectors
\begin{equation}
\tau_+=(-1, 1),\quad \tau_-=(1, 1).
\end{equation}
The diagonalising matrix $R$  and its inverse $R^{-1}$ (the matrix with columns $\tau_+, \tau_-$) are  given by
\begin{equation}
R=\frac 12\begin{pmatrix}
-1 & 1\\1 & 1
\end{pmatrix},\quad R^{-1}=\begin{pmatrix}
-1 & 1\\1 & 1
\end{pmatrix}.
\end{equation}
Thus, the normal modes are equal to $R(\bar\xi^A, \bar\xi^{B})^T$, i.e. 
\begin{equation}
\phi_x^+=-\frac 12\bar\xi^{A}_x+\frac 12 \bar\xi^{B}_x,\quad \phi_x^-=\frac 12\bar\xi^{A}_x+\frac 12\bar\xi^{B}_x.
\end{equation}
Observe that since we can multiply all the modes by a suitable  constant then we take as  fluctuation fields the fields in \eqref{e:case1}.
Now for the predictions,  note that  the coupling matrices are equal to 
\begin{equation}
	G^1=\frac{1}{2}\begin{pmatrix}0& -\frac{E}{N^\gamma}\\
		-\frac{E}{N^\gamma}&0
	\end{pmatrix}
	\quad \textrm{and }\quad 
	G^2=\begin{pmatrix}0& 0\\
		0&-\frac{E}{N^\gamma}
	\end{pmatrix}
\end{equation}
which   means that  the first field  should behave diffusively, and the second should behave as  KPZ, see \cite{Schuetz17KPZ}.
By symmetry the results we obtain in this case also give the same for the case  (II): $E_{B}-E_{A}=E_{C}-E_{A}=E\neq 0$.  Indeed,
in this case, we  also have $\delta=\frac 23E,$  and  the eigenvalues of $J$ are the same as in case (I), i.e. \eqref{mc:velo}, but the eigenvectors are given by
$
	\tau_+=(0, 1)\quad \tau_-=(2, -1).
$
Moreover, we have 
\begin{equation}
		R=\frac 12\begin{pmatrix}
			1 & 2\\1& 0
		\end{pmatrix},\quad R^{-1}=\begin{pmatrix}
			0& 2\\1& -1
		\end{pmatrix}.
\end{equation}
The normal modes are equal to $R(\bar\xi^A_x, \bar\xi^{B}_x)^T$, i.e. 
\begin{equation}\label{eq:modes2}
	\phi_x^+=\frac{1}{2}\bar\xi^{A}_x+\bar\xi^{B}_x,\quad \phi_x^-=\frac{1}{2}\bar\xi^{A}_x.
\end{equation}
Since we can multiply by suitable constants we consider then the fields given in  \eqref{e:Case2}. 
Similar computations to the ones above, suggest that the first field should behave diffusively, and the second should behave as KPZ.

\section{A version of Riemann-Lebesgue lemma}\label{a:RL}

In this appendix, we provide a version of the Riemann-Lebesgue lemma for products of stochastic processes 
for which a bound on the supremum and increments is available. 
In the context of the present paper, the following proposition is used in 
Section~\ref{sec:cross} to determine the behaviour of the time-integral of  the product of processes living in different 
time-frames.

\begin{prop}\label{p:RL}
Let $\{v_i^N\}_N$, $ i=1,\dots, n$, be $n$ diverging sequence of constants. 
Let $T>0$ and, for every $N\in\N$, let $\{\mcb A^{i,N}\colon i=1,\dots,n\}_{N}$ be $n$ 
real--valued stochastic processes on $[0,T]$
defined on the same probability space.  
Assume that there exists $\alpha\in(0,1)$ and $C>0$, such that uniformly in $N$,
\begin{align}
\E\bigg[\sup_{s\leq T} |\mcb P_s^{N}|\bigg]&\leq C\,,\label{e:sup}\\
\E\left[|\mcb P_t^{N}-\mcb P_s^{N}|\right]&\leq C(t-s)^{\alpha}\,,\qquad \text{for all $s,t\in[0,T]$,}\label{e:Hol}
\end{align}
where, for $s\in[0,T]$, we denoted by $\mcb P^N_s:=\prod_{i=1}^n\mcb A^{i, N}_s$.  

Then, for any $t\in[0,T]$ and any integers $k_1,\dots,k_n\in\Z$ {such that the limit in~\eqref{e:seqConst} holds},
we have
\begin{equation}\label{e:RL}
\lim_{N\to\infty} \E\left[\left| \int_0^t\prod_{i=1}^n\mcb A_s^{i,N} e^{-2\pi \iota k_i \frac{\lfloor v_i^N s\rfloor}{N}}\dd s\right|\right]=0\,,
\end{equation}
where $\lfloor \cdot \rfloor$ denotes the integer part of $\cdot$. 
\end{prop}

\begin{proof}
Let us introduce the following notation, 
\begin{equation}\label{e:tildi}
\widetilde {\mcb A}_s^{i,N}:=\mcb  A_s^{i,N} e^{-2\pi \iota k_i \frac{\lfloor v_i^N s\rfloor-v_i^N s}{N}}\,,\qquad \widetilde{\mcb P}^N_s:=\prod_{i=1}^n \widetilde {\mcb A}_s^{i,N}\,,\qquad  C_N := \frac1N\sum_{i=1}^n k_iv_i^N
\end{equation}
and let $\mcb I^N$ be the integral inside the expectation of~\eqref{e:RL}. 
Let $N$ be fixed and, without loss of generality, assume $C_N>0$. Note that 
\begin{align*}
\mcb I^N= \int_0^t \widetilde{\mcb P}_s^N e^{-2\pi\iota C_Ns}\dd s&=\int_0^{\frac1{C_N}} \widetilde{\mcb P}_s^N e^{-2\pi \iota  C_N s}\dd s+\int_{\frac1{C_N}}^t \widetilde{\mcb P}_s^N e^{-2\pi \iota  C_N s}\dd s\\
&=\int_0^{\frac1{C_N}} \widetilde{\mcb P}_s^N e^{-2\pi \iota  C_N s}\dd s-\int_{0}^{t-\frac1{C_N}} \widetilde{\mcb P}^N_{s+\frac1{C_N}} e^{-2\pi \iota  C_N s}\dd s
\end{align*}
where we used that $e^{-2\pi \iota}=1$, and similarly
\begin{equation*}
\mcb I^N=\int_0^{t-\frac1{C_N}} \widetilde{\mcb P}_s^N e^{-2\pi \iota  C_N s}\dd s+\int_{t-\frac1{C_N}}^t \widetilde{\mcb P}_s^N e^{-2\pi \iota  C_N s}\dd s.
\end{equation*}
By summing up the two previous equalities, we have 
\begin{align}
&|\mcb I^N|=\tfrac{1}{2}\Big|\int_0^{\frac1{C_N}} \widetilde{\mcb P}_s^N e^{-2\pi \iota  C_N s}+\int_{t-\frac1{C_N}}^t \widetilde{\mcb P}_s^N e^{-2\pi \iota  C_N s}\dd s+\int_0^{t-\frac1{C_N}} (\widetilde{\mcb P}_s^N-\widetilde{\mcb P}^N_{s+\frac1{C_N}}) e^{-2\pi \iota  C_N s}\dd s\Big| \nonumber\\
&\lesssim \Big|\int_0^{\frac1{C_N}} \widetilde{\mcb P}_s^N e^{-2\pi \iota  C_N s}\dd s\Big|+\Big|\int_{t-\frac1{C_N}}^t \widetilde{\mcb P}_s^N e^{-2\pi \iota  C_N s}\dd s\Big|+ \Big|\int_0^{t-\frac1{C_N}} (\widetilde{\mcb P}_s^N-\widetilde{\mcb P}^N_{s+\frac1{C_N}}) e^{-2\pi \iota  C_N s}\dd s\Big|\nonumber\\
&=:\one+\two+\three\label{e:three}
\end{align}
and we will separately bound each of the terms above. Let us begin with $\one+\two$. Since by definition~\eqref{e:tildi}, 
for all $i=1,\dots,n$,
$|\widetilde {\mcb A}_s^{i,N}|\leq |\mcb A_s^{i,N}|$, by bringing the modulus inside 
and applying a supremum bound, we have 
\begin{equation}\label{e:i+ii}
\E[\one+\two]\leq \frac{1}{C_N} \bigg[\sup_{s\leq T} |\widetilde{\mcb P}_s^{N}|\bigg]\lesssim \frac{1}{C_N} \bigg[\sup_{s\leq T} |{\mcb P}_s^{N}|\bigg]\lesssim  \frac1{C_N}
\end{equation}
where the last step is a consequence of~\eqref{e:sup}. 
For $\three$, we bound the time increment of $\widetilde{\mcb P}^N$ between $0\leq s_1<s_2\leq T$ as 
\begin{align*}
|\widetilde{\mcb P}^N_{s_1}-\widetilde{\mcb P}^N_{s_2}|\leq |{\mcb P}^N_{s_1}-{\mcb P}^N_{s_2}|+\sum_{j=1}^2|\widetilde{\mcb P}^N_{s_j}-{\mcb P}^N_{s_j}|=|{\mcb P}^N_{s_1}-{\mcb P}^N_{s_2}|+\sum_{j=1}^2|{\mcb P}_{s_i}^N|\left|e^{-2\pi \iota \sum_{i=1}^n k_i \frac{\lfloor v_i^N s_j\rfloor-v_i^N s_j}{N}}-1\right|\,,
\end{align*}
where in the last step we used the definition of $\widetilde{\mcb P}^N$ in~\eqref{e:tildi}. 
To control the last term, since for any $x\in\R$, $|x-\lfloor x\rfloor|\leq 1$, 
a simple application of Taylor's formula gives us 
\begin{equation*}
\left|e^{-2\pi \iota \sum_{i=1}^n k_i \frac{\lfloor v_i^N s_j\rfloor-v_i^N s_j}{N}}-1\right|\lesssim \frac{1}{N}\sum_{i=1}^n |k_i|\,.
\end{equation*}
Hence, we have 
\begin{align}
\E[\three]&=\E\bigg[\Big|\int_0^{t-\frac1{C_N}} (\widetilde{\mcb P}_s^N-\widetilde{\mcb P}^N_{s+\frac1{C_N}}) e^{-2\pi \iota  C_N s}\dd s\Big|\bigg]\leq \E\bigg[\int_0^{t-\frac1{C_N}} |\widetilde{\mcb P}^N_{s+\frac1{C_N}}-\widetilde{\mcb P}_s^N| \dd s\bigg]\nonumber\\
&\lesssim \int_0^{t-\frac1{C_N}} \E\Big[|{\mcb P}^N_{s+\frac1{C_N}}-{\mcb P}_s^N|\Big] \dd s +\frac{\sum_{i=1}^n|k_i|}{N}\E\bigg[\sup_{s\leq T} |\mcb P_s^{N}|\bigg]\lesssim\frac{1}{C_N^\alpha}+\frac{\sum_{i=1}^n|k_i|}{N}\label{e:iii}
\end{align}
where in the last step we used~\eqref{e:Hol} and~\eqref{e:sup}. Since the $k_i$'s are fixed and 
by assumption $C_N\to\infty$ as $N\to\infty$, both~\eqref{e:i+ii} and~\eqref{e:iii} converge {to} $0$, 
{ and} the conclusion follows at once. 
\end{proof}

\end{appendix}

\subsection*{Acknowledgement}

G.C. gratefully acknowledges financial support via the EPSRC grant EP/S012524/1 and the UKRI FL Fellowship 
MR/W008246/1.
P.G. thanks  Funda\c c\~ao para a Ci\^encia e Tecnologia FCT/Portugal for financial support through the
projects UIDB/04459/2020 and UIDP/04459/2020.   This project has received funding from the European Research Council (ERC) under  the European Union's Horizon 2020 research and innovative programme (grant agreement   n. 715734). The work of R.M. was supported by the CNPq grant Universal no. 403037/2021-2. The work of A. O. was supported in part by the aforementioned ERC project, by the ERC-2019-ADG Project 884584 LDRam, and was
partially developed while A.O. was a postdoctoral fellow at MSRI during the Program ``Universality and Integrability in Random Matrix Theory and Interacting Particle Systems''. The authors would like to thank L. Bertini, M. Jara,  and G. Sch\"utz for their useful suggestions about the manuscript.

\end{document}